\documentclass[11pt]{article}
\usepackage[margin=3cm]{geometry}
\usepackage{amsmath,amsthm,amsfonts,amssymb,amscd}
\usepackage{tikz}

\usepackage{dsfont}
\usepackage{framed}
\usepackage[utf8]{inputenc}
\usepackage{mathtools}
\usepackage{multicol}
\usepackage{multirow}
\usepackage{enumerate}
\usepackage{bbm}

 \usepackage[colorlinks = true]{hyperref}
 \hypersetup{
 	linkcolor = blue,
 	urlcolor = blue,
 	citecolor = blue
 }

\newcommand{\reff}[1]{(\ref{#1})}

\newcommand{\tb}{\boldsymbol{t}}
\newcommand{\betab}{\boldsymbol{\beta}}

\newcommand{\sG}{\mathcal{G}}
\newcommand{\sA}{\mathcal{A}}
\newcommand{\sU}{\mathcal{U}}
\newcommand{\sV}{\mathcal{V}}

\newcommand{\sN}{\mathcal{N}}

\newcommand{\IR}{\mathbb{R}}
\newcommand{\IZ}{\mathbb{Z}}
\newcommand{\IM}{\mathbb{M}}
\newcommand{\IW}{\mathcal{W}}
\newcommand{\IX}{\mathbb{X}}

\newcommand{\dt}{t}

\newcommand{\IE}{\mathbb{E}}
\newcommand{\IC}{\mathbb{C}}
\newcommand{\N}{\mathbb{N}}
\newcommand{\IN}{\mathbb{N}}
\newcommand{\E}{\mathbb{E}}

\newcommand{\IP}{\mathbb{P}}

\newcommand{\sF}{\mathcal{F}}

\newcommand{\G}{\mathbb{G}}
\newcommand{\Ii}{\mathds{1}}

\newcommand{\Var}{\mathrm{Var}}

\newcommand{\eps}{\varepsilon}

\newcommand{\norm}[1]{\left\lVert#1\right\rVert}

\DeclareMathOperator*{\argmin}{\arg\min}

\makeatletter
\newcommand{\undersim}[1]{\mathrel{\mathpalette\@undersim{#1}}}
\newcommand{\@undersim}[2]{%
  \vcenter{%
    \ialign{%
      ##\cr
      $\m@th#1#2$\cr
      \noalign{\nointerlineskip\kern.2ex}
      $\m@th#1\sim$\cr
      \noalign{\kern-.4ex}
    }%
  }%
}
\newcommand{\gsim}{\undersim{>}}
\newcommand{\lsim}{\undersim{<}}
\makeatother

\newtheorem{thm}{Theorem}[section]
\newtheorem{lem}[thm]{Lemma}
\newtheorem{cor}[thm]{Corollary}
\newtheorem{prop}[thm]{Proposition}
\newtheorem{assum}[thm]{Assumption}

\theoremstyle{definition}

\newtheorem{rem}[thm]{Remark}
\newtheorem{ex}[thm]{Example}

\numberwithin{equation}{section}

\RequirePackage[numbers]{natbib}

\begin{document}

\title{Notes on Time Series}

\thispagestyle{empty}

\begin{center}
{\LARGE \bf Forecasting time series with encoder-decoder neural networks}\\
{\large Nathawut Phandoidaen, Stefan Richter}\\

{phandoidaen@math.uni-heidelberg.de, stefan.richter@iwr.uni-heidelberg.de}\\

{\small Institut für angewandte Mathematik, Im Neuenheimer Feld 205, Universität Heidelberg}\\
\today
\end{center}

\begin{abstract}
    {In this paper, we consider high-dimensional stationary processes where a new observation is generated from a compressed version of past observations. The specific evolution is modeled by an encoder-decoder structure. We estimate the evolution with an encoder-decoder neural network and give upper bounds for the expected forecast error under specific structural and sparsity assumptions introduced by \cite{schmidthieber2017}. The results are shown separately for conditions either on the absolutely regular mixing coefficients or the functional dependence measure of the observed process. In a quantitative simulation we discuss the behavior of the network estimator under different model assumptions. We corroborate our theory by a real data example where we consider forecasting temperature data.}
\end{abstract}

\section{Introduction}
\label{sec_intro}

During the last years, machine learning has become a very active field of research. One of the main advantages of the algorithms and models considered in this area is their capability of dealing with high-dimensional input and output data. Especially in supervised learning, neural networks have drawn a lot of attention and have been shown to build a flexible model class. Moreover, existing methods such as the prominent stochastic gradient descent enable neural networks to train in a way which avoid excessive overfitting. While the network estimates do not allow for an easily accessible interpretation of the connection between input- and output data, their prediction abilities are remarkable and satisfactory in practice. Investigating regression problems, \cite{schmidthieber2017}, \cite{bauer2019} or \cite{imaizumi2018deep} provide statistical results which support this behavior theoretically.

In this work, we consider the forecasting of high-dimensional time series with neural networks. The general idea to use networks for forecasting was already described in \cite{tang1993}, \cite{kline2004}, \cite{Zhang2012}. However up to now, no theoretical results about the achieved prediction error seem to exist. Such results are of utmost value since the conditions needed can shed light on the choice of a network structure and is still an open problem in practice. Furthermore, quantification of the impact of the underlying dependence in the data can yield information on the number of training samples (or observation length, in a time series context) which is needed to bound the prediction error.

To obtain such results, we assume that the observed time series $X_i$, $i=1,...,n$, is a realization of a stationary stochastic process which obeys
\begin{equation}
	X_{i} = f_0(X_{i-1},...,X_{i-r}) + \varepsilon_i, \quad i=r+1,...,n\label{model_time_evolution}
\end{equation}
where $\varepsilon_i$ is an i.i.d. sequence of $d$-dimensional random variables, $r \in\IN$ is the number of lags considered and $f_0:\IR^{dr} \to \IR^d$ is an unknown function. The forecasting ability of an estimator $\hat f$ of $f_0$ is measured via $\IE D(\hat f)$ where
\begin{equation}
    D(f) := \frac{1}{d}\IE\big[ \big|X_{r+1} - f(X_r,...,X_1)\big|_2^2 \IW(X_r,...,X_1)\big]\label{definition_predictionerror}
\end{equation}
and $|\cdot|_2$ denotes the Euclidean norm and $\IW: \IR^{dr} \to \IR$ is a weight function.

The function $f_0$ is treated nonparametrically, that is, no specific evolution over time is imposed. It is clear that a standard nonparametric estimator $\hat f^{np}$ of $f_0$ will suffer from the curse of dimension. If, for instance, $f_0$ is $(\beta-1)$-times differentiable with Lipschitz continuous derivative $f_0^{(\beta-1)}$ and the $X_i$ are i.i.d., we would expect that there is some constant $C > 0$ (depending on characteristics of $f_0$) such that roughly,
\begin{equation}
    \IE D(\hat f^{np}) \le C\cdot n^{-\frac{2\beta}{2\beta+dr}},\label{naive_nonparametric_rate}
\end{equation}
which is a very slow rate if the dimension $d$ of the time series is large.
To overcome this issue, several structural assumptions for $f_0$ have been proposed, for example additive models (cf. \cite{stone1} for i.i.d. models, \cite{vogt2012} for locally stationary time series). In this paper we will impose a specific encoder-decoder structure on $f_0$ which we see as a reasonable approximation of the true evolution and simultaneously helps to drastically improve the convergence rate. Graphically, we assume that $f_0$ 
``compresses" the given information of the last $r$ lags into a vector of much smaller size and afterwards ``expands" this concentrated information to produce the observation of the next time step ahead. The details will be discussed in Section \ref{sec_model}.

Exploiting the given structure, we define a neural network estimator $\hat f^{net}$ and provide an upper bound for $\IE D(\hat f^{net})$. We quantify the underlying dependence of $X_i$, $i=1,...,n$, by either the functional dependence measure (cf. \cite{wu2005anotherlook}) or absolutely regular mixing coefficients (cf. \cite{rio1995}) which allow for wide-range applicability of the results. It should be noted that the recursion is only used in the fashion of a regression model and we do not impose any contraction condition on $f_0$. Thus, it is not necessary that $X_i$ itself has geometric decaying dependence coefficients. Moreover, for the same reason, our theory allows us to discuss the more general $d$-variate regression model
\[
    Y_i = f_0(X_{i-1},...,X_{i-r}) + \varepsilon_i, \quad i=r+1,...,n,
\]
where we do not impose a direct connection between input $X_i$ and output $Y_i$. In this context, let us emphasize that our stochastic results can be seen as a generalization of \cite{schmidthieber2017} who dealt with i.i.d. data $X_i$ and one-dimensional outputs $Y_i$, in particular. The encoder-decoder structure we impose is very important to transfer the strong convergence rates from \cite{schmidthieber2017} to the setting of high-dimensional outputs, especially in the case of recursively defined time series.

From a theoretical point of view, we make the following contributions: First, we derive oracle-type inequalities for the prediction error \reff{definition_predictionerror} under the two dependence paradigms mentioned above. These results are completely new and seem to be the first oracle-type inequalities for the prediction error under dependence. These are presented in an extra section and may be of independent interest to prove convergence rates of other nonparametric or semiparametric estimation procedures. Besides the additional difficulties posed by dependence, it turns out that some fiddly calculations are needed to unify several terms contributing to the upper bounds. Second, we introduce the encoder-decoder structure as a reasonable evolution scheme for time series and derive upper bounds of the approximation error. 

The paper is organized as follows. In Section \ref{sec_model} we describe the structural assumptions on $f_0$ and formulate the neural network estimator. In Section \ref{sec_theory} we introduce the two measures of dependence and provide upper bounds for the neural network estimator. Section \ref{sec_oracle} contains oracle-type inequalities for minimum empirical risk estimators with respect to $D(f)$ which may be of independent interest. In Section \ref{sec_simulation} we give a small simulation study about the behavior of the neural network estimator fro a practical point of view and apply it to real-world temperature data. In Section \ref{sec_conclusion}, a conclusion is drawn. Most of the proofs are deferred to the Appendix (Section \ref{sec_appendix}) or to the supplementary material without further reference.

Finally, let us shortly introduce some notation used in this paper. For $q > 0$, let $|v|_{q} := (\sum_{j=1}^{r}|v_j|^q)^{1/q}$ denote the $q$-norm of a vector $v\in \IR^{r}$ with the convention $|v|_{\infty} := \max_{j=1,...,r}|v_j|$ and $|v|_0 := \#\{j\in \{1,...,r\}: v_j\not=0\}$ where $\#$ denotes the number of elements of a set. For matrices $W \in \IR^{r \times s}$, let $|W|_{\infty} := \max_{j=1,...,r,k=1,...,s}|W_{jk}|$ and $|W|_0 := \#\{j\in \{1,...,r\}, k\in \{1,...,s\}: W_{jk}\not=0\}$. For mappings $f:\IR^{t} \to \IR$, we denote by $\|f\|_{\infty} := \sup_{x\in \IR
^t}|f(x)|$ the supremum norm. If $f:\IR^{t} \to \IR^d$, we use $\|f\|_{\infty} := \| |f|_{\infty} \|_{\infty}$.
For sequences $x_n,y_n$ we write $x_n \lsim y_n$ if there exist a constant $C > 0$ independent of $n$ such that $x_n \le C y_n$ for $n\in\IN$. We write $x_n \asymp y_n$ if $x_n \lsim y_n$ and $y_n \lsim x_n$.

\section{Encoder-decoder model and neural network estimator}
\label{sec_model}

To simplify the notation, we will abbreviate $\IX_{i-1} := (X_{i-1},...,X_{i-r})$. Thus, equation \reff{model_time_evolution} becomes 
\[
    X_i = f_0(\IX_{i-1}) + \varepsilon_i, \quad i=r+1,...,n,
\]
and
\[
    D(f) := \frac{1}{d}\IE\big[ \big|X_{r+1} - f(\IX_{r})\big|_2^2 \IW(\IX_r)\big]
\]
denotes the expected one-step forecasting error (averaged over the dimensions). For theoretical reasons we will impose that the weight function $\IW:\IR^{dr} \to \IR$ has compact support $\subset [0,1]^{dr}$. This means that we restrict ourselves to the prediction error of predictions from observations $\IX_r \in [0,1]^{dr}$. However, the concepts can easily be extended to general compact sets $A \subset \IR^{dr}$ instead of $[0,1]^{dr}$. For the ease of presentation we choose to discuss our theory on the unit interval, instead.

In the following, we restrict ourselves to the case of Subgaussian noise.

\begin{assum}\label{ass_subgaussian} $\varepsilon_1$ is Subgaussian, that is, for any $k\in\IN$ and any component $j\in\{1,...,d\}$,
\[
    \IE[|\varepsilon_{1j}|^k]^{1/k} \le C_{\varepsilon}\cdot \sqrt{k}.
\]
\end{assum}

\subsection{Encoder-decoder structure and smoothness assumptions}

We require that $f_0$ in \reff{model_time_evolution} has a specific ``sparse" form, which we model through several structural assumptions. 

\begin{assum}[Encoder-decoder assumption]\label{ass_autoencoder}
    We assume that
\begin{equation}
	f_0 = f_{dec} \circ f_{enc},\label{decomposition_f}
\end{equation}
for $f_{enc}:\IR^{dr} \to \IR^{\tilde d}$ with $\tilde d \in \{1,...,d\}$, and $f_{dec}: \IR^{\tilde d}\to \IR^d$ only depending on a maximum of $t_{dec} \in \{1,...,\tilde d\}$ arguments in each component. Furthermore,
    \begin{equation}
        f_{enc} = g_{enc,1} \circ g_{enc,0},\label{decomposition_fenc}
    \end{equation}
where $g_{enc,0}:\IR^{dr} \to \IR^D$, $D\in\IN$, only depends on a maximum of $t_{enc,0} \in \{1,...,dr\}$ arguments in each component and $g_{enc,1}:\IR^{D} \to \IR^{\tilde d}$ only depends on a maximum of $t_{enc,1} \in \{1,....,D\}$ arguments in each component.
\end{assum}

The structure of $f_0$ (which has not to be a neural net itself) is depicted in Figure \ref{figure_representation_f0}. Condition \reff{decomposition_f} asks $f_0$ to decompose into a function $f_{enc}$ which reduces the dimension from $dr$ to $\tilde d \in \{1,...,d\}$ (the ``encoder"), and $f_{dec}:\IR^{\tilde d} \to \IR^{d}$ which expands the dimension to $d$ (the ``decoder"). For $r = 1$, such structures typically arise when information has to be compressed into a vector $\IR^{\tilde d}$ (with the encoder) but also should be restorable (with the decoder). 

The domain of definition of $f_{enc}$ is $d$-dimensional, with a possibly large $d$. Therefore, a structural constraint in the form of \reff{decomposition_fenc} is one possibility to control the convergence rate of the corresponding network estimator. A typical example we have in mind are additive models of the following form, where $g_{enc,1}$ is basically chosen as a summation function.

\begin{ex}[Additive models]\label{ex_additive_model} \noindent
    \begin{itemize}
        \item[(1)] Reduction to one dimension: Suppose that $f_0 = f_{dec}\circ f_{enc}$ where $f_{dec}:\IR \to \IR^d$ and 
    \[
        f_{enc}(x) = \sum_{j=1}^{d}g_j(x_j)
    \]
    for functions $g_j:\IR \to \IR$. Then, Assumption \ref{ass_autoencoder} is fulfilled with $t_{enc,0} = \tilde d = t_{dec} = 1$, $t_{enc,1} = D = d$.
        \item[(2)] Reduction to $\tilde d$ dimensions: Suppose that $f_0 = f_{dec}\circ f_{enc}$ where $f_{dec}:\IR^{\tilde d} \to \IR^d$ and $f_{enc} = (f_{enc,k})_{k=1,...,\tilde d}$ with
        \[
            f_{enc,k}(x) = \sum_{i_1,...,i_{t_{enc,0}}=1}^{d}g_{i_1,...,i_{t_{enc}}}^{(k)}(x_{i_1},...,x_{i_{t_{enc}}})
        \]
        for functions $g_{i_1,...,i_{t_{enc,0}}}:\IR^{t_{enc,0}} \to \IR$. 
        Then, Assumption \ref{ass_autoencoder} is fulfilled with the given $t_{enc,0}, \tilde d = t_{dec}$ and $t_{enc,1} = D = d^{t_{enc,0}}$.
    \end{itemize}
\end{ex}

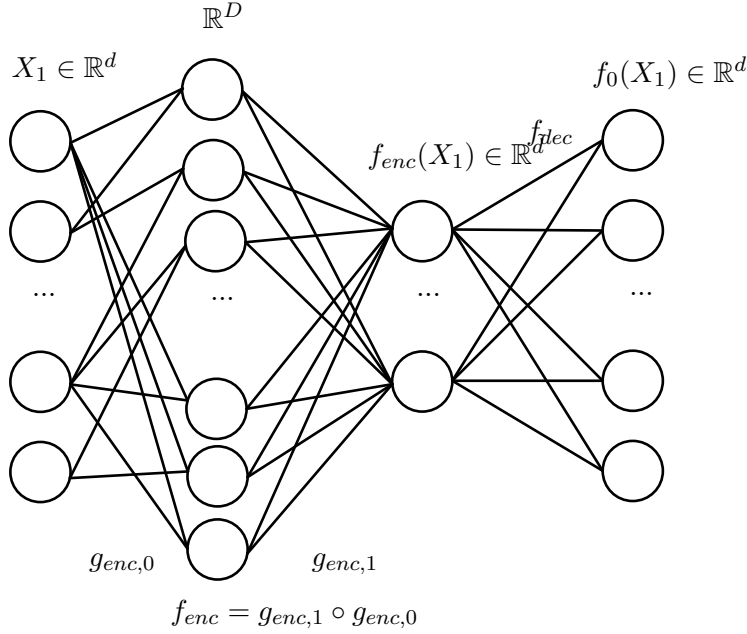
\begin{figure}
\centering
\begin{tikzpicture}[line cap=round,line join=round,x=1cm,y=1cm, scale = 0.8]
rectangle (9.960437252493387,10.168498059250371);
\draw [line width=1pt] (-2.341658071587629,3.982239637052577) circle (0.5cm);
\draw [line width=1pt] (-2.341658071587629,2.482239637052577) circle (0.5cm);
\draw [line width=1pt] (-2.341658071587629,-0.01776036294742267) circle (0.5cm);
\draw [line width=1pt] (-2.341658071587629,-1.5177603629474228) circle (0.5cm);
\draw (-2.6496273509171425,1.6221855282418773) node[anchor=north west] {...};
\draw [line width=1pt] (4.002648359216246,2.487317716577269) circle (0.5cm);
\draw [line width=1pt] (4.002648359216246,-0.012682283422730478) circle (0.5cm);
\draw (3.740549880098117,1.6221855282418773) node[anchor=north west] {...};
\draw [line width=1pt] (7.501439859858489,3.9996584045803636) circle (0.5cm);
\draw [line width=1pt] (7.501439859858489,2.4996584045803636) circle (0.5cm);
\draw [line width=1pt] (7.501439859858489,-0.00034159541963685314) circle (0.5cm);
\draw [line width=1pt] (7.501439859858489,-1.5003415954196369) circle (0.5cm);
\draw (7.276073283535731,1.647773290310765) node[anchor=north west] {...};
\draw [line width=1pt] (-1.8422888246853146,3.9571327605400497)-- (0.02432705899354673,4.861036007530361);
\draw [line width=1pt] (-1.8422277899395572,2.4583776636864094)-- (0.02432705899354673,4.861036007530361);
\draw [line width=1pt] (4.501838495957166,2.5157642872068093)-- (7.001439975876717,3.9999990190465433);
\draw [line width=1pt] (4.501838495957166,2.5157642872068093)-- (7.001440429519582,2.4989036458313163);
\draw [line width=1pt] (4.501838495957166,2.5157642872068093)-- (7.001439976545934,0);
\draw [line width=1pt] (4.501838495957166,2.5157642872068093)-- (7.001450003322883,-1.4971567309792098);
\draw [line width=1pt] (4.502487493025501,0)-- (7.001450003322883,-1.4971567309792098);
\draw [line width=1pt] (4.502487493025501,0)-- (7.001439976545934,0);
\draw [line width=1pt] (4.502487493025501,0)-- (7.001440429519582,2.4989036458313163);
\draw [line width=1pt] (4.502487493025501,0)-- (7.001439975876717,3.9999990190465433);
\draw (-2.986970354363784,5.639464173057247) node[anchor=north west] {$X_1 \in \mathbb{R}^d$};
\draw (-0.35143086126835194,-3.469779123466776) node[anchor=north west] {$f_{enc} = g_{enc,1} \circ g_{enc,0}$};
\draw (6.659615945606875,5.562700886850584) node[anchor=north west] {$f_0(X_1) \in \mathbb{R}^d$};
\draw (5.559342176644704,4.513602642026188) node[anchor=north west] {$f_{dec}$};
\draw (2.9238026835492725,4.283312783406198) node[anchor=north west] {$f_{enc}(X_1) \in \mathbb{R}^{\tilde d}$};
\draw [line width=1pt] (0.5127370400476697,4.843275644582937) circle (0.5cm);
\draw [line width=1pt] (0.5304974029950924,3.503118911109741) circle (0.5cm);
\draw [line width=1pt] (0.6015388547847831,-1.5898940792139746) circle (0.5cm);
\draw [line width=1pt] (0.601538854784783,-2.8057282720552106) circle (0.5cm);
\draw (0.31620200079827874,1.571010004104102) node[anchor=north west] {...};
\draw [line width=1pt] (0.5660181288899384,2.321304106048913) circle (0.5cm);
\draw [line width=1pt] (0.5837784918373614,-0.46707287669645337) circle (0.5cm);
\draw (0.18591214217828955,6.458272559261654) node[anchor=north west] {$\mathbb{R}^D$};
\draw [line width=1pt] (1.0121062869499842,4.818168768070409)-- (3.5047269335196267,2.5328616934814243);
\draw [line width=1pt] (1.0299276846431642,3.4792569377435734)-- (3.5047269335196267,2.5328616934814243);
\draw [line width=1pt] (1.0657106257858222,2.3037710359823906)-- (3.5047269335196267,2.5328616934814243);
\draw [line width=1pt] (1.0837784918373614,-0.46707287669645337)-- (3.5047269335196267,2.5328616934814243);
\draw [line width=1pt] (3.5046222834069853,-0.057067278793110646)-- (1.0837784918373614,-0.46707287669645337);
\draw [line width=1pt] (1.0968810874316106,-2.8738169829877656)-- (3.5046222834069853,-0.057067278793110646);
\draw [line width=1pt] (1.0968810874316106,-2.8738169829877656)-- (3.5047269335196267,2.5328616934814243);
\draw [line width=1pt] (1.0657106257858222,2.3037710359823906)-- (3.5046222834069853,-0.057067278793110646);
\draw [line width=1pt] (1.0299276846431642,3.4792569377435734)-- (3.5046222834069853,-0.057067278793110646);
\draw [line width=1pt] (1.0121062869499842,4.818168768070409)-- (3.5046222834069853,-0.057067278793110646);
\draw [line width=1pt] (1.1013858816330477,-1.6022613647265604)-- (3.5046222834069853,-0.057067278793110646);
\draw [line width=1pt] (1.1013858816330477,-1.6022613647265604)-- (3.5047269335196267,2.5328616934814243);
\draw [line width=1pt] (-1.8422277899395572,2.4583776636864094)-- (0.031198401030158363,3.52958599702453);
\draw [line width=1pt] (-1.8434317287177122,-0.05983780667286972)-- (0.031198401030158363,3.52958599702453);
\draw [line width=1pt] (-1.8520376559161444,-1.6191104915002903)-- (0.0716185855466438,2.2466777598838874);
\draw [line width=1pt] (-1.8434317287177122,-0.05983780667286972)-- (0.0716185855466438,2.2466777598838874);
\draw [line width=1pt] (-1.8422888246853146,3.9571327605400497)-- (0.10368084353557377,-0.3274081062813876);
\draw [line width=1pt] (-1.8434317287177122,-0.05983780667286972)-- (0.10368084353557377,-0.3274081062813876);
\draw [line width=1pt] (-1.8422888246853146,3.9571327605400497)-- (0.10825004664196874,-1.5082475743222081);
\draw [line width=1pt] (-1.8520376559161444,-1.6191104915002903)-- (0.10825004664196874,-1.5082475743222081);
\draw [line width=1pt] (-1.8434317287177122,-0.05983780667286972)-- (0.10171568016200122,-2.7924318774309103);
\draw [line width=1pt] (-1.8422888246853146,3.9571327605400497)-- (0.10171568016200122,-2.7924318774309103);
\draw (-1.7,-2.7) node[anchor=north west] {$g_{enc,0}$};
\draw (2.0,-2.7) node[anchor=north west] {$g_{enc,1}$};
\end{tikzpicture}
\caption{Graphical representation of the encoder-decoder assumption on $f_0$ in the special case $r = 1$.}
\label{figure_representation_f0}
\end{figure}

\subsection{Neural networks and the estimator} We now present the network estimator, formally. To do so, we use the formulation from \cite{schmidthieber2017}. Let $\sigma(x) := \max\{x,0\}$ be the ReLU activation function. For a vector $v = (v_1,...,v_r)\in \IR^{r}$, put
\[
    \sigma_v:\IR^r \to \IR^r, \quad \sigma_v(x) := (\sigma(x_1 - v_1),...,\sigma(x_r - v_r))'.
\]
Let $(L,p)$ denote the network architecture where $L \in \IN_0$ denotes the number of hidden layers and $p = (p_0,...,p_{L+1}) \in \IN^{L+2}$ denotes the number of hidden layers. A neural network with network architecture $(L,p)$ is a function of the form
\begin{equation}
    f:\IR^{p_0} \to \IR^{p_{L+1}}, \quad f(x) = W^{(L)} \sigma_{v^{(L)}} W^{(L-1)} \sigma_{v^{(L-1)}} \dots W^{(1)} \sigma_{v^{(1)}} W^{(0)} x\label{form_neuralnetwork}
\end{equation}
where $W^{(i)} \in \IR^{p_i \times p_{i+1}}$ are weight matrices and $v^{(i)} \in \IR^{p_i}$ are bias vectors. For $L_1 \in \{1,...,L\}$, let
\begin{eqnarray*}
    \sF_{ed}(L,L_1,p) &:=& \big\{f:\IR^{p_0} \to \IR^{p_{L+1}}\text{ is of the form \reff{form_neuralnetwork}}:\\
    &&\quad\quad\quad\quad\quad\quad\quad\quad \max_{k=0,...,L}|W^{(j)}|_{\infty}\vee |v^{(j)}|_{\infty} \le 1, p_{L_1} = \tilde d\big\},
\end{eqnarray*}
be a network, where the $L_1$-th hidden layer is $\tilde d$-dimensional. Since we aim to approximate $f_0$, it has to hold that $p_0 = dr$ and $p_{L+1} = d$.

As an empirical counterpart of the prediction error \reff{definition_predictionerror}, define
\begin{equation}
    \hat R_n(f) := \frac{1}{n}\sum_{i=r+1}^{n}\frac{1}{d}|X_i - f(\IX_{i-1})|_2^2 \IW(\IX_{i-1}).\label{empirical_prediction_error}
\end{equation}
It turns out to be the case that in practice, a neural network $\hat f \in \sF_{ed}(L,L_1,p)$ obtained by minimizing $\hat R_n(f)$ with a stochastic gradient descent method contains weight matrices and bias vectors in which a lot of entries are not relevant for the evaluation $\hat f(x)$ of $x \in [0,1]^d$. This behavior can be explained by the random initialization of the weight matrices and large step sizes of the gradient method. In fact, by employing dropout techniques during the learning process or imposing some additional penalties we can force $W^{(j)},v^{(j)}$, $j=0,...,L$, to be sparse. To indicate this type of sparsity in the model class, we introduce for $s \in \IN$ and $F > 0$,
\[
    \sF(L,L_1,p,s,F) := \big\{f \in \sF_{ed}(L,L_1,p): \sum_{j=0}^{L}|W^{(j)}|_0 + |v^{(j)}|_0 \le s, \| f \|_{\infty} \le F\big\}
\]
and define the final neural network estimator via
\begin{equation}
    \hat f^{net} \in \argmin_{f\in \sF(L,L_1,p,s,F)}\hat R_n(f).\label{definition_network_estimator}
\end{equation}
In particular, the resulting network $\hat f^{net}$ (with estimated weight matrices $\hat W^{(j)}$ and bias vectors $\hat v^{(j)}$) can provide an estimator of the encoder function $f_{enc}$ by only using its representation up to the $L_1$-th layer, that is,
\[
    \hat f_{enc}^{net}(x) := \hat W^{(L_1)}\sigma_{\hat v^{(L_1)}}\hat W^{(L_1 - 1)}\sigma_{\hat v^{(L_1-1)}} \dots \hat W^{(1)}\sigma_{\hat v^{(1)}} W^{(0)}x.
\]
Another typical observation made is that fitted neural networks $\hat f^{net}$ tend to be rather smooth functions. This can be enforced by adding a gradient penalty in the learning procedure (common, for instance, in the training of WGANs, where a restricted Lipschitz constant is part of the optimization functional, cf. \cite{wgan_gradientpenalty}). We will see in Section \ref{sec_theory} that we also formally need a bound on the Lipschitz constant when quantifying dependence with the functional dependence measure. We therefore introduce a second neural network estimator based on the function class
\[
    \sF(L,L_1,p,s,F,\mathrm{Lip}) := \big\{f \in \sF(L,L_1,p,s,F): \|f\|_{Lip} \le \mathrm{Lip}\big\}
\]
where $\|f\|_{Lip} := \sup_{x\in \IR^d}\frac{|f(x)-f(x')|_{\infty}}{|x-x'|_{\infty}}$. This estimator reads
\begin{equation}
    \hat f^{net,lip} \in \argmin_{f\in \sF(L,L_1,p,s,F,\mathrm{Lip})}\hat R_n(f).\label{definition_network_estimator_lip}
\end{equation}

\subsection{Smoothness assumptions}

To state convergence rates of $\hat f$, we have to quantify smoothness assumptions of the underlying true function $f_0$ and its components $g_{enc,1}, g_{enc,0}$ and $f_{dec}$. We measure smoothness with the well-known H\"older balls. A function has H\"older smoothness index $\beta$ if all partial derivatives up to order $\lfloor \beta\rfloor := \max\{k \in\IN_0: k < \beta\}$ exist, are bounded and the partial derivative of order $\lfloor \beta \rfloor$ are $\beta-\lfloor \beta\rfloor$. The ball of $\beta$-H\"older functions with radius $K > 0$ and domain of definition $P \subset \IR
^r$ reads
\begin{eqnarray*}
    C^{\beta}(P,K) &:=& \{f:P \to \IR:\\
    &&\quad\quad \sum_{\alpha:|\alpha|\le \beta}\|\partial^{\alpha}f\|_{\infty} + \sum_{\alpha:|\alpha| = \lfloor \beta\rfloor}\underset{x\not=y}{\sup_{x,y\in P}}\frac{|\partial^{\alpha}f(x) - \partial^{\alpha}f(y)|}{|x-y|_{\infty}^{\beta-\lfloor \beta\rfloor}} \le K\}
\end{eqnarray*}
where $\alpha = (\alpha_1,...,\alpha_r) \in \IN_0^r$ is a multi-index and $\partial^{\alpha} := \partial^{\alpha_1}...\partial^{\alpha_r}$, $|\alpha| := \alpha_1 + ... + \alpha_r$.

We now pose the following assumption.

\begin{assum}[Smoothness assumption]\label{ass_smoothness}
    Suppose that for some constant $K \ge 1$ and $\beta_{dec}, \beta_{enc,1}, \beta_{enc,0} \ge 1$, 
    \begin{itemize}
    \item $g_{enc,0} \in C^{\beta_{enc,0}}([0,1]^{dr},K)$ and $g_{enc,0}([0,1]^{dr}) \subset [0,1]^{D}$,
    \item $g_{enc,1} \in C^{\beta_{enc,1}}([0,1]^{D},K)$ and $g_{enc,1}([0,1]^{D}) \subset [0,1]^{\tilde d}$,
    \item $f_{dec} \in C^{\beta_{dec}}([a_{enc,1},b_{enc,1}]^{\tilde d},K)$.
    \end{itemize}
\end{assum}

The restriction to the unit intervals for the domain of definition and image is only done for the sake of simplicity in our presentation and can be easily enlarged to compact sets by rescaling.

\section{Theoretical results under dependence}
\label{sec_theory}

To state the theoretical results about $\IE D(\hat f)$, we have to quantify the dependence structure of $X_i$, $i = 1,...,n$. We now shortly introduce the two dependence concepts we consider in this paper.

\subsection{Absolutely regular mixing coefficients}

Let $\beta^{mix}(k)$, $k\in\IN_0$, denote the absolutely regular mixing coefficients of $X_i$, that is,
\begin{equation}
    \beta^{mix}(k) := \beta^{mix}(\sigma(X_i: i \le 0), \sigma(X_i:i \ge k)),\label{definition_beta_mixing}
\end{equation}
where for two sigma fields $\sU, \sV$ over some probability space $\Omega$,
\[
    2\beta^{mix}(\sU,\sV) := \sup \sum_{(i,j) \in I \times J}|\IP(U_i \cap V_i) - \IP(U_i) \IP(V_i)|,
\]
and the supremum is taken over all finite partitions $(U_i)_{i\in I}$, $(V_j)_{j\in J}$ of $\Omega$ such that $(U_i)_{i\in I} \subset \sU$, $(V_j)_{j\in J} \subset \sV$. Graphically, $\beta^{mix}(k)$, $k\in\IN_0$, measures the dependence between $\sigma(X_i: i \le 0)$ and $\sigma(X_i:i \ge k)$ and decays to $0$ for $k \to \infty$ if $\sigma(X_i: i \le 0)$ contains no information about $X_k$ for large $k$. We refer to \cite[Section 1.3]{rio2013} or \cite{rio1995invariance} for a more detailed introduction. There are several results available which state that linear processes, GARCH or ARMA processes have absolutely summable $\beta^{mix}(k)$, cf. \cite{bradley2005}, \cite{fryzlewicz2011} or \cite{doukhan_mixingbook}. The impact of the dependence in Theorem \ref{thm_final_convergence_rate_mixing} is measured via $\Lambda
^{mix}(\cdot)$, which is obtained from $\beta
^{mix}(\cdot)$ via the following construction.

\begin{assum}[Compatability assumptions]\label{ass_compatibility}
    Let $X_i$ have $\beta$-mixing coefficients $\beta^{mix}(k)$, $k\in\IN_0$, which are submultiplicative, that is, there exists a constant $C_{\beta,sub} > 0$ such that for any $q_1,q_2 \in \IN$,
    \begin{equation}
        \beta^{mix}(q_1q_2) \le C_{\beta,sub} \beta^{mix}(q_1) \beta^{mix}(q_2).\label{definition_beta_submultiplicative}
    \end{equation}
    Let $\phi:[0,\infty) \to [0,\infty)$ be a function which satisfies
     \begin{itemize}
        \item[(i)] $\phi(0) = 0$, $\phi$ is convex and differentiable with $c_0 := \sup_{y\in\IR}\frac{\phi'(y)y}{\phi(y)} < \infty$,
        \item[(ii)] $(0,\infty)\to (0,\infty), y \mapsto \frac{y}{\phi(y)}$ is convex and decreasing,
        \item[(iii)] $\sum_{k=0}^{\infty}(\phi^{*}(k+1) - \phi^{*}(k))\beta^{mix}(k) < \infty$.
    \end{itemize}
\end{assum}

Based on $\phi$, we define
\begin{equation}
    \psi(x) := \phi^{*}(x)x, \qquad \qquad  \Lambda^{mix}(x) := \lceil \psi^{-1}(x^{-1})\rceil x.\label{definition_lambda_mixing}
\end{equation}

In the special case of polynomial decay and exponential decay of $\beta^{min}(\cdot)$, explicit representations of $\Lambda^{mix}(\cdot)$ are available via the following lemma.

\begin{lem}\label{lemma_varphi_specialcase}
    \begin{enumerate}
        \item Suppose that $\sum_{k=0}^{\infty}k^{\alpha-1} \beta^{mix}(k) < \infty$ for some $\alpha > 1$. Then Assumption \ref{ass_compatibility} is fulfilled with $\phi(x) = x^{\frac{\alpha}{\alpha-1}}$ and
        \[
            \Lambda^{mix}(x) \le c_{\alpha}\cdot (x^{\frac{\alpha}{\alpha+1}} \vee x)
        \]
        where $c_{\alpha} > 0$ is some constant only depending on $\alpha$.
        \item Suppose that $\beta^{mix}(k) \le \kappa \rho^k$ for some $\kappa > 0$, $\rho \in (0,1)$. Then Assumption \ref{ass_compatibility} is fulfilled with $\phi(x) = x \frac{\log(x+1)}{\log(a)}$ ($a = \frac{\rho+1}{2\rho}$) and
        \[
            \Lambda^{mix}(x) \le c_{\rho}\cdot (1 \vee \log(x^{-1}))x
        \]
        where $c_{\rho} > 0$ is some constant only depending on $\rho$.
    \end{enumerate}
\end{lem}

\subsection{Functional dependence measure}

The functional dependence measure was introduced by \cite{wu2005anotherlook}. We assume that $X_i = (X_{ij})_{j=1,...,d}$, $i = 1,...,n$, has the form
\begin{equation}
    X_i = J(\sA_i)\label{representation_x}
\end{equation}
where $J:(\IR^{d})^{\N_0}\to \IR^d$ is some measurable function and $\sA_i = \sigma(\eps_i,\eps_{i-1},...)$ is the sigma-algebra generated by the i.i.d. sequence $\eps_i$, $i \in\IZ$.
For a copy $\eps_k^{*}$ of $\eps_k$, independent of $\eps_i, i\in\IZ$, we define $\sA_i^{*(i-k)} := (\eps_i,...,\eps_{i-k+1},\eps_{i-k}^{*},\eps_{i-k-1},...)$ and $X_i^{*(i-k)} := J(\sA_{i}^{*(i-k)})$. The functional dependence measure of $X_i$, $i\in\IZ$, for $q > 0$ is given by
\begin{equation}
    \delta^{X}_q(k) = \sup_{j=1,...,d}\big\|X_{ij} - X_{ij}^{*(i-k)}\big\|_q.\label{definition_uniform_functional_dependence_measure}
\end{equation}

\begin{rem}
    The representation \reff{representation_x} in terms of the i.i.d. sequence $\varepsilon_i$, which is also present in the recursion \reff{model_time_evolution}, is chosen for simplicity. Instead of $\sA_i = \sigma(\varepsilon_i,\varepsilon_{i-1},...)$ we could also choose $\sA_i = \sigma(\xi_i,\xi_{i-1},...)$ for some larger i.i.d. sequence $\xi_i \in \IR^{d_L}$, $i\in\IZ$, ($d_L > d$) which contains $\varepsilon_i$, $i\in\IZ$.
\end{rem}
In opposite to the case of absolutely regular mixing, the functional dependence measure in \reff{definition_uniform_functional_dependence_measure} requires the process $X_i$ to have at least a $q$-th moment. To transfer the dependence structure from $X_i$ to some function $g(X_i)$, we have to impose smoothness assumptions on $g$ (cf. \cite{empproc}) which also affect the dependence coefficients $\delta^{g(X)}$. We do this formally by the following assumption. 

\begin{assum}\label{ass_compatibility2}
    Let $X_i$ be of the form \reff{representation_x}. Given $L_{\sG} > 0$, let $\Delta(k)$, $k\in\IN_0$, be a decreasing sequence of real numbers such that for some $\theta \in (0,1]$,
    \begin{equation}
    L_{\sG}\cdot \sup_{l=1,...,r}\delta_{2\theta}^{X}(k-l)^{\theta} \le \Delta(k).\label{ass_compatibility2_eq1}
\end{equation}
\end{assum}

The parameter $\theta \in (0,1]$ in Assumption \ref{ass_compatibility2} can be chosen arbitrarily and regulates the number of moments which have to be imposed on $X_i$. A small $\theta$ however coincides with a slower decay rate of $\Delta(k)$ due to the exponent $\theta$ in \reff{ass_compatibility2_eq1}. The constant $L_{\sG}$ is specified below in Theorem \ref{thm_final_convergence_rate_depmeas} and Theorem \ref{theorem_oracle_inequality_dep_present}, respectively.

For $x \in [0,\infty)$, define
\begin{equation}
    \tilde V(x) = x^{1/2} + \sum_{j=0}^{\infty}\min\{x^{1/2}, \Delta(j)\}\label{definition_tildev_depmeasure}
\end{equation}
Let $\bar y(x) \in [0,\infty)$ be such that
\begin{equation}
    \tilde V(\sqrt{x} \bar y(x)) \le \bar y(x)\label{def_y_dependencemeasure}
\end{equation}
and put
\begin{equation}
    \Lambda^{dep}(x) = \sqrt{x}\bar y(x).\label{definition_lambda_depmeas}
\end{equation}

\begin{lem}[Special cases]\label{lemma_depmeas_raten_specialcase}
    \begin{enumerate}
        \item If $\Delta(j) \le \kappa  j^{-\alpha}$ with some $\kappa > 0, \alpha > 1$, then
        \[
            \Lambda^{dep}(x) \le c_{\kappa,\alpha}\max\{x^{\frac{\alpha}{\alpha+1}},x\}
        \]
        where $c_{\kappa,\alpha}$ is a constant only depending on $\kappa,\alpha$.
        \item If $\Delta(j) \le \kappa \rho^j$ with some $\kappa > 0, \rho \in (0,1)$, then
        \[
            \Lambda^{dep}(x) \le c_{\kappa,\rho} x \log(x^{-1} \vee 1)^2
        \]
        where $c_{\kappa,\rho}$ is a constant only depending on $\kappa,\rho$.
    \end{enumerate}
\end{lem}

\subsection{Network conditions}
For the following theorems, we impose the following assumptions on the network class. These assumptions are mainly adapted from \cite[Theorem 1]{schmidthieber2017} and are necessary to control the approximation error of the class $\sF(L,L_1,p,s,F)$ as well as the size $H(\delta, \sF(L,L_1,p,s,F), \|\cdot\|_{\infty})$ of the corresponding covering numbers. The parameter $N$ therein is a parameter in the final theorems. 

\begin{assum}\label{ass_network}
    Fix $N\in\{1,...,n\}$. The parameters $L,L_1,p,s,F$ of $\sF(L,L_1,p,s,F)$ are chosen such that 
    \begin{enumerate}
        \item $K \le F$,
        \item $ \{\log_2(4(t_{enc,0} \vee \beta_{enc,0})) + \log_2(4(t_{enc,1} \vee \beta_{enc,1}))\}\log_2(n) \le L_1$ and\\
        $L_1  +  \log_2(4(t_{dec} \vee \beta_{dec}))\log_2(n) \le L \lsim \log_2(n)$,
        \item $N \lsim \min_{i\in \{1,...,L\}\backslash\{L_1\}}\{p_i\}$,
        \item $N \log_2(n) \asymp s $.
    \end{enumerate}
\end{assum}

We now give a small discussion on the conditions. As we will see below, the optimal $N$ is roughly of the size $n^{a}$, where $a$ depends on smoothness properties of the underlying function $f_0$. Assumption (i) encodes the necessary fact that the network class has to include networks which have a supremum norm larger than the true function $f_0$. The second condtion (ii) is a condition on the layer size. It should be chosen of order $L \asymp \log_2(n)$. In fact, the upper bound on $L$ is not necessary but produces the best convergence rates (cf. the proof of Theorem \ref{thm_final_convergence_rate_mixing} or Theorem \ref{thm_final_convergence_rate_depmeas}, respectively). Condition (iii) poses a lower bound on the size of the hidden layers in the network. From a practical point of view, it seems rather unusual to impose such a large dimension $\ge n^{a}$ to \emph{all} the hidden layers. This is due to the approximation technique used and surely can be improved. The last condition (iv) asks the number of nonzero parameters $s \asymp N\log_2(n)$ which for instance could be enforced by computational methods during the learning process. 

\subsection{Theoretical results}

During this section, let $\IW:\IR^{dr} \to [0,1]$ be an arbitrary (measurable) weight function with $\text{supp}(\IW) \subset [0,1]^{dr}$. The weight function occurs in the optimization functional \reff{empirical_prediction_error} and the corresponding prediction error \reff{definition_predictionerror}.

\begin{thm}[Mixing]\label{thm_final_convergence_rate_mixing} Suppose that Assumptions \ref{ass_subgaussian}, \ref{ass_autoencoder}, \ref{ass_smoothness} and  \ref{ass_compatibility} hold. If Assumption  \ref{ass_network} is satisfied for some $N \in \{1,...,n\}$, then
\[
    \IE D(\hat f^{net}) \lsim \Lambda^{mix}(\frac{N \log(n)^3}{n}) + N^{-2A},
\]
where $A := \min\{\frac{\beta_{dec}}{t_{dec}}, \frac{\beta_{enc,0}}{t_{enc,0}}, \frac{\beta_{enc,1}}{t_{enc,1}}\}$.
\end{thm}
\begin{proof}[Proof of Theorem \ref{thm_final_convergence_rate_mixing}]
    Choose $\eta = 1$ and $\delta = n^{-1}$. By Theorem \ref{theorem_oracle_inequality_present} and Assumptions \ref{ass_compatibility} and \ref{ass_subgaussian},
    \begin{equation}
        \IE D(\hat f^{net}) \lsim \inf_{f\in \sF(L,L_1,p,s,F)}D(f) + \big( \Lambda^{mix}(\frac{H(n^{-1})}{n}) + n^{-1}\big).\label{thm_final_convergence_rate_mixing_eq1}
    \end{equation}
    By Theorem \ref{approx_error} and Assumptions  \ref{ass_autoencoder}, \ref{ass_smoothness} and \ref{ass_network},
    \begin{equation}
        \inf_{f\in \sF(L,L_1,p,s,F)}D(f) \le \inf_{f\in \sF(L,L_1,p,s,F)}\|f - f_0\|_{\infty} \lsim  \frac{N}{n} + N^{-2A}.\label{thm_final_convergence_rate_mixing_eq2}
    \end{equation}
    By Proposition \ref{covering_bound} and Assumption  \ref{ass_network},
    \begin{eqnarray}
        H(\delta) &\le& (s+1)\log(2^{2L+5}\delta^{-1}(L+1)p_0^2p_{L+1}^2 s^{2L}) \lsim s L \log(s)\nonumber\\
        &\lsim& N \log_2(n)\cdot \log_2(n)\log(n) \lsim N \log(n)^3.\label{thm_final_convergence_rate_mixing_eq3}
    \end{eqnarray}
    Insertion of \reff{thm_final_convergence_rate_mixing_eq2} and \reff{thm_final_convergence_rate_mixing_eq3} into \reff{thm_final_convergence_rate_mixing_eq1} yields the result.
\end{proof}

To formulate an analogeous result for the functional dependence measure, we have to assume that the weight function in \reff{empirical_prediction_error} is Lipschitz continuous in the sense that for some $\varsigma > 0$,
\[
    |\IW(x) - \IW(x')| \le \frac{1}{\varsigma}\cdot |x-x'|_{\infty}.
\]
A specific example is given by
\begin{equation}
    \IW(x) := 1 - \rho(\varsigma^{-1} d(x,[\varsigma,1-\varsigma]^{dr})) = \begin{cases} 1, & x \in [\varsigma,1-\varsigma]^{dr} \\
    0, & x \not\in [0,1]^{dr},\\
    \text{linear}, & \text{else}\end{cases}\label{definition_weight_function}
\end{equation}
where $\rho(z) := \max\{\min\{z,1\},0\}$ and $d_{\infty}(x,A) := \inf_{y\in A}|x - y|_{\infty}$.

\begin{thm}[Functional dependence]\label{thm_final_convergence_rate_depmeas}
Suppose that Assumptions \ref{ass_subgaussian}, \ref{ass_autoencoder}, \ref{ass_smoothness} hold. Let Assumption \ref{ass_compatibility2} hold with $L_{\sG} = 2dr\big(\frac{2}{\varsigma} + \frac{(\mathrm{Lip} + K)}{F}\big)$. Then there exists some constant $\IC_L > 0$ independent of $n$ such that if Assumption  \ref{ass_network} is satisfied for some $N \in \{1,...,n\}$ and $\mathrm{Lip} \ge \IC_L$, then
\[
    \IE D(\hat f^{net,lip}) \lsim \Lambda^{dep}(\frac{N \log(n)^3}{n}) + N^{-2A},
\]
where $A := \min\{\frac{\beta_{dec}}{t_{dec}}, \frac{\beta_{enc,0}}{t_{enc,0}}, \frac{\beta_{enc,1}}{t_{enc,1}}\}$.
\end{thm}
\begin{proof}[Proof of Theorem \ref{thm_final_convergence_rate_depmeas}]
    Choose $\eta = 1$ and $\delta = n^{-1}$. By Theorem \ref{theorem_oracle_inequality_dep_present} and Assumptions \ref{ass_compatibility} and \ref{ass_subgaussian},
    \begin{equation}
        \IE D(\hat f^{net}) \lsim \inf_{f\in \sF(L,L_1,p,s,F)}D(f) + \big( \Lambda^{dep}(\frac{H(n^{-1})}{n}) + n^{-1}\big).\label{thm_final_convergence_rate_depmeas_eq1}
    \end{equation}
    By Theorem \ref{approx_error} and Assumptions \ref{ass_autoencoder}, \ref{ass_smoothness} and \ref{ass_network},
    \begin{equation}
        \inf_{f\in \sF(L,L_1,p,s,F, \mathrm{Lip})}D(f) \le \inf_{f\in \sF(L,L_1,p,s,F, \mathrm{Lip})}\|f - f_0\|_{\infty} \lsim \frac{N}{n} + N^{-2A}.\label{thm_final_convergence_rate_depmeas_eq2}
    \end{equation}
    By Proposition \ref{covering_bound} and Assumption  \ref{ass_network},
    \begin{eqnarray}
        H(\delta) &\le& (s+1)\log(2^{2L+5}\delta^{-1}(L+1)p_0^2p_{L+1}^2 s^{2L}) \lsim s L \log(s)\nonumber\\
        &\lsim& N L \log_2(n)\log(n) \lsim N \log(n)^3.\label{thm_final_convergence_rate_depmeas_eq3}
    \end{eqnarray}
    Insertion of \reff{thm_final_convergence_rate_depmeas_eq2} and \reff{thm_final_convergence_rate_depmeas_eq3} into \reff{thm_final_convergence_rate_depmeas_eq1} yields the result.
\end{proof}

A specific expression for $\IC_L$ is available but due to its complicated form we reduce the statement to its formal existence.

\begin{rem}
    Note that in the case of independent observations $X_i$, one can choose $\Lambda^{mix}(x) = \Lambda^{dep}(x) = x$ in Theorems \ref{thm_final_convergence_rate_mixing} and \ref{thm_final_convergence_rate_depmeas} which yields then the same result as Theorem 1 in \cite{schmidthieber2017}.
\end{rem}

To get a glimpse on the convergence rates which can be achieved, we formulate the following two corollaries of Theorem \ref{thm_final_convergence_rate_mixing}. Due to the similar form, an analogue is available in the case of the functional dependence measure. The first corollary is a simple consequence of Lemma \ref{lemma_varphi_specialcase} and Theorem \ref{thm_final_convergence_rate_mixing} in the case of polynomial decaying dependence.

\begin{cor}[Mixing and polynomial decay]\label{corollary_mixing1}
    Suppose that Assumptions \ref{ass_subgaussian}, \ref{ass_autoencoder} and \ref{ass_smoothness} hold and that $X_i$ is mixing with coefficients satisfying $\sum_{k=0}^{\infty}k^{\alpha-1}\beta^{mix}(k) < \infty$ for some $\alpha > 1$. Let
    \[
        A = \min\{\frac{\beta_{dec}}{t_{dec}}, \frac{\beta_{enc,0}}{t_{enc,0}}, \frac{\beta_{enc,1}}{t_{enc,1}}\}.
    \]
    If Assumption \ref{ass_network} is satisfied with
    \[
        N = \Big\lceil n^{\frac{\frac{\alpha}{\alpha+1}}{2A+\frac{\alpha}{\alpha+1}}} \Big\rceil,
    \]
    then
    \[
       \IE D(\hat f^{net}) \lsim n^{-\frac{2A\cdot \frac{\alpha}{\alpha+1}}{2A + \frac{\alpha}{\alpha+1}}} \log(n)^{\frac{3\alpha}{\alpha+1}}.
    \]
\end{cor}

We now investigate this rate for a specific model from Example \ref{ex_additive_model}(2) with only one lag $r = 1$. Suppose that $t_{enc,0} = \tilde d$ and
\[
    f_{dec}, \quad g_{i_1,...,i_{\tilde d}} \in C^{\beta}([0,1]^{\tilde d},K)
\]
with some $\beta > 0$. This means that the encoder function produces a compressed result of $\tilde d \le d$ components, where each of the $\tilde d$ components is constructed as follows: For each possibility to choose $\tilde d$ from $d$ arguments, a different function can be used to process the given values. These results are all summed up. Since the summation is infinitely often differentiable with bounded derivatives, in Corollary \ref{corollary_mixing1} we have
\[
    A = \min\{\frac{\beta}{\tilde d}, \frac{\infty}{d^{\tilde d}}, \frac{\beta}{\tilde d}\} = \frac{\beta}{\tilde d},
\]
which yields the following result.

\begin{cor}\label{corollary_mixing2}
        Suppose that Assumption \ref{ass_subgaussian} holds and that $X_i$ is mixing with coefficients satisfying $\sum_{k=0}^{\infty}k^{\alpha-1}\beta^{mix}(k) < \infty$ for some $\alpha > 1$. Let Assumption \ref{ass_network} be satisfied with
    \[
        N = \Big\lceil n^{\tilde d\cdot \frac{\frac{\alpha}{\alpha+1}}{2\beta+\tilde d\cdot \frac{\alpha}{\alpha+1}}} \Big\rceil.
    \]
    Then,
    \[
       \IE D(\hat f^{net}) \lsim n^{-\frac{2\beta\cdot \frac{\alpha}{\alpha+1}}{2\beta + \tilde d\cdot \frac{\alpha}{\alpha+1}}}\log(n)^{\frac{3\alpha}{\alpha+1}}.
    \]
\end{cor}

In contrast to the rate of a naive estimator mentioned in \reff{naive_nonparametric_rate} which suffers from the curse of the dimension $d$, we are therefore able to formulate structural conditions on the evolution of the time series to obtain much faster rates which only depend on the compressed dimension $\tilde d \in \{1,...,d\}$. Of course, the list in Example \ref{ex_additive_model} is not exhaustive and much more models are suitable for our theory.

\section{Oracle-type inequalities for minimum empirical risk estimators}
\label{sec_oracle}

In this section, we consider general properties of minimum empirical risk estimators
\[
    \hat f \in \argmin_{f\in \sF} \hat R_n(f), \quad\quad \hat R_n(f) = \frac{1}{n}\sum_{i=r+1}^{n}\frac{1}{d}\big|X_i - f(\IX_{i-1})\big|_2^2 \IW(\IX_{i-1})
\]
over function classes
\[
    \sF \subset \{f:\IR^{dr} \to \IR^{d} \text{ measurable}\}.
\]
Here, $\IW:\IR^{dr} \to [0,1]$ is an arbitrary (measurable) weight function. We ask $\sF$ to satisfy
\[
    \sup_{f\in\sF}\sup_{j\in \{1,...,d\}}\sup_{x\in \text{supp}(\IW)}|f(x)| \le F
\]
for some constant $F > 0$.

Let $N(\delta, \sF, \|\cdot\|_{\infty})$ denote the smallest number of $\delta$-brackets with respect to $\|f\|_{\infty} := \sup_{j\in \{1,...,d\}}\|f_j\|_{\infty}$ which is needed to cover $\sF$, and let $H(\delta) := \log N(\delta, \sF, \|\cdot\|_{\infty})$ denote the corresponding bracketing entropy. 

\subsection{Oracle inequalities under absolutely regular mixing} \label{oracle_arm}
In the case that the process $X_i$, $i = 1,...,n$, is $\beta$-mixing, we obtain the following result which is proven in Theorem \ref{theorem_oracle_inequality} of Section \ref{oracle_arm_proof}.

\begin{thm}\label{theorem_oracle_inequality_present}
    Let Assumptions \ref{ass_subgaussian} and  \ref{ass_compatibility} hold and let $\Lambda^{mix}(\cdot)$ be the function defined in \ref{definition_lambda_mixing}. Then, for any $\delta \in (0,1), \eta >0$ there exists a constant $\IC = \IC(\eta,c_0,r,C_{\beta^{mix}}, C_{\varepsilon},F)$ such that
    \[
        \IE D(\hat f) \le (1+\eta)^2 \inf_{f\in \sF}D(f) + \IC\cdot \big\{ \Lambda^{mix}(\frac{H(\delta)}{n}) + \delta\big\}.
    \]
\end{thm}

\subsection{Oracle inequalities under functional dependence}

Suppose that $X_i$, $i = 1,...,n$, is of the form \reff{representation_x}. In this case, we have to impose smoothness assumptions on the underlying function class $\sF$, on $f_0$ and on $\IW$ in order to quantify the dependence of functions of $X_i$. Suppose that there exist $\varsigma, K, L_{\sF} > 0$ such that for all $f\in \sF$, $x,x' \in \IR^{dr}$,
\begin{eqnarray*}
    |\IW(x) - \IW(x')| &\le& \frac{1}{\varsigma}|x-x'|_{\infty},\\
    |f_{0}(x) - f_{0}(x')|_{\infty} &\le& K |x-x'|_{\infty},\\
    |f(x) - f(x')|_{\infty} &\le& L_{\sF}|x-x'|_{\infty}.
\end{eqnarray*}

The following theorem is proven in Theorem \ref{theorem_oracle_inequality_dep} in the Appendix. An example for an appropriate $\IW$ with support $[0,1]^{dr}$ is given in \reff{definition_weight_function}.

\begin{thm}\label{theorem_oracle_inequality_dep_present}
    Suppose that Assumption \ref{ass_subgaussian} and Assumption \ref{ass_compatibility2} hold with $L_{\sG} = 2dr\big(\frac{2}{\varsigma} + \frac{(L_{\sF} + K)}{F}\big)$. Let $\Lambda^{dep}(\cdot)$ be the function defined in \reff{definition_lambda_depmeas}. Then, for any $\delta \in (0,1),\eta > 0$ there exists a constant $\IC = \IC(\eta, C_{\varepsilon}, F)$ such that
    \[
        \IE D(\hat f) \le (1+\eta)^2 \inf_{f\in \sF}D(f) + \IC\cdot \big\{ \Lambda^{dep}(\frac{H(\delta)}{n}) + \delta\big\}.
    \]
\end{thm}

We give some short remarks.

\begin{rem}
    \begin{enumerate}
        \item While in Theorem \ref{theorem_oracle_inequality_present}, the parameter $r$ is directly contained in the constant $\IC$, in Theorem \ref{theorem_oracle_inequality_dep_present} it is contained in $\Lambda$ via Assumption \ref{ass_compatibility2}. Additionally, in the latter theorem the dimension $d$ is incorporated through \reff{ass_compatibility2_eq1} and may be incorporated through $L_{\sG}$. Besides these facts, both theorems are rather similar.
        \item Theorems \ref{theorem_oracle_inequality_present} and \ref{theorem_oracle_inequality_dep_present} are rather general and can be applied to any function class which allow for a measurement of their size via brackets with respect to the $\|\cdot\|_{\infty}$-norm. It may therefore be of interest for other nonparametric estimators.
        \item Theorems \ref{theorem_oracle_inequality_present} and \ref{theorem_oracle_inequality_dep_present} can be seen as generalizations of Lemma 4 in \cite{schmidthieber2017} for dependent observations.
    \end{enumerate}
\end{rem}

\section{Simulations}
\label{sec_simulation}
In this section, we discuss the behavior of the estimator $\hat f$ from \reff{definition_network_estimator}, respectively its approximation obtained with stochastic gradient descent. During the presentation, $v'$ denotes the transpose of a vector or matrix $v$.

\subsection{Simulated data}

We first consider a low-dimensional example given by
\[
	X_i = f_0(X_{i-1})+\varepsilon_i
\]
where $\varepsilon_i \sim \mathcal{N}(0,0.5 I_{5\times 5})$ ($I_{5\times 5}$ denoting the $5$-dimensional identity matrix) and
\begin{equation}
	f_0:\IR^5 \to \IR^5, \quad f_0(x) = vax \label{low_d_model}
\end{equation}
for $a = (0.5,0.6,0.2,0.3,0.5) \in \IR^5$ and $v=(0.4,0.6,0.5,-0.2,0.5)' \in \IR^5$, that is,
\[
    X_i = v\cdot \sum_{j=1}^5 a_j x_j.
\]
We generate $n = 1000$ observations $X_1,...,X_n$ following the above recursion and use $n_{test} - n = 1000$ further realizations of the time series to quantify the true prediction error. For the fitting process, we use an encoder-decoder network of the form
\[
	p = (5,20,10,1,10,20,5), \quad\quad L = 5,
\]
that is, the network encodes the given information to one dimension and afterwards spreads the value again to 5 dimensions. The network is learned with a standard stochastic gradient descent method of learning rate $\gamma = 0.003$ for the first 30 epochs and $\gamma = 0.0002$ afterwards. Furthermore, we use $\lambda = 0.00001$ and the ReLU activation function. We can deduce from Figure \ref{ex_ar_low} that the neural network can easily learn the underlying function $f$ already after $\approx 40$ epochs. We surmise that for low dimensional data the testing error can be seen on par with the training error, converging rapidly towards the optimal prediction error $\frac{1}{5}\IE[|\varepsilon_1|^2] = 0.5^2 = 0.25$.

We now turn to an example which is of higher dimension. We therefore take the same model but with the function
\begin{equation}
	f_0:\IR^{30} \to \IR^{30}, \quad f_0(x) = vax, \label{high_d_model}
\end{equation}
where we define the vector $s = (0.05, -0.05 , ... , 0.05 , -0.05) \in \IR^{24}$ that alternates between the values $0.05$ and $-0.05$ and put
\begin{eqnarray*}
    a &=& \begin{pmatrix}
        0.3 & 0.6 & 0.5 & s & 0 & -1 & 0.4 \\
        0.5 & -0.6 & 0.2 & s & 0.4 & 0.9 & 1
    \end{pmatrix} \in \IR^{2 \times 30},\\
    v &=& \begin{pmatrix}
        0.4 & 0.4 & \dots & 0.4 & 0.4 \\
        0.5 & -0.3 & \dots & 0.5 & -0.3
    \end{pmatrix}^t \in \IR^{30 \times 2}.
\end{eqnarray*}
for alternating values $0.5$ and $-0.3$.
The network architecture is adjusted to
\[
	p = (30,60,30,2,30,60,30), \quad\quad L = 5,
\]
using again a stochastic gradient descent method with learning rate $\gamma = 0.0002$ for the first 50 epochs and $\gamma = 0.00002$ afterwards. Furthermore, we set $\lambda = 0.00001$ and employ the ReLU activation function. Although the network is dealing with an input and output of dimension 30, Figure \ref{ex_ar_high} shows that a good prediction already can be realized and most of the information can be preserved despite the data passing a layer of only two dimensions.

\begin{figure} \label{ex_ar_low}
\centering
\includegraphics[width=12cm]{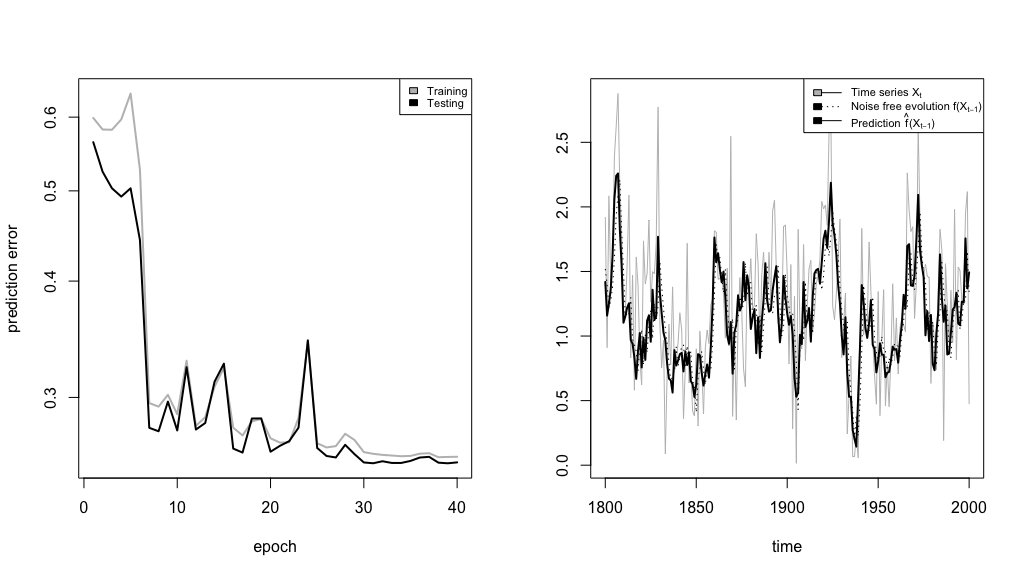}
\caption{We depict the learning process under model \reff{low_d_model}. After 40 epochs the neural network learned the underlying function $f_0$ provided by a noisy version of the data. We can clearly see that the neural network is able to predict the noise free evolution of the times series.}

\end{figure}

\begin{figure} \label{ex_ar_high}
\centering
\includegraphics[width=12cm]{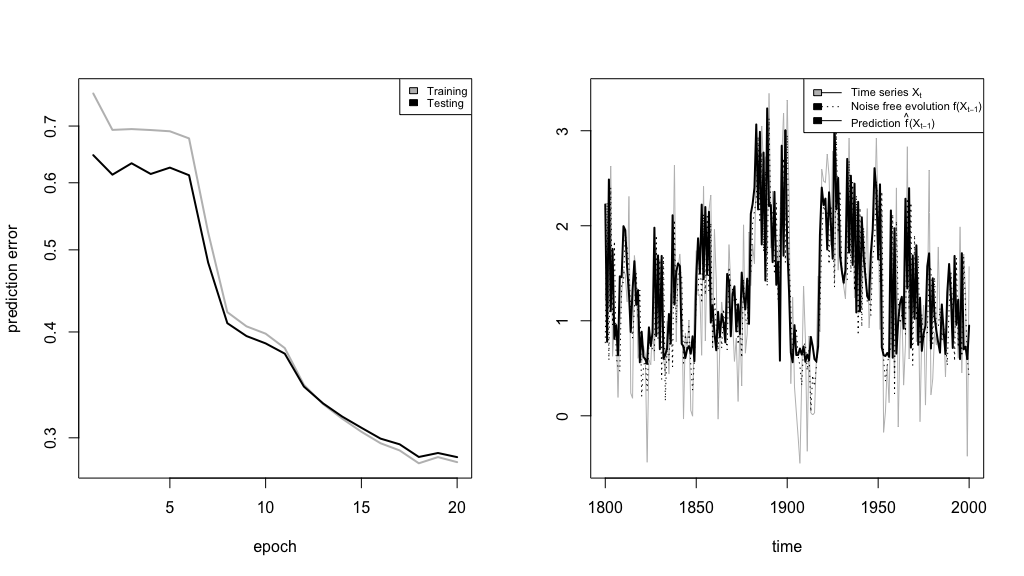}
\includegraphics[width=12cm]{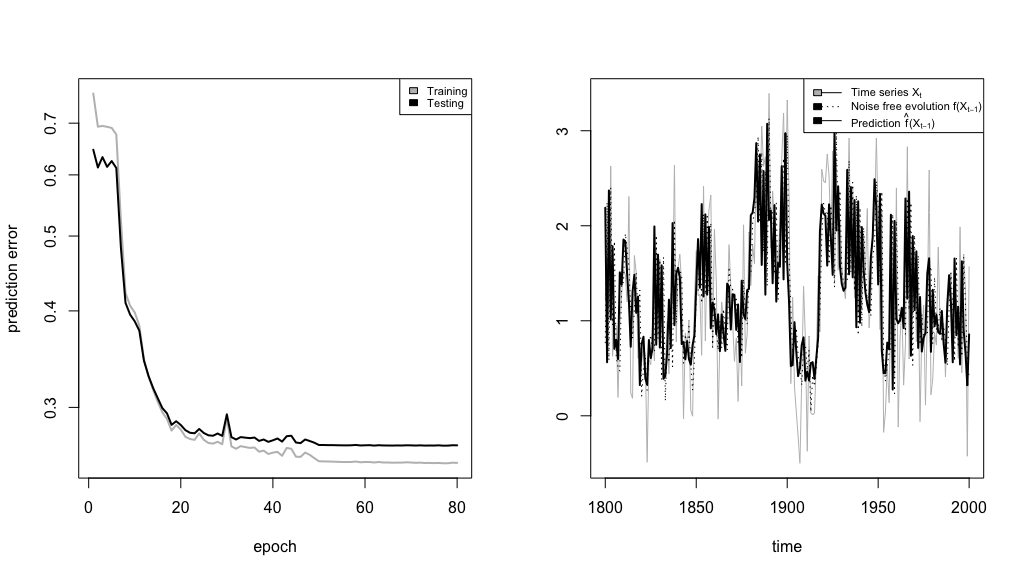}
\caption{The underlying model here is given by \reff{high_d_model}. After 20 epochs the neural network learned the overall behavior of the function $f_0$. The network has still the potential to improve for the lower peaks. After about 40 epochs the learning process can be seen as completed. Due to overfitting, the testing error now begins to slowly increase.}
\end{figure}

\subsection{Real data application}

For a simulation study we consider the weather data of $d=32$ German cities provided by the Deutscher Wetterdienst (German Meteorological Service). Note that the cities chosen are spread throughout Germany which can be seen in Figure \ref{fig:germany_map}. The data we are interested in is the daily mean of temperature and can be found on \url{https://opendata.dwd.de/climate_environment/CDC/observations_germany/climate/daily/kl/historical}. In total we observe 4779 temperature values for each city over the period of 2006/07/01 to 2019/08/01.
A subset of $n_\text{train} = 4415$ values serves as training data for the network and represents the data from 2006/07/01 to 2018/07/31. We validate our prediction on the year 2018/08/01 to 2019/07/31 which contains $n - n_\text{train} = 354$ values. For fitting, we use a network with architecture
\[
    \sF(5,(rd,rd,24,m,24,d,d))
\]
where $r \in \{1,2,3,5\}$ and $m \in \{4,6,8,10\}$, apply the stochastic gradient descent for learning the approximation $\hat f^{\approx}$ of $\hat f$ over 150 epochs. The learning rate is chosen to be $\gamma = 0.000002$ until epoch 45 and $\gamma = 0.0000002$ thereafter. We let the simulation run 5 times over every step $r$ for each network described by $m$.

In Figure \ref{sim_error} we summarize the prediction errors $D(\hat f^{\approx})$ obtained during the testing process.
The smallest prediction error value can be found for $r = 2$ (that is, using $X_{i-1}, X_{i-2}$ for predicting $X_i$) with a bottleneck layer of $10$ hidden units. However, it is also possible to take a layer with $m \in \{6,8\}$ hidden units and still obtain a comparable result. 
Thus, we surmise that according to our model when considering the errors, the weather should be predicted based on the two previous days. Taking the day or more than three days before the date of interest does not seem to yield a good prediction.

In comparison, the naive prediction method $\hat f^{naive}$ of taking the temperature value of the current day as it is
to forecast the next day's value yields an error of $D(\hat f^{naive}) \approx 4.99$. We therefore see that employing encoder-decoder neural networks produces more accurate predictions.

For $r=2, m = 6$, we depict the development of the training and testing error for the network $\sF(5,(2d,2d,24,6,24,d,d))$ in Figure \ref{fig:error}. After 45 epochs the testing error already drops down to a magnitude of $4$ which means that we anticipate a deviation of $2$ Kelvin for the prediction itself. The fitting process is displayed for the city of Mannheim in Figure \ref{fig:mannheim}.

Additionally, the $1$-step predictor can be used to forecast $k$-steps ahead in time by applying the learned neural network $k$-times, accordingly. In our example, we applied this to the next week's temperature, i.e. $k=7$. The chosen predictor with architecture $\sF(5,(d,d,24,6,24,d,d))$ yields a deviation of around 4.44 Kelvin.

\begin{figure}[ht]
    \centering
    \includegraphics[width=7.5cm]{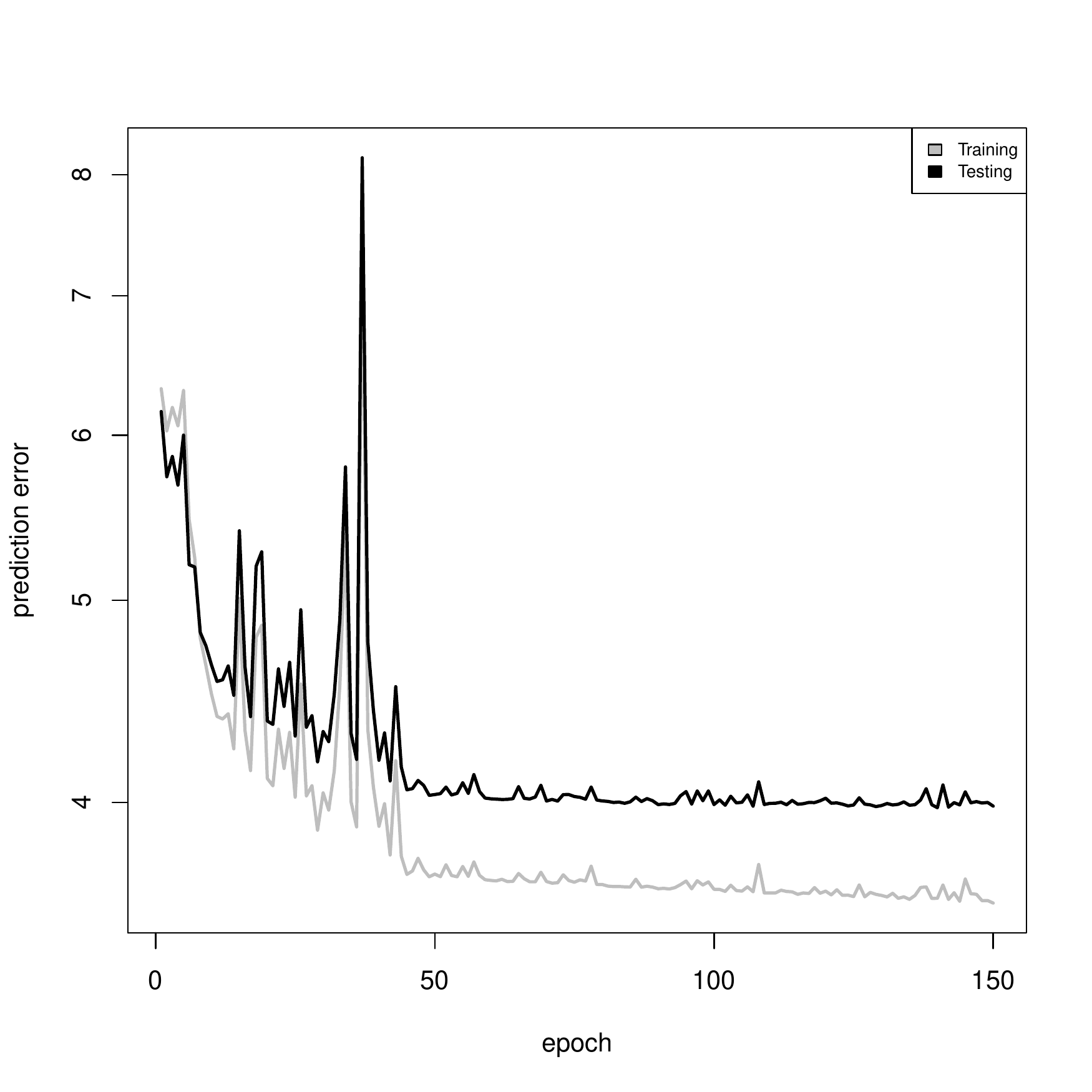}
    \caption{Depicted is the training and testing error in the learning process of the network $\sF(5,(2d,2d,24,6,24,d,d))$ applied to the weather data. We clearly see that consistently, as expected, the testing error is higher than the training error. At an early stage the network already learns basic properties of the evolution scheme of the time series because the testing error rapidly drops. After 45 epochs the error is in the range of 4.}
    \label{fig:error}
\end{figure}

\begin{figure}[ht]
    \centering
    \includegraphics[width=10cm]{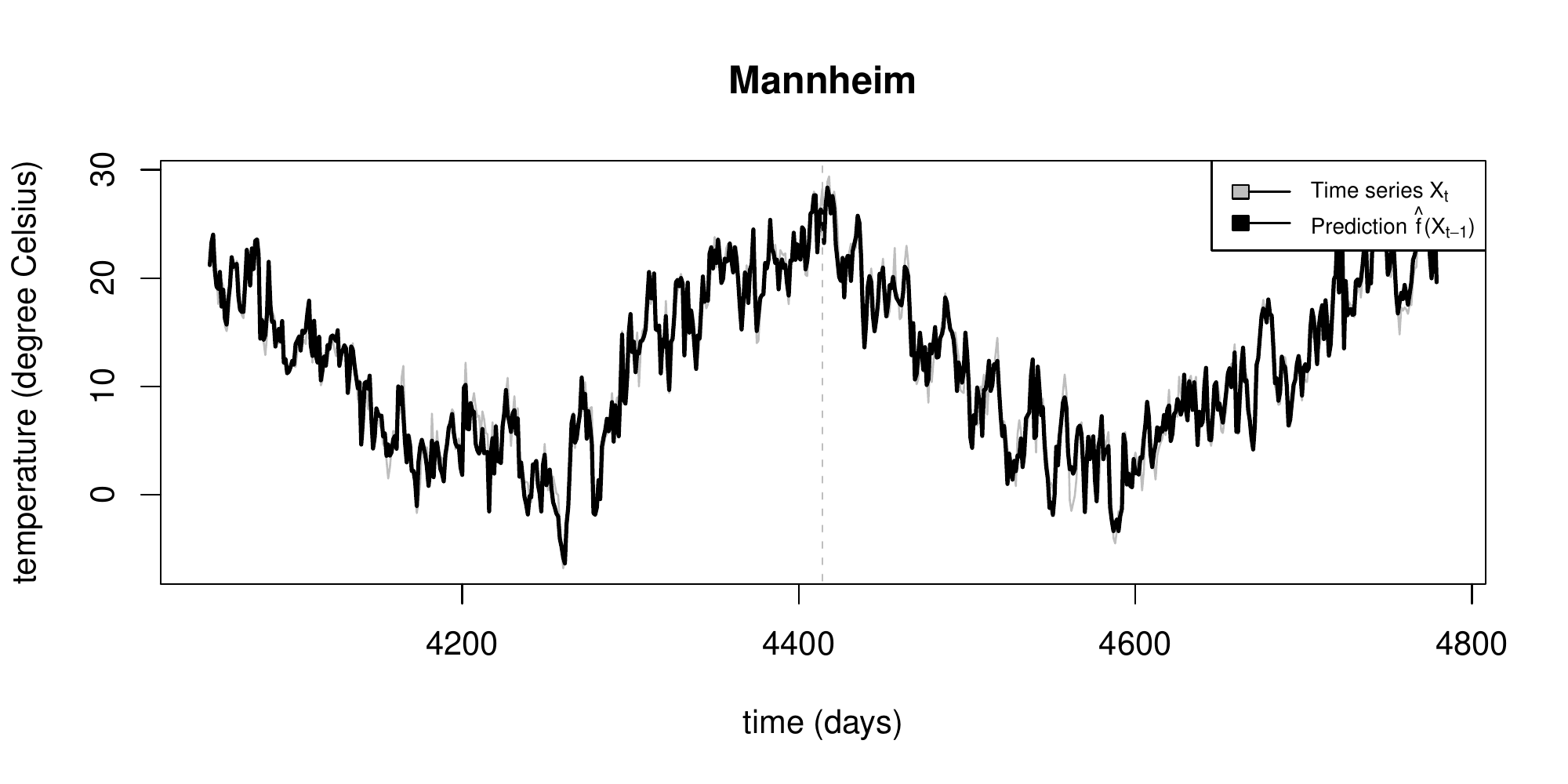}
    \caption{The graphic shows the daily mean temperature data from 2017/08/01 to 2019/07/31 measured in Mannheim. Note that we continuously count the days from 1 to 4779 beginning on 2006/07/01 (day 1). The training process ends on 2018/31/07 (day 4414), indicated by the gray vertical dashed line in the middle. Beginning on 2018/08/01 we see the values predicted $\hat f^{\approx}(X_{i-1},X_{i-2})$ by the learned neural network on top of the actual data observed.}
    \label{fig:mannheim}
\end{figure}

\begin{figure}
\centering
\fbox{\begin{tabular}{|ll|lllll|}
\hline
\multicolumn{2}{|l|}{$p=1$} & \multicolumn{5}{l|}{prediction error upon validation} \\ \hline
layer m    &     $4$      &  4.81 & 4.77 & 4.63 & 4.68 & 4.86\\
           &     $6$      &  4.41 & 4.65 & 4.75 & 4.83 & 4.64\\
           &     $8$      &  4.41 & 4.49 & 4.42 & 4.47 & 4.45\\
           &     $10$     &  4.41 & 4.55 & 4.44 & 4.45 & 4.45\\ \hline \hline
\multicolumn{2}{|l}{$p=2$} & \multicolumn{5}{l|}{} \\ \hline
layer m    &     $4$      &  4.63 & 4.22 & 4.41 & 4.19 & 4.03\\
           &     $6$      &  3.98 & 4.10 & 4.21 & 4.16 & 4.22\\
           &     $8$      &  3.98 & 3.95 & 4.23 & 4.05 & 4.01\\
           &     $10$     &  4.11 & 4.09 & 3.93 & 3.93 & 4.02\\ \hline \hline
\multicolumn{2}{|l}{$p=3$} & \multicolumn{5}{l|}{} \\ \hline
layer m   &     $4$      &  4.30 & 4.79 & 4.11 & 4.72 & 4.10\\
           &     $6$      &  4.27 & 4.46 & 4.18 & 4.04 & 4.18\\
           &     $8$      &  4.36 & 4.29 & 4.08 & 4.24 & 4.28\\
           &     $10$     &  4.27 & 4.12 & 4.08 & 4.15 & 4.28\\ \hline \hline
\multicolumn{2}{|l}{$p=5$} & \multicolumn{5}{l|}{} \\ \hline
layer m    &     $4$      &  4.28 & 4.75 & 4.83 & 4.48 & 4.85\\
           &     $6$      &  4.10 & 4.27 & 4.28 & 4.34 & 4.71\\
           &     $8$      &  4.81 & 4.09 & 4.06 & 4.42 & 4.45\\
           &     $10$     &  4.24 & 4.47 & 4.28 & 4.37 & 4.36\\ \hline \hline
\end{tabular}}
\caption{The testing errors obtained during the simulation. For each of the 4 distinct network architectures and each of $r$-step predictions we ran the simulation 5 times.}
\label{sim_error}
\end{figure}

\begin{figure}[ht]
    \centering
    \includegraphics[width=5.5cm]{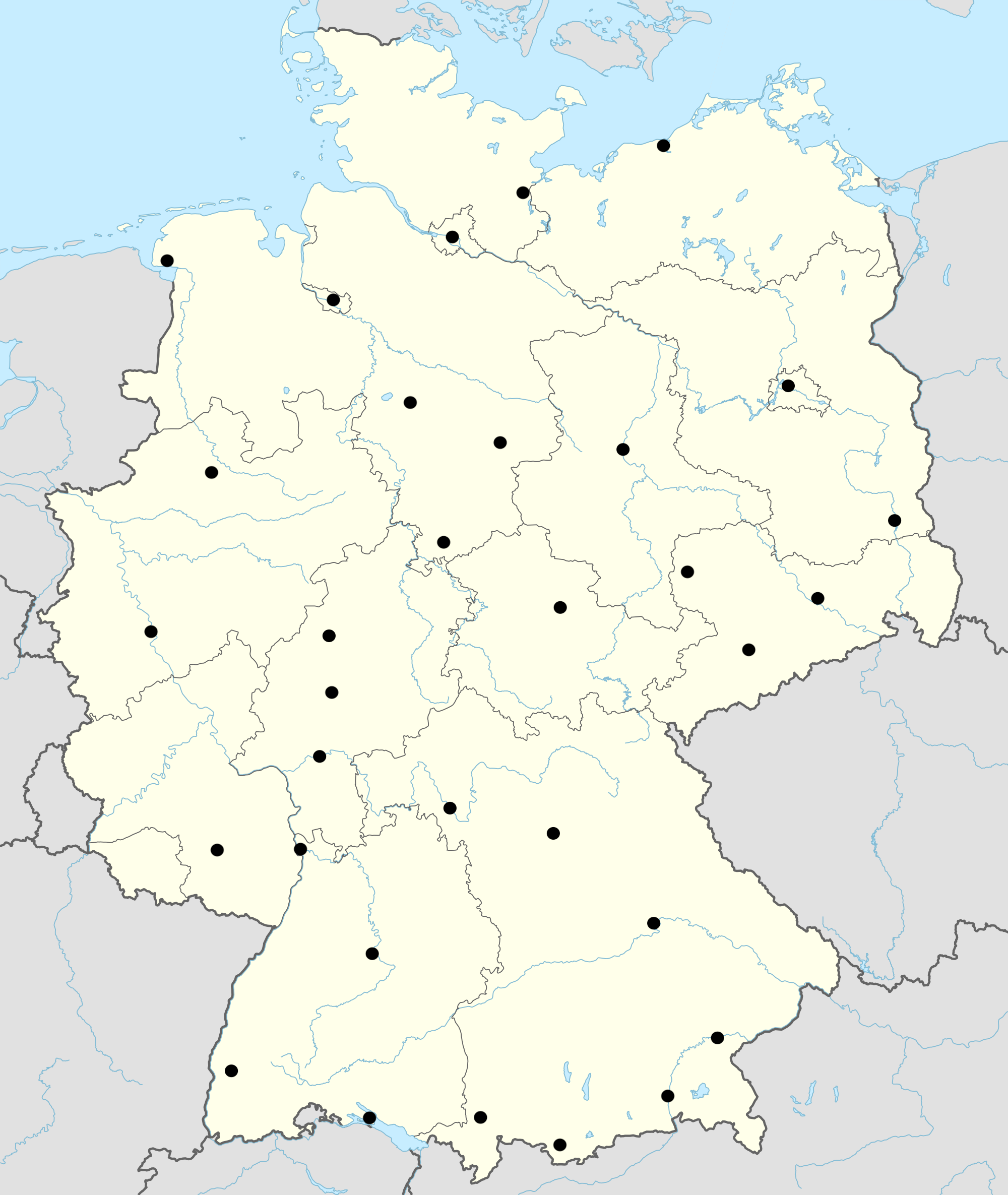}
    \caption{We collected weather data from the cities of Berlin, Braunschweig, Bremen, Chemnitz, Cottbus, Dresden, Erfurt, Frankfurt, Freiburg, Garmisch-Patenkirchen, Göttingen, Münster, Hamburg, Hannover, Kaiserslautern, Kempten, Köln, Konstanz, Leipzig, Lübeck, Magdeburg, Cölbe, Mühldorf, München, Nürnberg, Regensburg, Rosenheim, Rostock, Stuttgart, Würzburg, Emden and Mannheim.}
    \label{fig:germany_map}
\end{figure}

\section{Conclusion}
\label{sec_conclusion}

In this work, we have proposed a method to forecast high-dimensional time series with encoder-decoder neural networks and quantified their prediction abilities theoretically by a convergence rate. The encoder-decoder structure we used is fundamental to circumvent the curse of dimension. Besides the fact that the corresponding neural network is required to have a similar encoder-decoder structure to avoid overfitting in practice, we formulated conditions on the network parameters such as bounds for the number of layers or active parameters. These conditions are similar to \cite{schmidthieber2017} since we have used the same approximation results. 

Our theory can be seen as an extension of the upper bounds found in \cite{schmidthieber2017} to dependent observations and high-dimensional outputs. 
To prove the results, we derived oracle-type inequalities for minimizers of the empirical prediction error under mixing or functional dependence. These results may be of independent interest.

We studied the performance of our neural network estimators with simulated data and saw that the estimators could detect and adapt to a specific encoder-decoder structure of the true evolution function quite successfully. We applied our procedure to temperature data and have shown that without too much tuning we were able to outperform the naive forecast which proposes today's temperature for tomorrow.

A natural extension of our work would be the proof of lower bounds under given structural assumptions. Furthermore, a more general model of ARCH-type
\[
    X_i = f_0(\IX_{i-1}) + \sigma(\IX_{i-1})\eps_i, \quad\quad i = r+1,...,n,
\]
with additional function $\sigma:\IR^{dr} \to \IR^{d\times d}$ of matrix form could be considered. We conjecture that in such models, similar convergence rates could be obtained under appropriate structural assumptions on $\sigma(\cdot)$. Finally, it may be interesting to give more precise results about approximations of the estimators which are obtained via stochastic gradient descent. Similar to the theory of Boosting, one could hope for explicit or adaptive stopping rules. 

\section{Appendix: Selected proofs of Section \ref{oracle_arm}} \label{oracle_arm_proof}

Recall the definition of the $\beta$-mixing coefficients $\beta^{mix}(k)$, $k\in\IN_0$ from \reff{definition_beta_mixing}. In this section, we use the abbreviation $\beta(\cdot) = \beta^{mix}(\cdot)$.

We now introduce the $\|\cdot\|_{2,\beta}$-norm which originally was defined in \cite{rio1995invariance}.

Define $\beta(t) = \beta(\lfloor t\rfloor)$ for $t \ge 1$ and $\beta(t) = 1$, otherwise. For some cadlag function $g:I \to \IR$ defined on a domain $I \subset \IR$, the cadlag inverse is defined as
\[
    g^{-1}(u) := \inf\{s\in I: f(s) \le u\},
\]
which we especially use for $\beta^{-1}(u)$. For any measurable $h:\IR \to \IR$, let $Q_h(u)$ denote the quantile function of $h(X_1)$, that is, $Q_h(u)$ is the cadlag inverse of $t \mapsto \IP(h(X_1) \le t)$. Let
\[
    \|h\|_{2,\beta} := \Big(\int_0^{1}\beta^{-1}(u) Q_h(u)^2 du\Big)^{1/2}.
\]

This norm can be used to upper bound the variance of a sum $\sum_{i=1}^k h(X_i)$. Furthermore, it is possible to upper bound $\|h\|_{2,\beta}$ in terms of $\|h\|_1 = \IE|h(X_1)|$ and $\|h\|_{\infty}$ which we will need in the proofs to relate the variance of the empirical risk with the risk itself. Let
\begin{eqnarray*}
    \Phi := \{\phi:[0,\infty) \to [0,\infty)&|& \text{$\phi$ increasing, convex, differentiable,}\\
    && \text{$\phi(0) = 0$ and $\lim_{x\to\infty}\frac{\phi(x)}{x} = \infty$}\}.
\end{eqnarray*}
For $\phi \in \Phi$, let $\phi^{*}(y) := \sup_{y > 0}\{xy - \phi(x)\}$ be the convex dual function. Define the Orlicz norm associated to $\phi(x^2)$ via
\[
    \|h\|_{\phi,2} := \inf\{c > 0: \IE\phi\Big(\big(\frac{|h(X_1)|}{c}\big)^2\Big) \le 1\}.
\]

The following two results are from  \cite[Proposition 1 and Lemma 2]{rio1995invariance}.

\begin{lem}[Variance bounds and bound of $\|\cdot\|_{2,\beta}$-norm]\label{lemma_mixing_norm}
    For $k\in\IN$,
    \[
        \Var\Big(\sum_{i=1}^{k}h(X_i)\Big) \le 4k \|h\|_{2,\beta}^2.
    \]
    For $\phi \in \Phi$, assume that $\int_0^{1} \phi^{*}(\beta^{-1}(u)) < \infty$. Then,
    \[
        \|h\|_{2,\beta} \le C_{\beta}\cdot \|h\|_{\phi,2}, \quad\quad C_{\beta}:=\Big( 1 + \int_{0}^{1}\phi^{*}(\beta^{-1}(u)) du\Big)^{1/2}.
    \]
    If $\|h\|_{\infty} \le 1$, then
    \begin{equation}
         \|h\|_{\phi,2} \le C_{\beta} \varphi(\|h\|_1)\label{lemma_mixing_norm_res1}
    \end{equation}
    where $\varphi(x) := \phi^{-1}(x^{-1})^{-1/2}$.
\end{lem}

Only the last statement needs to be proven and is postponed to the Appendix in the Supplementary Material. The main goal of this section is to prove Theorem \ref{theorem_oracle_inequality}. To do so, we use techniques and decomposition ideas from \cite{Dedeck02}, \cite{rio1995} and \cite{liebscher1996}. We begin by establishing maximal inequalities under mixing. The proofs can be found the Appendix of the Supplementary Material, as well.

\subsection{Maximal inequalities under mixing} 
Let $\sG \subset \{g:\IR^{dr} \to \IR \text{ measurable}\}$ be a finite class of functions and \[
    S_n(g) := \sum_{i=r+1}^{n}\{g(\IX_{i-1}) - \IE g(\IX_{i-1})\}.
\]
In the following, let $H = 1 \vee \log |\sG|$. Recall 
\[
    q^{*}(x) = q^{*,mix}(x) = \min\{q\in\IN: \beta^{mix}(q) \le qx\}.
\]

\begin{lem}[Maximal inequalities for mixing sequences]\label{bernstein_mixing}
    Suppose that $\sup_{g\in \sG}\|g\|_{\infty} \le 1$ and that there exists $\nu(g) > 0$ such that $\sup_{g\in \sG}\|\frac{g}{\nu(g)}\|_{2,\beta} \le 1$. Then there exists another process $S_n^{\circ}(g)$ and some universal constant $c > 0$ such that
    \begin{equation}
        \IE \sup_{g\in \sG}|S_n(g) - S_n^{\circ}(g)| \le cnr\cdot  q^{*}(\frac{H}{n})\frac{H}{n}\label{bernstein_mixing_res1}.
    \end{equation}
    Furthermore, with $N(g) := q^{*}(\frac{H}{n})\sqrt{\frac{H}{n}}\vee \nu(g)$,
    \begin{enumerate}
         \item \begin{equation}
         \IE \sup_{g\in \sG}|\frac{S_n^{\circ}(g)}{N(g)}| \le c \sqrt{nrH},\label{bernstein_mixing_res2}
    \end{equation}
        \item \begin{equation}
         \IE\big[\sup_{g\in \sG}|\frac{S_n^{\circ}(g)}{N(g)}|^2\big] \le c nr^2H.\label{bernstein_mixing_res3}
    \end{equation}
    \end{enumerate}
\end{lem}

Now, for $\sG \subset \{g:\IR^{dr} \to \IR^d \text{ measurable}\}$, define
\[
    M_n(g) := \sum_{i=1}^{n}\frac{1}{d}\langle \varepsilon_i,g(\IX_{i-1})\rangle.
\]

\begin{lem}[Maximal inequalities for mixing martingale sequences]\label{bernstein_mixing_martingale}
    Let Assumption \ref{ass_subgaussian} hold. Furthermore, assume that $X_i$ is $\beta$-mixing and $\beta^{mix}(\cdot)$ is submultiplicative in the sense of \reff{definition_beta_submultiplicative}. Suppose that $\sG$ is such that $\sup_{g\in \sG}\|g\|_{\infty} \le 1$ and that there exists $\nu(g) > 0$ such that $\sup_{g\in \sG}\|\frac{|g(\IX_{r})|_2}{\sqrt{d}\nu(g)}\|_2 \le 1$. 
    Then, there exists another process $M_n^{\circ}(g)$ and some universal constant $c > 0$ such that
    \begin{equation}
        \IE \sup_{g\in \sG}|M_n(g) - M_n^{\circ}(g)| \le cr C_{\varepsilon}C_{\beta,sub}n\cdot q^{*}(\frac{H}{n})\frac{H}{n}\label{bernstein_mixing_martingale_res1}.
    \end{equation}
    Furthermore, with $N(g) := q^{*}(\frac{H}{n})\sqrt{\frac{H}{n}} \vee \nu(g)$, 
    \begin{enumerate}
         \item \begin{equation}
         \IE \sup_{g\in \sG}|\frac{M_n^{\circ}(g)}{N(g)}| \le c C_{\varepsilon}\sqrt{nH},\label{bernstein_mixing_martingale_res2}
    \end{equation}
        \item \begin{equation}
         \IE\big[\sup_{g\in \sG}|\frac{M_n^{\circ}(g)}{N(g)}|^2\big] \le c C_{\varepsilon}^2 nH.\label{bernstein_mixing_martingale_res3}
    \end{equation}
    \end{enumerate}
\end{lem}

\subsection{Oracle inequalities under mixing}
Recall and define
\begin{eqnarray*}
    D(f) &=& \frac{1}{d}\IE[|f(\IX_r) - f_0(\IX_r)|_2^2 \IW(\IX_r)],\\
    \hat D_n(f) &:=& \frac{1}{n}\sum_{i=r+1}^{n}\frac{1}{d}|f(\IX_{i-1}) - f_0(\IX_{i-1})|_2^2 \IW(\IX_{i-1})
\end{eqnarray*}
where $\IW: \IR^{dr} \to [0,1]$ denotes an arbitrary weight function. In this section, we show an oracle-type inequality  for minimum empirical risk estimators
\[
    \hat f \in \argmin_{f\in \sF}\hat R_n(f), \quad\quad \hat R_n(f) = \frac{1}{n}\sum_{i=r+1}^{n}\frac{1}{d}|X_i - f(\IX_{i-1})|_2^2 \IW(\IX_{i-1})
\]
where $\IW:\IR^{dr} \to [0,1]$ is any weight function. The function classes considered are of the form
\[
    \sF \subset \{f = (f_j)_{j=1,...,d}:\IR^{dr} \to \IR^d \text{ measurable}\}
\]
and have to satisfy $\sup_{f\in\sF}\sup_{j\in \{1,...,d\}}\sup_{x\in \text{supp}(\IW)}|f_j(x)| \le F$. For the proof of the following theorem we require Lemma \ref{convergenceproof_part1} and Lemma \ref{convergenceproof_part2} which are shown below.

\begin{thm}\label{theorem_oracle_inequality}
    Let Assumptions \ref{ass_subgaussian} and  \ref{ass_compatibility} hold. Suppose that each $f = (f_j)_{j=1,...,d} \in \sF$ satisfies $\sup_{j=1,...d}\sup_{x\in \text{supp}(\IW)}|f_j(x)| \le F$. Let $\delta \in (0,1)$ and $H = \log N(\delta, \sF, \|\cdot\|_{\infty})$. Then, for any $\eta > 0$ there exists a constant $\IC = \IC(\eta,c_0,r,C_{\beta^{mix}},C_\eps,F)$ such that
    \[
        \IE D(\hat f) \le (1+\eta)^2 \inf_{f\in \sF}D(f) + \IC\cdot \big\{ \Lambda(\frac{H}{n}) + \delta\big\}.
    \]
\end{thm}
\begin{proof}[Proof of Theorem \ref{theorem_oracle_inequality}]
It holds that $\IE \hat D_n(f) = D(f)$, and
\[
    \IE \hat D_n(\hat f) = \IE\Big[\frac{1}{n}\sum_{i=r+1}^{n}\frac{1}{d}|\hat f(\IX_{i-1}) - f_0(\IX_{i-1})|_2^2 \IW(\IX_{i-1})\Big].
\]

By the model equation \reff{model_time_evolution}, it holds for any $f\in \sF$ that
\begin{eqnarray*}
    \hat R_n(f) &:=&\frac{1}{n}\sum_{i=r+1}^{n}\frac{1}{d}|X_i - f(\IX_{i-1})|_2^2 \IW(X_{i-1})\\
    &=& \frac{1}{n}\sum_{i=r+1}^{n}\frac{1}{d}|\varepsilon_i + (f_0(\IX_{i-1}) - f(\IX_{i-1}))|_2^2\IW(\IX_{i-1})\\
    &=& \frac{1}{n}\sum_{i=r+1}^{n}\frac{1}{d}|\varepsilon_i|_2^2 \IW(\IX_{i-1}) + \frac{1}{n}\sum_{i=r+1}^{n}\frac{1}{d}|f_0(\IX_{i-1}) - f(\IX_{i-1})|_2^2 \IW(\IX_{i-1})\\
    &&\quad\quad + \frac{2}{n}\sum_{i=r+1}^{n}\frac{1}{d}\langle \varepsilon_i, f_0(\IX_{i-1}) - f(\IX_{i-1})\rangle \IW(\IX_{i-1})\\
    &=& \frac{1}{n}\sum_{i=r+1}^{n}\frac{1}{d}|\varepsilon_i|_2^2 \IW(\IX_{i-1})  + \frac{2}{n}\sum_{i=r+1}^{n}\frac{1}{d}\langle \varepsilon_i, f_0(\IX_{i-1}) - f(\IX_{i-1})\rangle \IW(\IX_{i-1}) + \hat D_n(f).
\end{eqnarray*}
Since $\hat f = \arg\min_{f\in \sF}\hat R_n(f)$, we have for all $f\in \sF$, 
\begin{eqnarray*}
    \hat D_n(\hat f) &=& \underbrace{\hat R_n(\hat f)}_{\le \hat R_n(f)} - \frac{1}{n}\sum_{i=r+1}^{n}\frac{1}{d}|\varepsilon_i|_2^2 \IW(\IX_{i-1}) - \frac{2}{n}\sum_{i=r+1}^{n}\frac{1}{d}\langle \varepsilon_i, f_0(\IX_{i-1}) - \hat f(\IX_{i-1})\rangle \IW(\IX_{i-1})\\
    &\le& \hat D_n(f) + \frac{2}{n}\sum_{i=r+1}^{n}\frac{1}{d}\langle \varepsilon_i, f_0(\IX_{i-1}) - f(\IX_{i-1})\rangle \IW(\IX_{i-1})\\
    &&\quad\quad - \frac{2}{n}\sum_{i=r+1}^{n}\frac{1}{d}\langle \varepsilon_i, f_0(\IX_{i-1}) - \hat f(\IX_{i-1})\rangle \IW(\IX_{i-1}).
\end{eqnarray*}
Since $\IE[\varepsilon_i|\sA_{i-1}] = \IE \varepsilon_i = 0$ for $\sA_{i-1} = \sigma(\eps_{i-1}, \eps_{i-2},...)$,
\[
    \IE \hat D_n(\hat f) \le \underbrace{\IE \hat D_n(f)}_{=D(f)} +  \IE\Big[\frac{2}{n}\sum_{i=r+1}^{n}\langle \varepsilon_i, \hat f(\IX_{i-1})\rangle \IW(\IX_{i-1})\Big],
\]
that is,
\begin{equation}
    \IE \hat D_n(\hat f) \le \inf_{f \in \sF}D(f) + 2\IE\Big[\frac{1}{n}\sum_{i=r+1}^{n}\langle \varepsilon_i, \hat f(\IX_{i-1})\rangle \IW(\IX_{i-1})\Big].\label{riskproof_eq1}
\end{equation}
Let $\eta > 0$. Define
\begin{eqnarray*}
    R_{1,n} &:=& (1+\eta)crF^2 q^{*}(\frac{H}{n})\frac{H}{n} + \frac{\eta F^2}{2} (\varphi^{-1})^{*}\Big(2\frac{1+\eta}{\eta F^2}rC_{\beta}\sqrt{\frac{H}{n}}\Big),\\
    R_{1,\delta} &:=& crF^2C_{\beta}\sqrt{\frac{H}{n}}\varphi(2F^{-2}\delta^2),\\
    R_{2,n} &:=& cC_\eps C_{\beta,sub}rFq^{*}(\frac{H}{n})\frac{H}{n},\\
    R_{2,\delta} &:=& C_{\varepsilon}\delta + cC_\eps C_{\beta,sub}rF \sqrt{\frac{H}{n}}\delta.
\end{eqnarray*}
By Lemma \ref{convergenceproof_part1}, \reff{riskproof_eq1} and Lemma \ref{convergenceproof_part2},  
\begin{eqnarray*}
    \IE D(\hat f) &\le& (1+\eta)\IE \hat D_n(\hat f) + R_{1,n} + (1+\eta)R_{1,\delta}\\
    &\le& (1+\eta)\Big\{\inf_{f \in \sF}D(f) + 2\IE\Big[\frac{1}{n}\sum_{i=r+1}^{n}\langle \varepsilon_i, \hat f(\IX_{i-1})\rangle \IW(\IX_{i-1})\Big] + R_{1,\delta}\Big\} + R_{1,n} \\
    &\le& (1+\eta)\Big\{\inf_{f \in \sF}D(f) + 2cC_\eps C_{\beta,sub}rF \sqrt{\frac{H}{n}}\IE[D(\hat f)]^{1/2} + R_{2,n} + R_{2,\delta} + R_{1,\delta} \Big\} + R_{1,n} .
\end{eqnarray*}
Due to $2ab \le a^2 + b^2$ with $a := (1+\eta)cC_\eps C_{\beta,sub}rF\sqrt{\frac{H}{n}}(\frac{1+\eta}{\eta})^{1/2}$, $b:= (\frac{\eta}{1+\eta})^{1/2}\IE[D(\hat f)]^{1/2}$, we obtain
\begin{eqnarray*}
    \IE D(\hat f) &\le& (1+\eta)\inf_{f \in \sF}D(f) + \frac{(1+\eta)^3}{\eta}(cC_\eps C_{\beta,sub}rF)^2\frac{H}{n} + \frac{\eta}{1+\eta}\IE[D(\hat f)]\\
    &&\quad\quad\quad\quad + (1+\eta)(R_{2,n} + R_{2,\delta} +  R_{1,\delta}) + R_{1,n}.
\end{eqnarray*}
This implies
\begin{eqnarray}
    \IE D(\hat f) &\le& (1+\eta)^2 \inf_{f \in \sF}D(f) + (1+\eta)R_{1,n} \nonumber\\
    &&\quad\quad\quad\quad + (1+\eta)^2 (R_{2,n} + R_{2,\delta} + R_{1,\delta}) + \frac{(1+\eta)^4}{\eta}(cC_\eps C_{\beta,sub}rF)^2\frac{H}{n}.\label{riskproof_eq2}
\end{eqnarray}

Using Young's inequality applied to $\varphi^{-1}$ ($\varphi^{-1}$ is convex) and Lemma \ref{lemma_varphi}, we obtain
\[
    R_{1,\delta} \le crF^2 C_{\beta}(\varphi^{-1})^{*}\big( \sqrt{\frac{H}{n}}\big) + 2crC_{\beta} \delta^2 \le crF^2 C_{\beta}(4c_0)^2\Lambda(\frac{H}{n}) + 2crC_{\beta} \delta^2.
\]
By Lemma \ref{lemma_varphi2}, $R_{2,n} \le 2cC_\eps C_{\beta,sub} F \Lambda(\frac{H}{n})$, and
\[
    R_{1,n} \le (1+\eta)crF^2 \Lambda(\frac{H}{n}) + \frac{\eta F^2}{2} \Big(2\frac{1+\eta}{\eta F^2}rC_{\beta}\Big)^2 (4c_0)^2\Lambda(\frac{H}{n}).
\]
Furthermore,
\[
    R_{2,\delta} \le C_{\varepsilon}\delta + cC_\eps C_{\beta,sub}rF\delta^2 + cC_\eps C_{\beta,sub}rF\frac{H}{n}.
\]
Insertion of these results into \reff{riskproof_eq2} yields
\begin{eqnarray*}
    \IE D(\hat f) &\le& (1+\eta)^2 \inf_{f\in \sF}D(f)\\
    &&\quad\quad + \Lambda(\frac{H}{n})\cdot \Big\{(1+\eta)^2 cF^2 + 32(\frac{(1+\eta)^3}{\eta F^2})C_{\beta}^2r^2 c_0^2\\
    &&\quad\quad\quad\quad\quad\quad\quad\quad+ 2(1+\eta)^2 crC_\eps C_{\beta,sub}F + 16(1+\eta)^2c_0^2crC_\eps C_{\beta,sub} F^2 \Big\}\\
    &&\quad\quad + \delta^2 \cdot (1+\eta)^2 cC_\eps C_{\beta,sub}rF + C_{\varepsilon}\delta \cdot (1+\eta)^2 \\
    &&\qquad + \frac{(1+\eta)^4}{\eta}(cC_\eps C_{\beta,sub}rF)^2\frac{H}{n},
\end{eqnarray*}
which shows the assertion.
\end{proof}


\begin{lem}\label{convergenceproof_part1}
Suppose that Assumption \ref{ass_subgaussian}, \ref{ass_compatibility} hold. Assume that each $f\in \sF$ satisfies $\sup_{x\in \text{supp}(\IW)}|f(x)|_{\infty} \le F$. Let $H = \log N(\delta, \sF, \|\cdot\|_{\infty})$. Then there exists an universal constant $c > 0$ such that for any $\delta \in (0,1), \eta > 0$, 
\begin{eqnarray*}
    \IE D(\hat f) &\le& (1+\eta)\IE \hat D_n(\hat f) + \Big\{(1+\eta)crF^2 q^{*}(\frac{H}{n})\frac{H}{n} + \frac{\eta F^2}{2} (\varphi^{-1})^{*}\Big(2\frac{1+\eta}{\eta F^2}rC_{\beta}\sqrt{\frac{H}{n}}\Big)\Big\}\\
    &&\quad\quad + (1+\eta)crF^2C_{\beta}\sqrt{\frac{H}{n}}\varphi(2F^{-2}\delta^2).
\end{eqnarray*}
\end{lem}
\begin{proof}[Proof of Lemma \ref{convergenceproof_part1}]
    Let $(f_j)_{j = 1,...,\sN_n}$ be a $\delta$-covering of $\sF$, where $\sN_n := N(\delta, \sF, \|\cdot\|_{\infty})$. Let $j^{*} \in \{1,...,\sN_n\}$ be such that $\|\hat f - f_{j^{*}}\|_{\infty} \le \delta$ for all $k = 1,...,d$. Without loss of generality, assume that  $\delta \le F$.
    
    Let $(X_i')_{i\in\IZ}$ be an independent copy of the original time series $(X_i)_{i\in\IZ}$. Then $(\IX_{i-1},\IX_{i-1}')$ is still $\beta$-mixing with the coefficients $2\tilde \beta(q) = 2 \beta(q-r)$. Then,
    \begin{eqnarray}
        &&\big|\IE D(\hat f) - \IE \hat D_n(\hat f)|\nonumber\\
        &=& \Big| \IE \Big[\frac{1}{nd}\sum_{i=r+1}^{n}|\hat f(\IX_{i-1}') - f_0(\IX_{i-1}')|_2^2\IW(\IX_{i-1}')\nonumber\\
        &&\quad\quad\quad\quad\quad - \frac{1}{nd}\sum_{i=r+1}^{n}|\hat f(\IX_{i-1}) - f_0(\IX_{i-1})|_2^2\IW(\IX_{i-1})\Big]\Big|\nonumber\\
        &\le& \Big| \IE \Big[\frac{1}{nd}\sum_{i=r+1}^{n}|f_{j^{*}}(\IX_{i-1}') - f_0(\IX_{i-1}')|_2^2\IW(\IX_{i-1}')\nonumber\\
        &&\quad\quad\quad\quad\quad- \frac{1}{nd}\sum_{i=r+1}^{n}|f_{j^{*}}(\IX_{i-1}) - f_0(\IX_{i-1})|_2^2\IW(\IX_{i-1})\Big]\Big| + 10\delta F\nonumber\\
        &\le& \IE \Big|\frac{1}{n}\sum_{i=r+1}^{n}g_{j^{*}}(\IX_{i-1},\IX_{i-1}')\Big| + 10\delta F\nonumber\\
        &=& \frac{F^2}{n}\IE|S_n(g_{j^{*}})| + 10\delta F,\label{convergenceproof_part1_eq1}
    \end{eqnarray}
    where we have used that for $a = \hat f(\IX_{i-1}) - f_{j^{*}}(\IX_{i-1}),b = f_{j^{*}}(\IX_{i-1}) - f_0(\IX_{i-1}) \in \IR^d$, 
    \[
        \big| |a+b|_2^2 - |b|_2^2\big| = |a|_2^2 + 2|\langle a,b\rangle| \le |a|_2^2 + 2|a|_2|b|_2 \le d\delta^2 + 4d \delta F \le 5d \delta F,
    \]
    defined for $x,x' \in \IR^{d}$,
    \[
        g_j(x,x') := \frac{1}{dF^2}|f_{j}(x') - f_0(x')|_2^2\IW(x') - \frac{1}{dF^2}|f_{j}(x) - f_0(x)|_2^2\IW(x),
    \]
    and $S_n(\cdot)$ is from Lemma \ref{bernstein_mixing}. By Lemma \ref{bernstein_mixing}, there exists another process $S_n^{\circ}(\cdot)$ and some universal constant $c > 0$ such that
    \begin{equation}
        \IE|S_n(g_{j^{*}}) - S_n^{\circ}(g_{j^{*}})| \le \IE \sup_{g\in \sG}|S_n(g) - S_n^{\circ}(g)| \le cr q^{*}(\frac{H}{n})\frac{H}{n}.\label{convergenceproof_part1_eq2}
    \end{equation}
    Note that $\norm{g}_{2,\tilde\beta} \le r\norm{g}_{2,\beta}$.
    Put
    \[
        N(g) := q^{*}(\frac{H}{n})\sqrt{\frac{H}{n}} \vee 2\|g\|_{2,\beta}.
    \]
    We use the Cauchy Schwarz inequality and Lemma \ref{bernstein_mixing}, which yields a universal constant $c > 0$ such that
    \begin{eqnarray}
        \IE|S_n^{\circ}(g_{j^{*}})| &=& \IE\Big|\frac{S_n^{\circ}(g_{j^{*}})}{N(g_{j^{*}})}\cdot N(g_{j^{*}})\Big| \le \Big\|\frac{S_n^{\circ}(g_{j^{*}})}{N(g_{j^{*}})}\Big\|_2 \IE[ \|g\|_{2,\beta}^2\big|_{g=g_{j^{*}}}]^{1/2} + \IE\Big|\frac{S_n^{\circ}(g_{j^{*}})}{N(g_{j^{*}})}\Big|\cdot q^{*}(\frac{H}{n})\sqrt{\frac{H}{n}}\nonumber\\
        &\le& c\Big[r\sqrt{nH} \IE[ \|g\|_{2,\beta}^2\big|_{g=g_{j^{*}}}]^{1/2} + \sqrt{r}q^{*}(\frac{H}{n})H\Big].\label{convergenceproof_part1_eq3}
    \end{eqnarray}
    Insertion of \reff{convergenceproof_part1_eq2} and \reff{convergenceproof_part1_eq3} into \reff{convergenceproof_part1_eq1} yields
    \begin{equation}
        |\IE D(\hat f) - \IE \hat D_n(\hat f)| \le crF^2\Big[q^{*}(\frac{H}{n})\frac{H}{n} + \sqrt{\frac{H}{n}}\IE[ \|g\|_{2,\beta}^2\big|_{g=g_{j^{*}}}]^{1/2}\Big].\label{convergenceproof_part1_eq4}
    \end{equation}
    By Lemma \ref{lemma_mixing_norm},
    \[
        \|g\|_{2,\beta} \le C_{\beta} \|g\|_{\phi,2} \le C_{\beta}\varphi( \|g\|_1),
    \]
    where $\varphi(x)^2 = \phi^{-1}(x^{-1})^{-1}$ is concave. Thus by Jensen's inequality and due to the fact that $\varphi$ is concave (therefore  subadditive),
    \begin{eqnarray*}
        \IE[ \|g\|_{2,\beta}^2\big|_{g=g_{j^{*}}}]^{1/2} &\le& C_{\beta}\IE[\varphi(\|g\|_1\big|_{g= g_{j^{*}}})^2]^{1/2} \le C_{\beta}\varphi(\IE[ \|g\|_1 \big|_{g = g_{j^{*}}}]) \le C_{\beta} \varphi(F^{-2}\IE D(f_{j^{*}}))\\
        &\le& C_{\beta}[\varphi(2F^{-2}\delta^2) + \varphi(2F^{-2}\IE D(\hat f))].
    \end{eqnarray*}
    Insertion into \reff{convergenceproof_part1_eq4} yields \[
        \big|\IE D(\hat f) - \IE \hat D_n(\hat f)| \le cF^2r\Big[q^{*}(\frac{H}{n})\frac{H}{n} + C_{\beta}\sqrt{\frac{H}{n}}[\varphi(2F^{-2}\delta^2) + \varphi(2F^{-2}\IE D(\hat f))]\Big].
    \]
    By Lemma \ref{lemma_recursive_risk_general},
    \begin{eqnarray*}
        \IE D(\hat f) &\le& (1+\eta) \Big[\IE \hat D_n(\hat f) + crF^2 q^{*}(\frac{H}{n})\frac{H}{n} + crF^2C_{\beta}\sqrt{\frac{H}{n}}\varphi(2F^{-2}\delta^2)\Big]\\
        &&\quad\quad + \frac{\eta F^2}{2} (\varphi^{-1})^{*}\Big(2\frac{1+\eta}{\eta F^2}rC_{\beta}\sqrt{\frac{H}{n}}\Big),
    \end{eqnarray*}
    which yields the assertion.
\end{proof}


\begin{lem}\label{convergenceproof_part2}
    Suppose that Assumption \ref{ass_subgaussian}, \ref{ass_compatibility} hold. Let each $f \in \sF$ satisfy $\sup_{x\in \text{supp}(\IW)}|f(x)|_{\infty} \le F$. Then there exists an universal constant $c > 0$ such that for any $\delta \in (0,1)$,
    \begin{eqnarray*}
        &&\Big|\IE\Big[\frac{1}{nd}\sum_{i=r+1}^{n}\langle \varepsilon_i, \hat f(\IX_{i-1})\rangle \IW(\IX_{i-1})\Big]\Big|\\
        &\le& C_{\varepsilon}\delta +  c C_{\varepsilon}C_{\beta,sub}rF\Big[q^{*}(\frac{H}{n})\frac{H}{n} + \sqrt{\frac{H}{n}}(\IE[D(\hat f)]^{1/2} + \delta)\Big].
    \end{eqnarray*}
\end{lem}
\begin{proof}[Proof of Lemma \ref{convergenceproof_part2}]
    Let $(f_j)_{j=1,...,\sN_n}$ denote a $\delta$-covering of $\sF$ w.r.t. $\|\cdot\|_{\infty}$. Let $j^{*} \in \{1,...,\sN_n\}$ be such that $\|\hat f - f_{j^{*}}\|_{\infty} \le \delta$. Let $H = H(\delta) = \log N(\delta, \sF, \|\cdot\|_{\infty})$.  Since $\varepsilon_i$ is independent of $\IX_{i-1}$ and $\IE \varepsilon_i = 0$, we have
    \begin{equation}
         \Big|\IE\Big[\frac{1}{nd}\sum_{i=r+1}^{n}\langle \varepsilon_i, \hat f(\IX_{i-1})\rangle \IW(\IX_{i-1})\Big]\Big| \le \delta \cdot \underbrace{\frac{1}{nd}\sum_{i=r+1}^{n}\IE|\varepsilon_i|_1}_{\le \frac{1}{d}\sum_{k=1}^{d}\IE|\varepsilon_{1k}| \le C_{\varepsilon}} + \frac{F}{n}|\IE M_n(g_{j^{*}})|,\label{convergenceproof_part2_eq1}
    \end{equation}
    where $g_j(x) = \frac{1}{F}(f_{j}(x) - f_0(x))\IW(x)$ and $M_n$ is from Lemma \ref{bernstein_mixing_martingale}. By Lemma \ref{bernstein_mixing_martingale}, there exists a process $M_n^{\circ}(\cdot)$ and some universal constant $c > 0$ such that
    \begin{equation}
        \frac{1}{n}|\IE\{M_n(g_{j^{*}}) - M_n^{\circ}(g_{j^{*}})\}| \le \frac{1}{n}\IE \sup_{g\in \sG}|M_n(g) - M_n^{\circ}(g)| \le cr C_{\varepsilon} C_{\beta,sub}q^{*}(\frac{H}{n})\frac{H}{n}.\label{convergenceproof_part2_eq2}
    \end{equation}
    
    Define $N(g) := \|\frac{1}{\sqrt{d}}|g(\IX_r)|_2\|_2 \vee q^{*}(\frac{H}{n})\sqrt{\frac{H}{n}}$. Note that
    \begin{eqnarray*}
        \IE|M_n^{\circ}(g_{j^{*}})| &=& \IE\Big|\frac{M_n^{\circ}(g_{j^{*}})}{N(g_{j^{*}})}\cdot N(f_{j^{*}})\Big|\\
        &\le& \Big\|\frac{M_n^{\circ}(g_{j^{*}})}{N(g_{j^{*}})}\Big\|_2 \underbrace{\IE[ \|\frac{1}{\sqrt{d}}|g(\IX_r)|_2\|_2^2 \big|_{g = g_{j^{*}}}]^{1/2}}_{=F^{-1}\IE[D(f_{j^{*}})]^{1/2}} + \IE \Big|\frac{M_n^{\circ}(g_{j^{*}})}{N(g_{j^{*}})}\Big|\cdot q^{*}(\frac{H}{n})\sqrt{\frac{H}{n}}\\
        &\le& \IE\Big[\sup_{j=1,...,\sN_n}\Big|\frac{M_n^{\circ}(g_{j})}{N(g_{j})}\Big|^2\Big]^{1/2}\cdot F^{-1}(\IE[D(\hat f)]^{1/2} + \delta)\\
        &&\quad\quad + \IE\Big[\sup_{j=1,...,\sN_n}\Big|\frac{M_n^{\circ}(g_{j})}{N(g_{j})}\Big|\Big]\cdot q^{*}(\frac{H}{n})\sqrt{\frac{H}{n}}.
    \end{eqnarray*}
    By Lemma \ref{bernstein_mixing_martingale}, there exists some universal constant $c > 0$ such that
    \begin{equation}
        \IE |M_n^{\circ}(g_{j^{*}})| \le cC_{\varepsilon}\Big[\sqrt{nH}\cdot F^{-1}(\IE[D(\hat f)]^{1/2} + \delta) + q^{*}(\frac{H}{n})H\Big].\label{convergenceproof_part2_eq3}
    \end{equation}
    Insertion of \reff{convergenceproof_part2_eq2} and \reff{convergenceproof_part2_eq3} into \reff{convergenceproof_part2_eq1} yields
    \[
         \Big|\IE\Big[\frac{1}{nd}\sum_{i=1}^{n}\langle \varepsilon_i, \hat f(X_{i-1})\rangle \IW(X_{i-1})\Big| \le C_{\varepsilon}\delta + 2cC_{\varepsilon}\Big[F r C_{\beta,sub} q^{*}(\frac{H}{n})\frac{H}{n} + \sqrt{\frac{H}{n}}(\IE[D(\hat f)]^{1/2} + \delta)\Big].
    \]
\end{proof}

\section{Approximation error} \label{sec_approx_error}

We consider a network $\tilde f_0$ approximating the true regression function $f_0$. The network $\tilde f_0$ is assumed to have the form
\begin{equation} \label{dec_enc_net}
    \tilde f_0 = \tilde f_{dec} \circ \tilde f_{enc} : \IR^{dr} \to \IR^d
\end{equation} where $\tilde f_{enc} : \IR^{dr} \to \IR^{\tilde d}$ and $\tilde f_{dec}:\IR^{\tilde d} \to \IR^d$, $\tilde d \in \{1,...,d\}$, and $\tilde f_{enc}$ has the additional network structure
\[
    \tilde f_{enc} = \tilde g_{enc,1} \circ \tilde g_{enc,0},
\] where $\tilde g_{enc,0}: \IR^{dr} \to \IR^{D}$, $D \in \IN$ depending on at most $t_{enc,0} \in \{1,...,dr\}$ arguments in each component, and $\tilde g_{enc,1}:\IR^{D} \to \IR^{\tilde d}$ depending on at most $t_{enc,1} \in \{1,...,D\}$ arguments in each component. We denote by
\begin{eqnarray*}
    \sF(L,p,s) &:=& \big\{f:\IR^{p_0} \to \IR^{p_{L+1}} \text{ is of the form \reff{form_neuralnetwork}}: \\
    && \qquad \qquad \max_{k=0,...,L}|W^{(j)}|_{\infty}\vee |v^{(j)}|_{\infty} \le 1,  \sum_{j=0}^{L}|W^{(j)}|_0 + |v^{(j)}|_0 \le s\}
\end{eqnarray*}
the set of all networks where we explicitly do not ask for the presence of an encoder-decoder structure and an intermediate hidden layer at position $L_1$.

Now, let $\tb:=(t_{dec},t_{enc,0},t_{enc,1})$ and $\betab:=(\beta_{dec},\beta_{enc,0},\beta_{enc,1})$ where $t_{dec} \in \{1,...,\tilde d\}$ $t_{enc,1} \in \{1,...,D\}$ and $t_{enc,0} \in \{1,...,dr\}$.

\begin{thm} \label{approx_error}
    Consider the $d$-dimensional time series that follow the recursion \reff{model_time_evolution} and Assumptions \ref{ass_autoencoder}, \ref{ass_smoothness}. Let $N \in \{1,...,n\}$. Suppose that the parameters of  $\sF(L,L_1,p,s,F,\mathrm{Lip})$ satisfy 
    \begin{enumerate}
        \item $K \le F$,
        \item $\sum_{i \in \{enc,0;enc,1\}} \log_2(4(t_i \vee \beta_i))\log_2(n) \le L_1$ and \\
        $L_1  +  \log_2(4(t_{dec} \vee \beta_{dec}))\log_2(n) \le L$,
        \item $N \lsim \min_{i}\{p_i\}$,
        \item $N \log_2(n) \lsim s$,
        \item $\mathrm{Lip} \gsim 1$.
    \end{enumerate}
    Then,
    \[
        \inf_{f^\ast \in \sF(L,L_1,p,s,F,\mathrm{Lip})} \norm{f^* - f_0}_{\infty}^2 \le  C \max_{k \in \{dec;enc,0;enc,1\}} \big\{\frac{N}{n} + N^{-\frac{2\beta_k}{t_k}}\big\}
    \]
    for a large enough constant $C$ that only depends on $\tilde d, d,\tb, \betab$.
\end{thm}

The proof can be found in the Appendix.

\bibliographystyle{plain}
\bibliography{reference}

\newpage

\begin{center}
{\Large \bf Supplementary Material}\\
\end{center}

\section{Appendix}
\label{sec_appendix}

In this section, we provide some details of the proofs for the oracle inequality under mixing.

\subsection{Results for mixing time series}
\label{sec_appendix_mixing}

\subsubsection{Variance bound for mixing}

\begin{proof}[Proofs of Lemma \ref{lemma_mixing_norm}]
    We only have to prove the last inequality. Note that for any $c > 0$, due to monotonicity of $x \mapsto \frac{\phi(x)}{x}$,
    \[
        \IE \phi\Big( \big(\frac{|h(X_1)|}{c}\big)^2\Big) = \IE \Big[\frac{\phi\Big( \big(\frac{|h(X_1)|}{c}\big)^2\Big)}{|h(X_1)|}\cdot |h(X_1)|\Big] \le \frac{\phi\Big( \big(\frac{\|h\|_{\infty}}{c}\big)^2\Big)}{\|h\|_{\infty}}\cdot \|h\|_1.
    \]
    This upper bound attains the value $1$ for
    \[
        c = \|h\|_{\infty}\cdot \phi^{-1}(\frac{\|h\|_{\infty}}{\|h\|_1})^{-1/2},
    \]
    which shows
    \begin{equation}
        \|h\|_{\phi,2} \le \|h\|_{\infty}\cdot \phi^{-1}(\frac{\|h\|_{\infty}}{\|h\|_1})^{-1/2}.\label{lemma_mixing_norm_res3}
    \end{equation}
    The result \reff{lemma_mixing_norm_res1} now follows from $\|h\|_{\infty} \le 1$.
\end{proof}


During the proofs for the oracle inequalities under mixing, there will occur two quantities:
\begin{equation}
    (\varphi^{-1})^{*}(\sqrt{x}) \quad\quad \text{ and }\quad\quad q^{*}(x)x,\label{definition_mixing_quantitiestoupperbound}
\end{equation}
where $q^{*}(x) = q^{*,mix}(x) = \min\{q\in \IN: \beta^{mix}(q) \le qx\}$. For $x$, we have to plug in a specific rate of the form $\frac{H}{n}$. It is therefore of interest to upper bound both quantities in \reff{definition_mixing_quantitiestoupperbound} by one common quantity.

Recall the definitions $\psi(x) := \phi^{*}(x)x$ and $\Lambda(x) := \lceil \psi^{-1}(x^{-1})\rceil x$ from  \reff{definition_lambda_mixing}.

In Lemma \ref{lemma_varphi}, we show that $(\varphi^{-1})^{*}(\sqrt{x})$ is upper bounded by a constant times $\Lambda(x)$. Lemma \ref{lemma_varphi2} shows that $q^{*}(x)x$ is upper bounded by a constant times $\Lambda(x)$. Thus, $\Lambda(x)$ serves as a common upper bound for both quantities in \reff{definition_mixing_quantitiestoupperbound}.

\subsubsection{Unification of \reff{definition_mixing_quantitiestoupperbound}}

\begin{lem}\label{lemma_varphi}
    Let Assumption \ref{ass_compatibility} hold. Then $\varphi(x) = \phi^{-1}(\frac{1}{x})^{-1/2}$ and $\psi(x) = \phi^{*}(x)x$ satisfy:
    \begin{enumerate}
        \item for any $C \ge 1$, $(\varphi^{-1})^{*}(Cx) \le C^2(\varphi^{-1})^{*}(x)$,
        \item $\varphi^2, \varphi$ are concave and thus subadditive,
        \item $(\varphi^{-1})^{*}(\sqrt{x}) \le (4c_0)^2\psi^{-1}(\frac{1}{x})x \le (4c_0)^2 \Lambda(x)$.
    \end{enumerate}
\end{lem}
\begin{proof}[Proof of Lemma \ref{lemma_varphi}]
    \begin{enumerate}
        \item Since $y \mapsto \frac{\phi(y)}{y}$ is increasing, $\frac{\phi(y)}{y} \le \frac{\phi(C^2y)}{C^2y}$. Thus $\frac{1}{\phi(y)} \ge \frac{C^2}{\phi(C^2y)}$. We obtain
        \begin{eqnarray*}
            (\varphi^{-1})^{*}(Cx) &=& \sup_{z > 0}\{Cxz - \phi(\frac{1}{z^2})^{-1}\}\\
            &=& C^2 \sup_{z > 0}\{x\frac{z}{C} - \frac{\frac{1}{C^2}}{\phi(\frac{1}{z^2})}\}\\
            &\le& C^2 \sup_{z > 0}\{x\frac{z}{C} - \frac{1}{\phi(\frac{C^2}{z^2})}\}\\
            &\overset{u:= \frac{z}{C}}{=}& C^2\sup_{u > 0}\{xu - \phi(\frac{1}{u^2})^{-1}\} = C^2 (\varphi^{-1})^{*}(x).
        \end{eqnarray*}
        \item We have
        \begin{equation}
            (\varphi^{2})^{-1}(y) = \phi(\frac{1}{y})^{-1}.\label{lemma_varphi_proof_eq0}
        \end{equation}
        By assumption, $y \mapsto \frac{y}{\phi(y)}$ is convex. Since $x \mapsto f(\frac{1}{x})$ is convex on $(0,\infty)$ if and only if $x\cdot f(x)$ is convex on $(0,\infty)$ (cf. \cite{convex_equivalent}, page 1), we obtain that $y \mapsto \phi(\frac{1}{y})^{-1}$ is convex. By \reff{lemma_varphi_proof_eq0}, its inverse $\varphi^2$ is concave. By concavity of $\sqrt{\cdot}$, $\varphi$ is concave. Since concavity implies subadditivity, the claim follows. 
        \item \emph{First Claim:} If $f:[0,\infty) \to \IR$ is a convex function with $f(0)=0$ and $w_0 \in [0,\infty)$ is such that $xw_0 - f(w_0) \le 0$, then
        \[
            f^{*}(x) \le w_0f'(w_0).
        \]
        \emph{Proof:} $F:[0,\infty)\to\IR, F(w) = xw-f(w)$ is concave with $F(0) = 0$, $F(w_0) \le 0$. Thus, $F$ attains its global maximum in $[0,w_0]$. Since $F$ is concave, the tangent $t(w) := F'(w_0)(w-w_0) + F(w_0)$ at $w_0$ satisfies
        \begin{eqnarray*}
            f^{*}(x) &=& \sup_{w > 0}F(w) \le \sup_{w > 0}t(w) = t(0) = -F'(w_0)w_0 + F(w_0)\\
            &=& -(x - f'(w_0))w_0 + (xw_0 - f(w_0))\\
            &=& f'(w_0)w_0 - f(w_0) \le f'(w_0)w_0.
        \end{eqnarray*}
        This proves the claim.
        
        \emph{Second Claim:} If $f:[0,\infty) \to \IR$ is convex with $f(0)=0$, then $f(x) \le f'(x) x$ for all $x > 0$.
        
        \emph{Proof:} Since $f$ is convex, the tangent $t(y) = f'(x)(y-x) + f(x)$ at $x$ satisfies $t(0) \le f(0) = 0$, which gives the result.
        
        Let $y_0(x) := (\phi')^{-1}(c_0x)$. Then,
        \[
            x = \frac{1}{c_0}\phi'(y_0(x)) \le \frac{\phi(y_0(x))}{y_0(x)}.
        \]
        Application of the first claim to $\phi^{*}$ and $y_0(x)$ yields
        \[
            \psi(x) = \phi^{*}(x)x \le x y_0(x) = c_0 x^2 (\phi')^{-1}(c_0 x) =: g(x).
        \]
        Since $g,\psi:[0,\infty) \to [0,\infty)$ are strictly increasing, we conclude that for any $y \in [0,\infty)$,
        \[
             g^{-1}(y) = \psi^{-1}(\psi(g^{-1}(y))) \le \psi^{-1}(g(g^{-1}(y))) = \psi^{-1}(y).
        \]
        Especially, we obtain for any $x > 0$,
        \begin{equation}
            \psi^{-1}(\frac{1}{x})x \ge g^{-1}(\frac{1}{x})x.\label{lemma_varphi_proof_eq1}
        \end{equation}
        
        As in (ii), we obtain that $y \mapsto \phi(\frac{1}{y})^{-1}$ is convex. Additionally, it is increasing, thus $y \mapsto \phi(\frac{1}{y^2})^{-1}$ is convex. Moreover, for all $z > 0$,
        \begin{equation}
            z\cdot \partial_z(\phi(\frac{1}{z^2})^{-1}) = \frac{2\phi'(\frac{1}{z^2})\frac{1}{z^2}}{\phi(\frac{1}{z^2})^2} \le 2c_0 \phi(\frac{1}{z^2})^{-1}.\label{lemma_varphi_proof_eq2}
        \end{equation}
        Choose $z_0(x) > 0$ such that
        \begin{equation}
            \sqrt{x}z_0(x) - \phi(\frac{1}{z_0(x)^2})^{-1} = 0.\label{lemma_varphi_proof_eq3}
        \end{equation}
        We obtain from the first claim and \reff{lemma_varphi_proof_eq2} that
        \begin{eqnarray}
            (\varphi^{-1})^{*}(\sqrt{x}) &=& \sup_{y > 0}\{xy - \phi(\frac{1}{y^2})^{-1}\} \le z_0(x) \partial_z(\phi(\frac{1}{z^2})^{-1})\big|_{z=z_0(x)} \le 2c_0 \phi(\frac{1}{z_0(x)^2})^{-1}\nonumber\\
            &=& 2c_0\cdot \sqrt{x} z_0(x).\label{lemma_varphi_proof_eq5}
        \end{eqnarray}
        By the second claim, we obtain
        \begin{equation}
            \sqrt{x} = \frac{\phi(\frac{1}{z_0(x)^2})^{-1}}{z_0(x)} \le \partial_z(\phi(\frac{1}{z^2})^{-1})\big|_{z=z_0(x)} = \frac{2\phi'(\frac{1}{z_0(x)^2})}{\phi(\frac{1}{z_0(x)^2})^2 z_0(x)^3}.\label{lemma_varphi_proof_eq4}
        \end{equation}
        By \reff{lemma_varphi_proof_eq3}  and \reff{lemma_varphi_proof_eq4},
        \[
            \frac{z_0(x)}{\sqrt{x}} = \sqrt{x}z_0(x)^3 \phi(\frac{1}{z_0(x)^2})^2 \le 2 \phi'(\frac{1}{z_0(x)^2}).
        \]
        Thus,
        \[
            g(\frac{1}{c_0}\frac{z_0(x)}{2\sqrt{x}}) = \frac{1}{c_0}(\frac{z_0(x)}{2\sqrt{x}})^2\cdot (\phi')^{-1}\Big(\frac{z_0(x)}{2\sqrt{x}}\Big) \le \frac{1}{c_0}(\frac{z_0(x)}{2\sqrt{x}})^2\cdot \frac{1}{z_0(x)^2} = \frac{1}{4c_0}\frac{1}{x}.
        \]
        Since $g$ is increasing, 
        \[
            g^{-1}(\frac{1}{4c_0}\frac{1}{x}) \ge \frac{1}{c_0}\frac{z_0(x)}{2\sqrt{x}}, \quad\text{ and thus }\quad g^{-1}(\frac{1}{4c_0}\frac{1}{x})x \ge \frac{1}{2c_0}\sqrt{x} z_0(x).
        \]
        By \reff{lemma_varphi_proof_eq1} and \reff{lemma_varphi_proof_eq5}, $g$ is increasing. Due to the fact that $c_0 \ge 1$ (see the second claim), we have
        \[
            (\varphi^{-1})^{*}(\sqrt{x}) \le 2c_0\cdot \sqrt{x}z_0(x) \le (4c_0)^2 g^{-1}(\frac{1}{4c_0}\frac{1}{x})x \le (4c_0)^2 g^{-1}(\frac{1}{x})x \le (4c_0)^2 \psi^{-1}(\frac{1}{x})x.
        \]
    \end{enumerate}
\end{proof}

\begin{lem}\label{lemma_varphi2}
    Let Assumption \ref{ass_compatibility} hold. Then,
    \[
        q^{*}(x)x \le 2C \Lambda(x)
    \]
    where $C \le \sum_{k=0}^{\infty}\{\phi^{*}(k+1) - \phi^{*}(k)\}\beta(k)$.
\end{lem}
\begin{proof}[Proof of Lemma \ref{lemma_varphi2}] Since $\beta^{-1}(u) = \sum_{i=0}^{\infty}\Ii_{\{u < \beta(i)\}}$ and $\beta(0)=1$, we have
    \begin{eqnarray}
        \int_0^{1} \phi^{*}(\beta^{-1}(u)) du &=& \sum_{i=0}^{\infty}\int_{\beta(i+1)}^{\beta(i)}\phi^{*}(i+1)du = \sum_{i=0}^{\infty}(\beta(i) - \beta(i+1))\phi^{*}(i+1)\nonumber\\
        &=& \sum_{i=0}^{\infty}\sum_{k=0}^{i}\{\phi^{*}(k+1) - \phi^{*}(k)\}(\beta(i) - \beta(i+1))\nonumber\\
        &=& \sum_{k=0}^{\infty}\{\phi^{*}(k+1) - \phi^{*}(k)\}\cdot \sum_{i=k}^{\infty}(\beta(i) - \beta(i+1))\nonumber\\
        &=& \sum_{k=0}^{\infty}\{\phi^{*}(k+1) - \phi^{*}(k)\}\beta(k) < \infty.\label{lemma_varphi2_proof_eq1}
    \end{eqnarray}
    Let $Z$ be a nonnegative $\IN_0$-valued random variable with $\IP(Z \ge k) = \beta(k)$. Then, $\IP(Z=k) = \beta(k) - \beta(k+1)$, so \reff{lemma_varphi2_proof_eq1} shows that $C := \IE \phi^{*}(Z) \le \IE \phi^{*}(Z+1) < \infty$. Markov's inequality implies
    \[
        \beta(k) = \IP(Z\ge k) \le \frac{\IE \phi^{*}(Z)}{\phi^{*}(k)},
    \]
    that is, $\beta(k) \phi^{*}(k) \le C$. We then obtain
\begin{eqnarray*}
    q^{*}(x) &=& \min\{q\in\IN: \frac{\beta(q)}{q} \le x\} \le \min\{q\in\IN: \frac{C}{\phi^{*}(q)q} \le x\}\\
    &=& \min\{q\in\IN: Cx^{-1} \le \phi^{*}(q)q\} \le  \lceil\psi^{-1}(Cx^{-1})\rceil,
\end{eqnarray*}
whence
\begin{equation}
    q^{*}(x)x \le \lceil\psi^{-1}(Cx^{-1})\rceil x.\label{lemma_varphi2_proof_eq0}
\end{equation}
Since $\phi^{*}$ is increasing, it holds that $\psi(Cx) = Cx \phi^{*}(Cx) \ge C x \phi^{*}(x) = C \psi(x)$. This implies for any $y > 0$ and $z := \psi^{-1}(y)$,
\[
    \psi^{-1}(Cy) = \psi^{-1}(C\psi(z)) \le \psi^{-1}(\psi(Cz)) \le Cz = C \psi^{-1}(y).
\]
From \reff{lemma_varphi2_proof_eq0} we obtain
\[
    q^{*}(x)x \le \lceil C \psi^{-1}(x^{-1})\rceil x \le \lceil C \lceil\psi^{-1}(x^{-1})\rceil\rceil x \le 2C\lceil\psi^{-1}(x^{-1})\rceil x .
\]
\end{proof}

The following proof shows the announced special forms for $\Lambda = \Lambda^{mix}$ in Lemma \ref{lemma_varphi_specialcase}.

\begin{proof}[Proof of Lemma \ref{lemma_varphi_specialcase}]
    \begin{enumerate}
        \item Let $\phi(x) = x^{\frac{\alpha}{\alpha-1}}$ with $\alpha > 1$. Then obviously, Assumption \ref{ass_compatibility} (i),(ii) are fulfilled with $c_0 = \frac{\alpha}{\alpha-1}$. Furthermore,
        \[
            \phi^{*}(x) = \sup_{y > 0}\{xy - \phi(y)\} = C_{\alpha}x^{\alpha},\quad\quad C_{\alpha} := (1 - \frac{1}{\alpha})^{\alpha}\cdot \frac{1}{\alpha-1},
        \]
        which implies $\phi^{*}(k+1) - \phi^{*}(k) = O(k^{\alpha-1})$ and thus proves $\sum_{k=0}^{\infty}(\phi^{*}(k+1)-\phi^{*}(k))\beta(k) = O(\sum_{k=0}^{\infty}k^{\alpha-1}\beta(k)) < \infty$.
        
        In this case, we have $\psi(x) = \phi^{*}(x)x = C_{\alpha} x^{\alpha+1}$, $\psi^{-1}(x) = (C_{\alpha}^{-1} x)^{\frac{1}{\alpha+1}}$ and
        \[
            \Lambda(x) = \lceil\psi^{-1}(x^{-1})\rceil x = \lceil C_{\alpha}^{-\frac{1}{\alpha+1}} x^{-\frac{1}{\alpha+1}}\rceil x.
        \]
        For $x > 1$, the above is bounded by $2C_{\alpha}^{-\frac{1}{\alpha+1}}x$, for $x < 1$, the above is bounded by $2C_{\alpha}^{-\frac{1}{\alpha+1}} x^{-\frac{1}{\alpha}}x$. This yields the result with $c_{\alpha} = 2C_{\alpha}^{-\frac{1}{\alpha+1}}$.
        \item Let $a := \frac{\rho+1}{2\rho} > 1$. Then $a\rho < 1$. Define $\phi(x) = x \frac{\log(x+1)}{\log(a)}$. Obviously, Assumption \ref{ass_compatibility} (i),(ii) are fulfilled with $c_0 = 2$. Furthermore, by the first claim in the proof of Lemma \ref{lemma_varphi} applied to $w_0 := a^{x}-1$, 
        \[
            \phi^{*}(x) \le w_0\cdot \phi'(w_0) \le 2 \phi(w_0)  = 2x \{a^{x}-1\}.
        \]
        On the other hand, for $x \ge 1$, $0 \le w_1 := a^{x-1}-1$, thus
        \begin{equation}
            \phi^{*}(x) \ge x w_1 - \phi(w_1) = a^{x-1}-1.\label{lemma_varphi_specialcase_eq0}
        \end{equation}
        We obtain
        \[
            \sum_{k=1}^{\infty}(\phi^{*}(k+1) - \phi^{*}(k))\beta(k) \le \sum_{k=1}^{\infty}\big(2(k+1)(a^{k+1}-1) - a^{k-1} + 1)\kappa\rho^k = O(\sum_{k=1}^{\infty}k (\rho a)^{k}) < \infty.
        \]
        To upper bound the rate function, we use \reff{lemma_varphi_specialcase_eq0} to obtain
        \[
            \psi(x) = \phi^{*}(x)x \ge x (a^{x-1}-1) =: g(x)
        \]
        where $g:[1,\infty) \to [0,\infty)$ is bijective. Thus, for any $y \ge 0$, we obtain
        \[
            \psi^{-1}(y) = \psi^{-1}(g(g^{-1}(y))) \le \psi^{-1}(\psi(g^{-1}(y))) = g^{-1}(y).
        \]
        We conclude that $\Lambda(x) \le \lceil g^{-1}(\frac{1}{x})\rceil x$. Here,
        \[
            g(\frac{\log(2a(y\vee e))}{\log(a)}) = \frac{\log(2a(y\vee e))}{\log(a)}(\frac{2a(y \vee e)}{a}-1) \ge 2(y \vee e) - 1 \ge y \vee e.
        \]
        Thus for $y \ge e$,
        \[
            g^{-1}(y) \le \frac{\log(2ay)}{\log(a)} \le 2+ \frac{\log(y)}{\log(a)}.
        \]
        We obtain for $x \le e^{-1}$,
        \begin{equation}
            \Lambda(x) \le 2\Big(2 + \frac{\log(x^{-1})}{\log(a)}\Big)x.\label{lemma_varphi_specialcase_eq1}
        \end{equation}
        Note that for $y \le e$, $c := 1 + \frac{e}{a-1}$ satisfies  $\phi^{*}(c) = \sup_{y > 0}\{cy - \phi(y)\} \overset{y=a-1}{=} (a-1)[c - \frac{\log((a-1)+1)}{\log(a)}] = e$, and $\psi(c) = \phi^{*}(c)c \ge e c \ge e \ge y$. Thus
        \[
            \psi^{-1}(y) \le c,
        \]
        so for $x \ge \frac{1}{e}$, \begin{equation}
           \Lambda(x) = \lceil \psi^{-1}(x^{-1})\rceil x \le 2c x. \label{lemma_varphi_specialcase_eq2}
        \end{equation} A combination of \reff{lemma_varphi_specialcase_eq1} and \reff{lemma_varphi_specialcase_eq2} gives the result.
    \end{enumerate}
\end{proof}

\subsubsection{Proofs of maximal inequalities under mixing}

In this section, we prove, as announced, maximal inequalities under mixing. To do so, we use techniques and decomposition ideas from \cite{Dedeck02}, \cite{rio1995} and \cite{liebscher1996}.

\begin{proof}[Proof of Lemma \ref{bernstein_mixing}] During the proof, let $q \in \{1,...,n\}$ be arbitrary. Later we will choose $q = q^{*}(\frac{H}{n}) \le n$. Note that 
\begin{eqnarray*}
    \beta^\IX(k) &:=& \beta(\sigma(\IX_{i-1}:i \le 0), \sigma (\IX_{i-1}:i \ge k)) \\
    &=& \beta(\sigma(X_{i-1}:i \le 0), \sigma (X_{i-1}:i \ge k - r + 1)) = \beta^X(k-r+1) \le \beta^X(k-r).
\end{eqnarray*} We define $\tilde \beta(k) := \beta^X((k-r)\vee 0)$. Now, following \cite{Dedeck02}, there exist random variables $\IX_{i-1}^{\circ}$ with the following properties:
\begin{itemize}
    \item for all $i \ge 0$, $U_{i-1}^{\circ} = (\IX_{(i-1)q+1}^{\circ},...,\IX_{(i-1)q+q}^{\circ})$ and $U_{i-1} = (\IX_{(i-1)q+1},...,\IX_{(i-1)q+q})$ have the same distribution,
    \item $(U_{2(i-1)}^{\circ})_{i \ge 1}$ and $(U_{2i}^{\circ})_{i \ge 1}$ are i.i.d.,
    \item for all $i \ge 1$, $\IP(\IX_{i-1} \not=\IX_{i-1}^{\circ}) \le \IP(U_{i-1} \not= U_{i-1}^\circ) \le \tilde \beta(q)$.
\end{itemize}
Put
\[
    S_n^{\circ}(g) := \sum_{i=r+1}^{n}\{g(\IX_{i-1}^{\circ}) - \IE g(\IX_{i-1}^{\circ})\}.
\]
Then,
\begin{equation}
    |S_n(g) - S_n^{\circ}(g)| \le 2\|g\|_{\infty}\Big|\sum_{i=r+1}^{n}(\Ii_{\{\IX_{i-1} \not= \IX_{i-1}^{\circ}\}} + \IP(\IX_{i-1} \not= \IX_{i-1}^{\circ}))\Big|.\label{proof_bernstein_mixing_eqminus1}
\end{equation}
We now proceed with the proof of the announced inequalities. First, we have
\[
    \IE \sup_{g\in \sG}|S_n(g) - S_n^{\circ}(g)| \le 4\sum_{i=r+1}^{n}\IP(\IX_{i-1} \not= \IX_{i-1}^{\circ}) = 4n\tilde \beta(q).
\]
Let $\tilde q^{*}(x) = \min\{q\in\IN: \tilde \beta(q) \le qx\}$.
For $q = \tilde q^{*}(\frac{H}{n}) \le n$, we obtain
\[
    \IE \sup_{g\in \sG}|S_n(g) - S_n^{\circ}(g)| \le 4n \tilde \beta(q^{*}(\frac{H}{n})) \le 4n \tilde q^{*}(\frac{H}{n})\frac{H}{n}.
\]
Now, let $\tilde q := q^\ast(x) + r$. Then,
\begin{eqnarray*}
    \beta(\tilde q - r) = \beta(q^\ast(x)) \le q^\ast(x)x = (\tilde q - r)x \le \tilde q x.
\end{eqnarray*}
This yields $\tilde q ^\ast (x) \le \tilde q = q^\ast(x) + r \le rq^\ast(x)$. Finally,
\[
    \IE \sup_{g\in \sG}|S_n(g) - S_n^{\circ}(g)| \le 4n \tilde q^{*}(\frac{H}{n})\frac{H}{n} \le 4n r q^{*}(\frac{H}{n})\frac{H}{n}
\]
which proves \reff{bernstein_mixing_res1}.

We now show \reff{bernstein_mixing_res2} and \reff{bernstein_mixing_res3}. It holds that
\[
    S_n^{\circ}(g) = \sum_{k=r+1, k \text{ even}}^{\lfloor \frac{n}{q}\rfloor + 1}Y_k^{\circ}(g) + \sum_{k=r+1, k \text{ odd}}^{\lfloor \frac{n}{q}\rfloor + 1}Y_k^{\circ}(g)
\]
where
\[
    Y_k^{\circ}(g) := \sum_{i=(k-1)q+1}^{kq \wedge n}\{g(\IX_{i-1}^{\circ}) - \IE g(\IX_{i-1}^{\circ})\}.
\]
Furthermore, $(Y_k)_{k \text{ even}}$ and $(Y_k)_{k \text{ odd}}$ are independent with 
\[
    \|\frac{Y_k^{\circ}(g)}{N(g)}\|_{\infty} \le 2q N(g)^{-1}, \quad\quad \|\frac{Y_k^{\circ}(g)}{N(g)}\|_2^2 \le \frac{1}{N(g)^2}\Var(\sum_{i=(k-1)q+1}^{kq\wedge n}g(\IX_{i-1})) \le 4q\frac{\|g\|_{2,\tilde \beta}^2}{N(g)^2}.
\]

Next,
\begin{eqnarray*}
    \norm{g}_{2,\tilde\beta} &=& \int_{0}^1 \tilde{\beta}^{-1}(u)Q_g(u)du = \sum_{i=1}^\infty \int_{\tilde \beta (i+1)}^{\tilde \beta(i)} i Q_g(u) du \\
    &=& \sum_{i=1}^\infty \int_{\beta ((i-r+1) \vee 0)}^{\beta ((i-r) \vee 0))} i Q_g(u) du \\
    &=& \sum_{i=r+1}^\infty  \int_{\beta ((i-r+1) \vee 0)}^{\beta ((i-r) \vee 0))} i Q_g(u) du \\
    &=& \sum_{j=1}^\infty  \int_{\beta (j+1)}^{\beta (j)} (j+r) Q_g(u) du \\
    &\le& r \sum_{j=1}^\infty  \int_{\beta (j+1)}^{\beta (j)} j Q_g(u) du \\
    &=& r \norm{g}_{2,\beta}
\end{eqnarray*}

Hence,
\[
    \|\frac{Y_k^{\circ}(g)}{N(g)}\|_2^2 \le 4q\frac{\|g\|_{2,\tilde \beta}^2}{N(g)^2} \le 4qr\frac{\|g\|_{2,\beta}^2}{N(g)^2} \le 4qr.
\]

We obtain by Bernstein's inequality that
\begin{eqnarray}
    &&\IP\Big(\big|\frac{S_n^{\circ}(g)}{N(g)}\big| > x\Big)\nonumber\\
    &\le& \IP\Big(\big|\sum_{k=r+1,k \text{ even}}^{\lfloor \frac{n}{q}\rfloor + 1}\frac{Y_k^{\circ}(g)}{N(g)}\big| > \frac{x}{2}\Big) + \IP\Big(\big|\sum_{k=r+1,k \text{ odd}}^{\lfloor \frac{n}{q}\rfloor + 1}\frac{Y_k^{\circ}(g)}{N(g)}\big| > \frac{x}{2}\Big)\nonumber\\
    &\le& 4\exp\Big(-\frac{1}{2}\frac{(x/2)^2}{4nr + 2qN(g)^{-1} x/2}\Big)\nonumber\\
    &\le& 4 \exp\Big(-\frac{1}{32}\frac{x^2}{nr + q N(g)^{-1}x}\Big).\label{proof_bernstein_mixing_eq2}
\end{eqnarray}

\begin{enumerate}
    \item
    Using standard arguments (cf. \cite{Vaart98}, Lemma 19.35), we obtain from   \reff{proof_bernstein_mixing_eq2} that there exists a universal constant $c > 0$ such that
    \[
        \IE \sup_{g\in \sG}|\frac{S_n^{\circ}(g)}{N(g)}| \le c \big[\sqrt{nrH} + qN(g)^{-1}H\big].
    \]
    For $q = q^{*}(\frac{H}{n}) \le n$, we obtain
    \[
        \IE \sup_{g\in \sG}|\frac{S_n^{\circ}(g)}{N(g)}| \le c \big[\sqrt{nrH} + q^{*}(\frac{H}{n})N(g)^{-1}H\big] \le 2c \sqrt{nrH}.
    \]
    This shows \reff{bernstein_mixing_res2}.
    \item Here, we use
\begin{eqnarray*}
    \IE\big[\sup_{g}|\frac{S_n^{\circ}(g)}{N(g)}|^2\big] = \int_0^{\infty}\IP\Big(\sup_g |\frac{S_n^{\circ}}{N(g)}| > \sqrt{t}\Big) dt.
\end{eqnarray*}
Put $a := q^{*}(\frac{H}{n})\sqrt{\frac{H}{n}}$. Choose $G := 64\frac{nr}{q}a$. Then for $t \ge G^2$, $\frac{q}{a}\sqrt{t} \ge nr$. With \reff{proof_bernstein_mixing_eq2} and  $\int_{b^2}^{\infty}\exp(-b_2\sqrt{t}) dt = \int_b^{\infty}2s \exp(-b_2 s) ds = 2(b_2b+1)b_2^{-2}\exp(-b_2b)$, we obtain
\begin{eqnarray}
    &&\int_0^{\infty}\IP\Big(\sup_g |\frac{S_n^{\circ}(g)}{N(g)}| > \sqrt{t}\Big) dt\nonumber\\
    &=&G^2 + \int_{G^2}^{\infty}\IP\Big(\sup_g |\frac{S_n^{\circ}(g)}{N(g)}| > \sqrt{t}\Big) dt\nonumber\\
    &\le& G^2 + 4|\sG|\int_{G^2}^{\infty}\exp\Big(-\frac{1}{32}\frac{t}{nr + qN(g)^{-1}\sqrt{t}}\Big) dt\nonumber\\
    &\le& G^2 + 4|\sG|\int_{G^2}^{\infty}\exp\Big(-\frac{1}{32}\frac{t}{nr + q a^{-1}\sqrt{t}}\Big) dt\nonumber\\
    &\le& G^2 + 4|\sG| \int_{G^2}^\infty \exp(-\frac{1}{64}\frac{\sqrt{t}}{qa^{-1}}) dt\nonumber\\
    &\le& G^2 + 8|\sG|\Big(\frac{Ga}{q} + 1\Big)(64qa^{-1})^2 \exp(-\frac{1}{64}\frac{Ga}{q})\nonumber\\
    &\le& 2^{15}\Big[\Big(\frac{nar}{q}\Big)^2 + |\sG|\cdot \exp\Big(-nr(\frac{a}{q})^2\Big)\cdot (nr + (qa^{-1})^2)\Big].\label{proof_bernstein_mixing_eq7}
\end{eqnarray}
We receive
\begin{eqnarray*}
    \IE\big[\sup_{g}|\frac{S_n^{\circ}(g)}{N(g)}|^2\big] &\le& 2^{16}\Big[\Big(\frac{nar}{q}\Big)^2 + |\sG|\cdot \exp\Big(-nr(\frac{a}{q})^2\Big)\cdot (nr + (qa^{-1})^2)\Big].
\end{eqnarray*}
With $q = q^{*}(\frac{H}{n})$, the latter is upper bounded by
\begin{eqnarray*}
    2^{16}\Big[nr^2H +  nr + \frac{n}{H}\Big] = 2^{18}n r^2 H,
\end{eqnarray*}
proving \reff{bernstein_mixing_res3}.
\end{enumerate}
\end{proof}

\begin{proof}[Proof of Lemma \ref{bernstein_mixing_martingale}]
   
    
    Note that $(\varepsilon_i,\IX_{i-1})$ is still $\beta$-mixing with coefficients $\tilde \beta(k) := \beta(k-r)$ . This is due to the following argument: The model equation yields $X_i = f_0(\IX_{i-1}) + \varepsilon_i$, that is, $\varepsilon_i = X_i - f_0(\IX_{i-1})$. Thus, the generated sigma fields fulfill
    \[
        \sigma( (\varepsilon_i, \IX_{i-1}): i \le 0) = \sigma(X_{i}:i \le 0)
    \]
    and
    \[
        \sigma( (\varepsilon_i, \IX_{i-1}):i \ge k) = \sigma(X_{i-1}:i\ge k-r+1) = \sigma(X_i: i \ge k-r).
    \]
    Similar to the proof of Lemma \ref{bernstein_mixing}, for each $q \in \{1,...,n\}$ we can construct
    define coupled versions $(\varepsilon_i^{\circ},\IX_{i-1}^{\circ})$ of $(\varepsilon_i,\IX_{i-1})$ and define
    \[
        M_n^{\circ}(g) := \sum_{i=1}^{n}\frac{1}{d}\langle \varepsilon_i^{\circ},g(\IX_{i-1}^{\circ})\rangle.
    \]
    We will apply the following theory to $q = q^{*}(\frac{H}{n})^2$. Since $\sum_{q\in\IN}\beta(q) < \infty$, $q^{*}(\frac{H}{n}) \le \sqrt{\frac{n}{H}}$ and thus $q = q^{*}(\frac{H}{n})^2 \le n$.
    
    Now we have
    \begin{eqnarray}
        |M_n(g) - M_n^{\circ}(g)| &\le& \sum_{i=1}^{n}\frac{1}{d}\big(|\varepsilon_i^{\circ}|_2 |g(\IX_{i-1}^{\circ})|_2 + |\varepsilon_i|_2 |g(\IX_{i-1})|_2\big)\Ii_{\{(\varepsilon_i,\IX_{i-1}) \not= (\varepsilon_i^{\circ},\IX_{i-1}^{\circ})\}}\nonumber\\
        &\le& 2\sum_{i=1}^{n}\frac{1}{\sqrt{d}}(|\varepsilon_i|_2 + |\varepsilon_i^{\circ}|_2)\Ii_{\{(\varepsilon_i,\IX_{i-1}) \not= (\varepsilon_i^{\circ},\IX_{i-1}^{\circ})\}}.\label{bernstein_mixing_martingale_proof_eqminus1}
    \end{eqnarray}
    With \reff{bernstein_mixing_martingale_proof_eqminus1} and the Cauchy Schwarz inequality, we obtain
        \begin{equation}
            \IE[\sup_{g\in \sG}|M_n(g) - M_n^{\circ}(g)|] \le 2 n \frac{1}{\sqrt{d}}\||\varepsilon_1|_2\|_2 \|\Ii_{\{(\varepsilon_i,\IX_{i-1}) \not= (\varepsilon_i^{\circ},\IX_{i-1}^{\circ})\}}\|_2 \le 4n C_{\varepsilon} \tilde \beta(q)^{1/2} .\label{bernstein_mixing_martingale_proof_eq6}
        \end{equation}
    With $q = q^{*}(\frac{H}{n})^2$, we obtain from \reff{bernstein_mixing_martingale_proof_eq6} that
    \begin{eqnarray*}
         \IE[\sup_{g\in \sG}|M_n(g) - M_n^{\circ}(g)|] &\le& 4n C_{\varepsilon}\tilde \beta\Big(q^{*}(\frac{H}{n})^2\Big)^{1/2} \le 4n C_{\varepsilon} C_{\beta,sub}(\tilde \beta(q^{*}(\frac{H}{n}))^2)^{1/2}\\
         &\le& 4n C_{\varepsilon} C_{\beta,sub} \tilde q^{*}(\frac{H}{n})\frac{H}{n}
    \end{eqnarray*}
    where $\tilde q^{*}(x) = \min\{q\in\IN: \tilde \beta(q) \le qx\}$. With a similar argument as discussed in Lemma \ref{bernstein_mixing},
    \[
        \IE[\sup_{g\in \sG}|M_n(g) - M_n^{\circ}(g)|] \le 4nr C_{\varepsilon}C_{\beta,sub} q^{*}(\frac{H}{n})\frac{H}{n},
    \]
    which shows \reff{bernstein_mixing_martingale_res1}.
    
    We now show \reff{bernstein_mixing_martingale_res2} and \reff{bernstein_mixing_martingale_res3}. Therefore, we decompose
    \begin{equation}
        M_n^{\circ}(g) = \sum_{k=r+1, \text{ $k$ even}}^{\lfloor \frac{n}{q}\rfloor + 1}Y_k^{\circ}(g) + \sum_{k=r+1, \text{ $k$ odd}}^{\lfloor \frac{n}{q}\rfloor + 1}Y_k^{\circ}(g)\label{bernstein_mixing_martingale_proof_eq0}
    \end{equation}
    where
    \[
        Y_{k}^{\circ}(g) := \sum_{i=(k-1)q + 1}^{kq \wedge n}\frac{1}{d}\langle \varepsilon_i^{\circ}, g(\IX_{i-1}^{\circ})\rangle
    \]
    are independent. Since $(\langle \varepsilon_i, g(\IX_{i-1})\rangle)_i$ is a martingale, it holds by Theorem 2.1 in \cite{rio2009} that
    \begin{equation}
        \|Y_k^{\circ}(g)\|_m \le (m-1)^{1/2}\Big(\sum_{i=(k-1)q+1}^{kq\wedge n}\|\frac{1}{d}\langle \varepsilon_i^{\circ}, g(\IX_{i-1}^{\circ})\rangle\|_m^2\Big)^{1/2} \le (m-1)^{1/2}q^{1/2}\frac{1}{d}\|\langle \varepsilon_1^{\circ}, g(\IX_{r}^{\circ})\rangle\|_m.\label{bernstein_mixing_martingale_proof_eq1}
    \end{equation}
    By independence,
    \begin{equation}
        \Big(\frac{1}{d}\|\langle \varepsilon_1^{\circ}, g(\IX_{r}^{\circ})\rangle\|_m\Big)^m \le \Big\| \frac{|\varepsilon_1|_2}{\sqrt{d}} \Big\|_m^m \cdot \frac{1}{d}\IE[|g(\IX_{r})|_2^2]\cdot \|g\|_{\infty}^{m-2}.\label{bernstein_mixing_martingale_proof_eq2}
    \end{equation}
    Furthermore,
    \begin{eqnarray}
        \Big\| \frac{|\varepsilon_1|_2}{\sqrt{d}} \Big\|_m^m &\le& \IE\Big[\Big(\frac{1}{d}\sum_{j=1}^{d}\varepsilon_{1j}^2\Big)^{m/2}\Big] = \Big\|\frac{1}{d}\sum_{j=1}^{d}\varepsilon_{1j}^2\Big\|_{m/2}^{m/2} \le \Big(\frac{1}{d}\sum_{j=1}^{d}\|\varepsilon_{1j}^2\|_{m/2}\Big)^{m/2}\nonumber\\
        &\le& \|\varepsilon_{11}\|_m^m \le C_{\varepsilon}^{m}m^{m/2}.\label{bernstein_mixing_martingale_proof_eq3}
    \end{eqnarray}
    Insertion of \reff{bernstein_mixing_martingale_proof_eq3} into \reff{bernstein_mixing_martingale_proof_eq2} and afterwards into \reff{bernstein_mixing_martingale_proof_eq1} yields with $a := q^{*}(\frac{H}{n})\sqrt{\frac{H}{n}}$,
    \begin{eqnarray}
        \|\frac{Y_k^{\circ}(g)}{N(g)}\|_m^m &\le& (m-1)^{m/2}m^{m/2}\cdot (C_{\varepsilon}a^{-1}q^{1/2})^{m-2} \cdot qC_{\varepsilon}^2\IE[\frac{\frac{1}{d}|g(\IX_r)|_2^2}{\nu(g)^2}]\nonumber\\
        &\le& \frac{m!}{2}\cdot 2e^2 qC_{\varepsilon}^2 (C_{\varepsilon}e \cdot a^{-1}q^{1/2})^{m-2}. \label{bernstein_mixing_martingale_proof_eq100}
    \end{eqnarray}
    By Bernstein's inequality for independent variables, we conclude from \reff{bernstein_mixing_martingale_proof_eq100} that
    \[
        \IP\Big(\Big|\frac{1}{N(g)}\sum_{k=r+1, \text{ $k$ even}}^{\lfloor \frac{n}{q}\rfloor + 1}Y_k^{\circ}(g)\Big| > x\Big) \le 2\exp\Big(-\frac{1}{2}\frac{x^2}{\frac{n}{q}\cdot 2e^2 C_{\varepsilon}^2 q + e C_{\varepsilon} q^{1/2}a^{-1} x}\Big).
    \]
    Insertion into \reff{bernstein_mixing_martingale_proof_eq0} yields 
    \begin{eqnarray}
        \IP(|\frac{M_n^{\circ}(g)}{N(g)}| > x) &\le& 4\exp\Big(-\frac{1}{2}\frac{(x/2)^2}{2e^2 C_{\varepsilon}^2 n + e C_{\varepsilon} q^{1/2}a^{-1} (x/2)}\Big)\nonumber\\
        &\le& 4\exp\Big(-\frac{1}{8}\frac{x^2}{2e^2 C_{\varepsilon}^2 n + e C_{\varepsilon} q^{1/2}a^{-1} x}\Big) .\label{bernstein_mixing_martingale_proof_eq4}
    \end{eqnarray}
    \begin{enumerate}
        \item Standard arguments (cf. \cite{Vaart98}, Lemma 19.35) applied to \reff{bernstein_mixing_martingale_proof_eq4} yield that there exists some universal constant $c > 0$ such that
        \begin{equation}
            \IE \sup_{g\in \sG}|\frac{M_n^{\circ}(g)}{N(g)}| \le cC_{\varepsilon}\Big[\sqrt{nH} + q^{1/2}a^{-1}H\Big].\label{bernstein_mixing_martingale_proof_eq7}
        \end{equation}
        With $q = q^{*}(\frac{H}{n})^2$, we obtain
        \[
            \IE[\sup_{g\in \sG}|M_n^{\circ}(g)|] \le cC_{\varepsilon}\Big[\sqrt{nH} + q^{*}(\frac{H}{n})a^{-1}H\Big] \le 2cC_{\varepsilon} \sqrt{nH},
        \]
        which shows \reff{bernstein_mixing_martingale_res2}.
        \item Here, we use
\begin{eqnarray*}
    \IE\big[\sup_{g}|\frac{M_n^{\circ}(g)}{N(g)}|^2\big] = \int_0^{\infty}\IP\Big(\sup_g |\frac{M_n^{\circ}(g)}{N(g)}| > \sqrt{t}\Big) dt.
\end{eqnarray*}
Choose $G := 16eC_{\varepsilon}\frac{n a}{q^{1/2}}$. Then for $t \ge G^2$, $q^{1/2}a^{-1}\sqrt{t} \ge 2e^2 C_{\varepsilon}^2 n$. With \reff{bernstein_mixing_martingale_proof_eq4} and $\int_{b^2}^{\infty}\exp(-b_2\sqrt{t}) dt = \int_b^{\infty}2s \exp(-b_2 s) ds = 2(b_2b+1)b_2^{-2}\exp(-b_2b)$,
\begin{eqnarray}
    &&\int_0^{\infty}\IP\Big(\sup_g |\frac{M_n^{\circ}(g)}{N(g)}| > \sqrt{t}\Big) dt\nonumber\\
    &&G^2 + \int_{G^2}^{\infty}\IP\Big(\sup_g |\frac{M_n^{\circ}(g)}{N(g)}| > \sqrt{t}\Big) dt\nonumber\\
    &\le& G^2 + 4|\sG|\int_{G^2}^{\infty}\exp\Big(-\frac{1}{8}\frac{t}{2e^2 C_{\varepsilon}^2 n + eC_{\varepsilon}q^{1/2}a^{-1}\sqrt{t}}\Big) dt\nonumber\\
    &\le& G^2 + 4|\sG| \int_{G^2}\exp(-\frac{1}{16}\frac{\sqrt{t}}{eC_{\varepsilon}q^{1/2}a^{-1}}) dt\nonumber\\
    &\le& G^2 + 8|\sG|\Big(\frac{Ga}{eC_{\varepsilon}q^{1/2}} + 1\Big)(16eC_{\varepsilon}q^{1/2}a^{-1})^2 \exp(-\frac{1}{8}\frac{Ga}{eC_{\varepsilon}q^{1/2}})\nonumber\\
    &\le& 2^{10}e^2 C_{\varepsilon}^2\Big[\Big(\frac{na}{q^{1/2}}\Big)^2 + |\sG|\cdot \exp\Big(-n(\frac{a}{ q^{1/2}})^2\Big)\cdot (n + ( q^{1/2}a^{-1})^2)\Big].\label{proof_bernstein_mixing_martingale_eq7}
\end{eqnarray}
We obtain
\begin{eqnarray*}
    \IE\big[\sup_{g}|M_n^{\circ}(g)|^2\big] &\le& 2^{10}e^2C_{\varepsilon}^2\Big[ \Big(\frac{ na}{q^{1/2}}\Big)^2 + |\sG|\cdot \exp\Big(-n(\frac{a}{q^{1/2}})^2\Big)\cdot (n + (q^{1/2}a^{-1})^2)\Big].
\end{eqnarray*}
For $q = q^{*}(\frac{H}{n})^2$ this reads
\begin{eqnarray*}
    \IE\big[\sup_{g}|\frac{M_n^{\circ}(g)}{N(g)}|^2\big] &\le& 2^{10}e^2C_{\varepsilon}^2\Big[ \Big(\frac{ na}{q^{*}(\frac{H}{n})}\Big)^2 + |\sG|\cdot \exp\Big(-n(\frac{a}{q^{*}(\frac{H}{n})})^2\Big)\cdot (n + (q^{*}(\frac{H}{n})a^{-1})^2)\Big]\\
    &\le& 2^{10}e^2C_{\varepsilon}^2\Big[nH +  n + \frac{n}{H}\Big] \le 2^{11}e^2C_{\varepsilon}^2 n H,
\end{eqnarray*}
proving \reff{bernstein_mixing_martingale_res3}.
\end{enumerate}
\end{proof}

\begin{lem} \label{sub_mult_mod_beta}
    Let $\tilde \beta (k) := \beta ((k - r) \vee 0)$. Suppose that $\beta^{mix}(\cdot)$ is submultiplicative in the sense of \reff{definition_beta_submultiplicative}. Then for any $q_1,q_2,r \in \N$ there exists $C_{\beta,sub}$, such that
    \[
        \tilde \beta (q_1q_2) \le C_{\beta,sub} \tilde \beta(q_1)\tilde \beta(q_2).
    \]
\end{lem}
\begin{proof}[Proof of Lemma \ref{sub_mult_mod_beta}]
   By case distinction it is elementary to prove 
   \[
    ((q_1-r) \vee 0)((q_2-r) \vee 0) \le (q_1q_2 - r) \vee 0.
    \]Since $\beta$ is decreasing, we directly have
        \begin{eqnarray*}
        \tilde \beta(q_1q_2) &=& \beta((q_1q_2 - r) \vee 0) \\
        &\le& \beta(((q_1-r) \vee 0)((q_2-r) \vee 0)) \\
        &\le& C_{\beta,sub} \beta((q_1-r) \vee 0)\beta((q_2-r) \vee 0) = C_{\beta,sub} \tilde \beta(q_1)\tilde \beta(q_2)
        \end{eqnarray*}
\end{proof}

\subsubsection{Auxiliar results for oracle inequalities under mixing}

The following lemma is applied to $\varphi(x) = \phi^{-1}(\frac{1}{x})^{-1/2}$ in the proof of Theorem \ref{theorem_oracle_inequality}.

\begin{lem}\label{lemma_recursive_risk_general}
    Let $r_1,r_2,b,P > 0$ and $\varphi$ some concave function with $\varphi(0)=0$. If $a \ge 0$ satisfies $|a-b| \le r_1 \varphi(r_2 a)+ P$, then for any $\eta > 0$, 
    \[
        a \le \frac{\eta}{r_2} (\varphi^{-1})^{*}(\frac{1+\eta}{\eta}r_1r_2) + (1+\eta)(b+P).
    \]
\end{lem}
\begin{proof}[Proof of Lemma \ref{lemma_recursive_risk_general}]
The mapping $g(x) := \frac{\eta}{1+\eta}\frac{1}{r_2}\varphi^{-1}(x)$ is convex. By Young's inequality and denoting with $g^{*}$ the convex conjugate of $g$, 
\[
    r_1 \varphi(r_2 a) \le g^{*}(r_1) + g(\varphi(r_2 a)) = \frac{\eta}{(1+\eta)r_2}(\varphi^{-1})^{*}(\frac{1+\eta}{\eta}r_1 r_2) + \frac{\eta}{1+\eta}a.
\]
We therefore have
\[
    a \le |a-b| + b \le r_1 \varphi(a) + (b+P) \le \frac{\eta}{(1+\eta)r_2}(\varphi^{-1})^{*}(\frac{1+\eta}{\eta}r_1r_2) + \frac{\eta}{1+\eta}a + (b+P).
\]
Rearranging terms leads to
\[
    a \le \frac{\eta}{r_2} (\varphi^{-1})^{*}(\frac{1+\eta}{\eta}r_1r_2) + (1+\eta)(b+P).
\]

\end{proof}



\subsection{Results for the functional dependence measure}

\subsubsection{Dependence measure}

Recall the definition of the functional dependence measure coefficients $\delta^{X}_q(k)$, $k\in\IN_0$ from \reff{definition_uniform_functional_dependence_measure}.

During the proofs for the oracle inequalities under functional dependence, there will occur two quantities:
\begin{equation}
    (\tilde V^{-1})^{*}(\sqrt{x}) \quad\quad \text{ and }\quad\quad q^{*}(\sqrt{x})x,\label{definition_depmeasure_quantitiestoupperbound}
\end{equation}
where
\begin{equation}
    q^{*}(x) = q^{*,dep}(x) = \min\{q\in \IN: \beta^{dep}(q) \le qx\}\label{definition_qstar_depmeas}
\end{equation}
and
\begin{equation}
    \beta^{dep}(q) = \sum_{j=q}^{\infty}\Delta(j).\label{definition_beta_depmeas}
\end{equation}
Here, $\Delta(k)$, $k\in\IN_0$, is an upper bound chosen dependent on the function class of interest and specified below.
For $x$, we have to plug in a specific rate of the form $\frac{H}{n}$. It is therefore of interest to upper bound both quantities in \reff{definition_depmeasure_quantitiestoupperbound} by one common quantity.

Recall the definitions
\[
    \tilde V(x) = x^{1/2} + \sum_{j=0}^{\infty}\min\{x^{1/2}, \Delta(j)\}
\]
and $ \Lambda(x) = \sqrt{x}\bar y(x)$ as well as $\bar y(x)$ from \reff{def_y_dependencemeasure} and \reff{definition_lambda_depmeas}. 

In Lemma \ref{lemma_depmeas_raten}, we show that both terms in \reff{definition_depmeasure_quantitiestoupperbound} are bounded by a constant times $\Lambda(x)$. Thus, $\Lambda(x)$ serves as a common upper bound for both quantities in \reff{definition_depmeasure_quantitiestoupperbound}. Besides that, we show in Lemma \ref{V_concave} that $\tilde V$ is a concave function which is needed to obtain meaningful upper bounds in Theorem \ref{theorem_oracle_inequality_dep}.

\subsubsection{Unification of \reff{definition_depmeasure_quantitiestoupperbound}}

\begin{lem}\label{lemma_depmeas_raten}
Let $\tilde V'$ denote the left derivative of $\tilde V$.
\begin{enumerate}
    \item Let $h,\delta \ge 0$. Then $\tilde V(h) \le \delta$ implies $\sqrt{h} \le r(\delta)$.
    \item If there exists $C > 0$ such that for all $q\in\IN$, $\beta^{dep}(q) \le C q \cdot \Delta(q)$, then
    \[
        \inf_{x\in[0,\infty)}\frac{\tilde V'(x) x}{\tilde V(x)} \ge \frac{1}{2(1+C)}.
    \]
    \item Under the assumptions of (ii), it holds that
    \begin{equation}
        (\tilde V^{-1})^{*}(\sqrt{x}) \le 2(1+C)\cdot \Lambda(x), \quad\quad q^{*,dep}(\sqrt{x})x \le \Lambda(x).\label{vupperbound}
    \end{equation}
    \end{enumerate}
\end{lem}
\begin{proof}[Proof of Lemma \ref{lemma_depmeas_raten}]
    \begin{enumerate}
        \item It can be shown as in the proof of Lemma 3.6(ii) in \cite{empproc} that for any $h \in [0,\infty)$,
    \begin{equation}
        \tilde V(h) = \sqrt{h}\cdot a^{*} + \beta^{dep}(a^{*})\label{lemma_depmeas_raten_eq6}
    \end{equation}
    with some $a^{*}\in\IN$ dependent on $h$.
    
    Let $\delta > 0$. If $\tilde V(h) \le \delta$, then $\beta^{dep}(a^{*}) \le \delta$, that is, $\frac{\beta^{dep}(a^{*})}{a^{*}} \le \frac{\delta}{a^{*}}$. By definition of $q^{*,dep}$, $q^{*,dep}(\frac{\delta}{a^{*}}) \le a^{*}$. This implies $q^{*,dep}(\frac{\delta}{a^{*}})\frac{\delta}{a^{*}} \le \delta$ and thus by definition of $r(\cdot)$ and \reff{lemma_depmeas_raten_eq6},
    \[
        r(\delta) \ge \frac{\delta}{a^{*}} \ge \sqrt{h}.
    \]
        \item Let $x\in[0,\infty)$. If $\Delta(N) < \sqrt{x} < \Delta(N-1)$ for some $N \in\IN$. Then we have
    \[
        \tilde V(x) = x^{1/2} + \sum_{j=0}^{N-1}x^{1/2} + \sum_{j=N}^{\infty}\Delta(j) = (N+1) \sqrt{x} + \beta^{dep}(N),
    \]
    and thus $x\cdot  \partial_x \tilde V(x) = (N+1)\cdot \frac{1}{2}x^{1/2}$.
    By assumption $\beta^{dep}(N) \le N \Delta(N)$ and $\sqrt{x} > \Delta(N)$, 
    \begin{equation}
        \frac{x\cdot \partial_x \tilde V(x)}{\tilde V(x)} \ge \frac{1}{2}\cdot \frac{(N+1)\sqrt{x}}{(N+1)\sqrt{x} + C N\Delta(N)} \ge \frac{1}{2}\frac{(N+1)\Delta(N)}{(N+1) \Delta(N) + CN \Delta(N)} \ge \frac{1}{2(1+C)}.\label{lemma_depmeas_raten_eq2}
    \end{equation}
    Writing the left derivative as a limit of $\partial_x \tilde V(x)$, the result follows.
        
        \item Fix $x \ge 0$. Define $y_0(x) \in[0,\infty)$ such that it solves
        \begin{equation}
            \tilde V(\sqrt{x} y_0(x)) = y_0(x).\label{vup2}
        \end{equation}
        Since $z \mapsto \frac{\tilde V(z)}{z}$ is decreasing, $y_0(x) \le \bar y(x)$. It is therefore enough to show that the quantities in \reff{vupperbound} are bounded by multiples of $\sqrt{x}y_0(x) \le \Lambda(x)$.

        Let $\tilde W$ denote the right derivative of $\tilde V^{-1}$. By the First Claim in the proof of Lemma \ref{lemma_varphi}(iii) (which also holds for left or right derivatives), we obtain
        \begin{equation}
            (\tilde V^{-1})^{*}(\sqrt{x}) \le \tilde W(y_0)\cdot y_0.\label{lemma_depmeas_raten_eq7}
        \end{equation}
        For any $y \ge 0$ with $z = \tilde V^{-1}(y)$,
        \[
            \frac{\tilde W(y) y}{\tilde V^{-1}(y)} \le  \frac{\tilde V(z)}{z \tilde V'(z)} \le 2(1+C).
        \]
        Insertion into \reff{lemma_depmeas_raten_eq7} and using the definition of $y_0$ yields
        \[
            (\tilde V^{-1})^{*}(\sqrt{x}) \le \tilde W(y_0)\cdot y_0 \le  \tilde V^{-1}(y_0) = 2(1+C) y_0 \sqrt{x}.
        \]
        By \reff{vup2} and (i), we have
    \[
        \sqrt{x}y_0(x) \le r(y_0(x))^2.
    \]
    Together with $r(\delta)\le\delta$ for any $\delta \ge 0$, this implies
    \[
        \sqrt{x} = \frac{\sqrt{x}y_0}{y_0} \le \frac{r(y_0)^2}{y_0} \le r(y_0).
    \]
    Since $z \mapsto q^{*,dep}(z)z$ is increasing,
    \[
        q^{*,dep}(\sqrt{x})x \le q^{*,dep}(r(y_0)) r(y_0)\sqrt{x} \le y_0 \sqrt{x}. 
    \]
    \end{enumerate}
\end{proof}

\begin{lem}\label{V_homogen}
    For $c >0$, $(\tilde V^{-1})^\ast(cx) \le c^2(\tilde V^{-1})^\ast(x)$.
\end{lem}

\begin{proof}[Proof of Lemma \ref{V_homogen}]
We have
    \begin{eqnarray*}
        (\tilde V^{-1})^\ast(cx) &=& \sup_{y>0}\{cxy - \tilde V^{-1}(y)\} \\
        &=& \sup_{z>0}\{cx\tilde V(z) - z \} \\
        &=& \sup_{w>0}\{cx\tilde V(c^2w) - c^2w \} \\
        &\le& \sup_{w>0}\{c^2x\tilde V(w) - c^2w \} \\
        &\le& c^2 \sup_{w>0}\{x\tilde V(w) - w \} = c^2(\tilde V^{-1})^\ast(x)
    \end{eqnarray*} due to
    \[
        \tilde V(c'a) = \sum_{j = 0}^\infty \min\{c'a,\Delta(j)\} \le \sqrt{c'}\tilde V(a).
    \] for $c' > 0$.
\end{proof}

The following lemma shows that $\tilde V(\cdot)$ defined in \reff{definition_tildev_depmeasure} is concave. It is furthermore needed in the next section to get a good upper bound in the maximal inequalities.

\begin{lem} \label{V_concave}
    Let $(a_k)_{k\in\N}$ be a decreasing nonnegative sequence of real numbers for which $\sum_{k=0}^\infty a_k < \infty$. Then,
    \[
        y \mapsto v(y) := \Big( \sum_{k = 1}^N \min\{\sqrt{y},a_k\}\Big)^2, \qquad N\in\N \cup \{\infty\}
    \]
    is a concave map.
\end{lem}
\begin{proof}[Proof of Lemma \ref{V_concave}]
    It is obvious that $v$ is concave on $y \in (a_j^2,a_{j-1}^2]$ because $v$ can be represented as a sum of concave functions, namely
    \begin{eqnarray*}
        v(y) = \Big( (j-1)\sqrt{y} + \sum_{k=j}^N a_k \Big)^2 = (j-1)^2y + 2(j-1)\sum_{k=j}^Na_k\sqrt{y} + (\sum_{k=j}^N a_k)^2.
    \end{eqnarray*}
    We investigate the slope's behavior on the interval's open boundary.
    The derivative's left limit at $a_j$ yields
    \begin{eqnarray*}
        \lim_{y \to a_j^2 , y < a_j^2} \partial_yv(y) =  \lim_{y \to a_j^2 , y < a_j^2} \frac{j}{\sqrt{y}}\Big(j\sqrt{y} + \sum_{k = j+1}^Na_k \Big) &=& \frac{j}{a_j}\Big(ja_j + \sum_{k = j+1}^Na_k \Big) \\
        &=& \frac{j}{a_j}\Big((j-1)a_j + \sum_{k = j}^Na_k \Big).
    \end{eqnarray*}
    On the other hand, the right limit is given by
    \[
        \lim_{y \to a_j^2 , y > a_j^2} \partial_yv(y) =  \lim_{y \to a_j^2 , y < a_j^2} \frac{j-1}{\sqrt{y}}\Big((j-1)\sqrt{y} + \sum_{k = j}^Na_k \Big) = \frac{j-1}{a_j}\Big((j-1)a_j + \sum_{k = j}^Na_k \Big).
    \] Hence,
    \[
         \lim_{y \to a_j^2 , y < a_j^2} \partial_yf(y) \ge  \lim_{y \to a_j^2 , y > a_j^2} \partial_yf(y).
    \]
    Since $f$ is concave on intervals of the form $(a_j^2,a_{j-1}^2]$, the just proven inequality for the derivative implies that $f$ has a representation as $f(y) = \int_0^y \tilde v (x) dx$.
    Now let $\lambda \in (0,1)$. Since
    \begin{eqnarray*}
        \tilde v(x + \lambda (y-x)) =  \int_0^{x + \lambda (y-x)} \tilde v (z) dz &=& \int_0^x \tilde v (z) dz + \int_x^{x + \lambda (y-x)}  \tilde v (z) dz \\
        &=& \int_0^x \tilde v (z) dz + \int_0^{\lambda(y-x)} \tilde v (x+u)du.
    \end{eqnarray*}
    and
    \begin{eqnarray*}
        \tilde v(x) + \lambda (\tilde v(y) - \tilde v(x)) &=& \int_0^x \tilde v (z) dz + \lambda \int_x^y \tilde v(z) dz \\
        &=& \int_0^x \tilde v (z) dz + \lambda\int_0^{x-y} \tilde v(x+u)du
    \end{eqnarray*}
    for $x,y >0$, we conclude that $v$ is concave.
    
    The result for $N = \infty$ can be obtained since the limit of concave functions is concave.
\end{proof}

In this lemma we prove the upper bounds from Lemma \ref{lemma_depmeas_raten_specialcase} for $\Lambda(\cdot)$ which arise in the special cases of polynomial and exponential decay.

\begin{proof}[Proof of Lemma \ref{lemma_depmeas_raten_specialcase}]
    \begin{enumerate}
        \item In Lemma 8.13 in \cite{empproc} it was shown that
        \[
            \tilde V(z) \le C_{\kappa,\alpha}\cdot \max\{z^{\frac{1}{2}\frac{\alpha-1}{\alpha}},z^{\frac{1}{2}}\},
        \]
        where $C_{\kappa,\alpha} > 0$ is some constant only depending on $\kappa,\alpha$. Fix $x \in [0,\infty)$. With $\bar y(x) = c \max\{x^{\frac{1}{2}\frac{\alpha-1}{\alpha+1}},x^{\frac{1}{2}}\}$, we have
        \[
            \sqrt{x}\bar y(x) = c \max\{x^{\frac{\alpha}{\alpha+1}},x\}
        \]
        and by case distinction $x > 1$, $x \le 1$,
        \begin{eqnarray*}
            \tilde V(\sqrt{x} \bar y(x)) &\le& C_{\kappa,\alpha}\cdot \max\{c^{\frac{1}{2}\frac{\alpha-1}{\alpha}}, c^{\frac{1}{2}}\} \cdot  \max\{x^{\frac{1}{2}\frac{\alpha-1}{\alpha+1}},x^{\frac{1}{2}\frac{\alpha-1}{\alpha}},x^{\frac{1}{2}\frac{\alpha}{\alpha+1}},x^{\frac{1}{2}}\}\\
            &\le& C_{\kappa,\alpha} \max\{c^{\frac{1}{2}\frac{\alpha-1}{\alpha}}, c^{\frac{1}{2}}\}\cdot  \max\{x^{\frac{1}{2}\frac{\alpha-1}{\alpha+1}}, x^{\frac{1}{2}}\} = C_{\kappa,\alpha} \max\{c^{-\frac{1}{2}\frac{\alpha+1}{\alpha}}, c^{-\frac{1}{2}}\} \bar y(x),
        \end{eqnarray*}
        so choosing $c = \max\{C_{\kappa,\alpha}^{\frac{2\alpha}{\alpha+1}}, C_{\kappa,\alpha}^{2}\}$ yields the result.
        \item In Lemma 8.13 in \cite{empproc} it was shown that
        \[
            \tilde V(z) \le C_{\kappa,\rho}\cdot z^{\frac{1}{2}}\log(z^{-1} \vee 1),
        \]
        where $C_{\kappa,\rho} > 1$ is some constant only depending on $\kappa,\rho$. Fix $x \in [0,\infty)$. With $\bar y(x) = c x^{\frac{1}{2}}\log(x^{-1} \vee 1)^2$, we have
        \[
            \sqrt{x}\bar y(x) = c x\log(x^{-1}\vee 1)^2
        \]
        and by case distinction $x > 1$, $x \le 1$,
        \begin{eqnarray*}
            \tilde V(\sqrt{x} \bar y(x)) &\le& c^{\frac{1}{2}}\log(c^{-1} \vee e) C_{\kappa,\rho}\cdot \max\{x^{\frac{1}{2}}\log(x^{-1} \vee 1) (\log(x^{-1} \vee 1) - 2\log (\log(x^{-1} \vee 1) \vee 1)),\\
            &&\quad\quad x^{\frac{1}{2}}\log(x^{-1} \vee 1), x^{\frac{1}{2}}\}\\
            &\le& \frac{1}{4}c^{\frac{1}{2}}\log(c^{-1} \vee e)C_{\kappa,\rho} \cdot \max\{x^{\frac{1}{2}}\log(x^{-1} \vee 1)^2, x^{\frac{1}{2}}\}\\
            &=& \frac{1}{4}c^{-\frac{1}{2}}\log(c^{-1} \vee e)C_{\kappa,\rho}\bar y(x),
        \end{eqnarray*}
        so choosing $c = 16C_{\kappa,\rho}^2$ yields the result. 
    \end{enumerate}
\end{proof}

\subsubsection{Maximal inequalities under functional dependence}

The following empirical process results are based on the theory from \cite{empproc}.

Let $\sG \subset \{g:\IR^{dr} \to \IR\text{ measurable}\}$ be a finite class of  Lipschitz continuous functions in the sense that there exists $b \in (0,1]$ such that for $x_j \in \IR^d, x_j'\in \IR^d$, $j=1,...,r$,
\begin{equation}
    |g(x_1,...,x_r) - g(x_1',...,x_r')| \le L_{\sG}\cdot \max_{j=1,...,r}|x_j - x_j'|_{\infty},\label{cond_depmeas_lipschitz1}
\end{equation}
and for some $G > 0$,
\begin{equation}
    \sup_{g\in \sG}\|g\|_{\infty} \le G.\label{cond_depmeas_lipschitz2}
\end{equation}

For $\theta \in (0,1]$, it holds that
    \begin{eqnarray*}
        |g(x_1,...,x_r) - g(x_1',...,x_r')| &\le& \min\{2G,L_{\sG}\max_{i=1,...,r}|x_i-x_i'|_\infty\} \\
        &\le& (2G)^{1-\theta}L_{\sG}^{\theta}(\max_{i=1,...,r}|x_i-x_i'|_\infty)^{\theta}.
    \end{eqnarray*}
    
Then, using the notation from \reff{representation_x} and \reff{definition_uniform_functional_dependence_measure},
\begin{eqnarray}
    \delta_2^{g(\IX_{\cdot-1})}(k) &=& \|g(\IX_{i-1}) - g(\IX_{i-1}^{*(i-k)})\|_2\nonumber\\
    &\le& (2G)^{1-\theta}L_{\sG}^{\theta}\cdot \|\max_{j=1,...,r}|X_{i-j} - X_{i-j}^{*(i-k)}|_{\infty}^{\theta}\|_2\nonumber\\
    &\le& dr(2G)^{1-\theta}L_{\sG}^{\theta} \max_{j=\{1,...,r\},l\in \{1,...,d\}}\|X_{i-j,l} - X_{i-j,l}^{*(i-k)}\|_{2\theta}^{\theta}\nonumber\\
    &\le& dr(2G)^{1-\theta}L_{\sG}^{\theta} \cdot \sup_{l=1,...,r}\sup_{j=1,...,r}\delta_{2\theta}^{X_{\cdot,l}}(k-j)^{\theta}.\label{standard_bound_depmeas_g}
\end{eqnarray}

In the following, we suppose that $\Delta(k)$, $k\in\IN_0$ is a decreasing sequence chosen such that 
\begin{equation}
    dr(2G)^{1-\theta}L_{\sG}^{\theta} \cdot \sup_{l=1,...,r}\sup_{j=1,...,r}\delta_{2\theta}^{X_{\cdot,l}}(k-j)^{\theta} \le \Delta(k).\label{depmeas_condition_delta}
\end{equation}

Recall the definition of $q^{*,dep}$ and $\beta^{dep}$ from \reff{definition_qstar_depmeas} and \reff{definition_beta_depmeas}.

Put
\[
    S_n(g) := \sum_{i=r+1}^{n}\{g(\IX_{i-1}) - \IE g(\IX_{i-1})\}.
\]
To prove maximal inequalities for $S_n(g)$, we use the decomposition technique from \cite{empproc}, equation (3.1) therein. For $j \ge 1$ define
\[
    S_{n,j}(g) := \sum_{i=r+1}^{n}W_{i,j}(g), \quad\quad W_{i,j}(g) := \IE[g(\IX_{i-1})|\varepsilon_{i-j},...,\varepsilon_{i-1}].
\]
Then,
\begin{equation}
    S_n(g) = S_n(g) - S_{n,q}(g) + \sum_{l=1}^{L}(S_{n,\tau_l} - S_{n,\tau_{l-1}}) + S_{n,0}(g)\label{bernstein_depmeas_decomp1}
\end{equation}
where $L = \lfloor \frac{\log(q)}{\log(2)}\rfloor$ and  $\tau_l = 2^{l}$ ($l=0,...,L-1$), $\tau_L=q$ for arbitrary $q \in\{1,...,n\}$. 
Set
\[
    S_n^{\circ}(g) := S_{n,q}(g)
\]
and
\[
    S_{n,\tau_l}(g) - S_{n,\tau_{l-1}}(g) = \sum_{i=1}^{\lfloor \frac{n}{\tau_l}\rfloor+1}T_{i,l}(g), \quad\quad T_{i,l}(g) := \sum_{k=(i-1)\tau_l+1}^{(i\tau_l)\wedge n}\big[W_{k,\tau_l}(g) - W_{k,\tau_{l-1}}(g)\big].
\]
Hence,
\begin{eqnarray*}
    S_n^{\circ}(g) = \sum_{l=1}^{L}\Big[\underset{i \text{ even}}{\sum_{i=1}^{\lfloor \frac{n}{\tau_l}\rfloor + 1}}T_{i,l}(g) + \underset{i \text{ odd}}{\sum_{i=1}^{\lfloor \frac{n}{\tau_l}\rfloor + 1}}T_{i,l}(g)\Big] + S_{n,0}(g).
\end{eqnarray*}

\begin{lem}[Maximal inequalities under functional dependence]\label{bernstein_depmeas}
    Assume that $X_i$ is of the form \reff{representation_x}. Suppose that $\sG$ satisfies \reff{cond_depmeas_lipschitz1} and \reff{cond_depmeas_lipschitz2} with $G = 1$ and some $L_{\sG} > 0$. Let $\theta \in (0,1]$. Then, for any decreasing sequence $\Delta(k)$, $k\in\IN_0$, satisfying
    \[
        dr L_{\sG}^{\theta}\cdot \sup_{l=1,...,d}\sup_{j=1,...,r}\delta_{2\theta}^{X_{\cdot,l}}(k-j)^{\theta} \le \Delta(k)
    \]
    there exists some universal constant $c > 0$ such that
    \begin{equation}
         \IE \sup_{g\in \sG}|S_n(g) - S_n^{\circ}(g)| \le c  H q^{\ast,dep} (\sqrt{\frac{H}{n}}). \label{bernstein_depmeas_res1}
    \end{equation}
    Furthermore, for any estimator $\hat g \in \sG$ we have with some universal constant $c > 0$
    \begin{equation}
        \E |S_n^\circ(\hat g)| \le c(\sqrt{nH}\tilde V (\IE[\norm{g(\IX_r)}_1|_{g=\hat g}]) + qH).
        \label{bernstein_depmeas_res2}
    \end{equation}
\end{lem}

\begin{proof}[Proof of Lemma \ref{bernstein_depmeas}] Let $q \in \{1,...,n\}$ be arbitrary.
Then, as in the proof of Theorem 3.2 in \cite{empproc} (cf. the term $A_1$ in equation (8.13) therein), there exists a universal constant $c > 0$ such that
\[
    \IE \sup_{g\in \sG}|S_n(g) - S_n^{\circ}(g)| \le c \sqrt{nH}\beta^{dep}(q).
\] If $q := q^{\ast,dep} (\sqrt{\frac{H}{n}})$,
\[
    \IE \sup_{g\in \sG}|S_n(g) - S_n^{\circ}(g)| \le c \sqrt{nH} \beta(q^{\ast,dep} (\sqrt{\frac{H}{n}}))) \le c \sqrt{nH} q^{\ast,dep} (\sqrt{\frac{H}{n}})\sqrt{\frac{H}{n}} = c  H q^{\ast,dep} (\sqrt{\frac{H}{n}})
\] which proves \reff{bernstein_depmeas_res1}.

We employ a similar strategy as in the proof of Lemma \ref{bernstein_mixing}.
Let $N_l(g) := \tau_l\sqrt{\frac{H}{n}}\vee V_l(g)$ for $V_l(g):=\sum_{j = \tau_{l-1}+1}^{\tau_l}\min\{\norm{g}_2,\Delta(\lfloor \frac{j}{2} \rfloor)\}$.
We show the following two inequaltities first, where $c$ denote universal constants.
\begin{enumerate}[(i)]
    \item \begin{equation}
         \IE \sup_{g\in \sG}|\frac{S_{n,\tau_l}(g) - S_{n,\tau_{l-1}}(g)}{N_l(g)}| \le c \sqrt{nH},\label{bernstein_depmeas_res3}
    \end{equation}
        \item \begin{equation}
         \IE\big[\sup_{g\in \sG}|\frac{S_{n,\tau_l}(g) - S_{n,\tau_{l-1}}(g)}{N_l(g)}|^2\big] \le c nH.\label{bernstein_depmeas_res4}
    \end{equation}
\end{enumerate}
For $g\in \sG$, we have
\[
    \Big|\frac{T_{i,l}(g)}{N_l(g)}\Big| \le 2 \tau_l\norm{g}_\infty N_l(g)^{-1} \le 2 \tau_l \tau_l^{-1}\sqrt{\frac{n}{H}} = 2\sqrt{\frac{n}{H}}
\]
and by the same calculation as in the proof of Theorem 3.2 in \cite{empproc},
\[
   \frac{1}{\frac{n}{\tau_l}} \underset{i \text{ even}}{\sum_{i=1}^{\lfloor \frac{n}{\tau_l}\rfloor + 1}}\norm{\frac{T_{i,l}(g)}{N_l(g)}}_{2}^2 \le \frac{1}{N_l(g)^2}\Big(\sqrt{\tau_l}\sum_{j = \tau_{l-1}+1}^{\tau_l}\min\{\norm{g}_2,\Delta(\lfloor \frac{j}{2} \rfloor)\}\Big)^2 \le \tau_l.
\]
By Bernstein's inequality we obtain
\begin{eqnarray}
    &&\IP\Big(\big|\frac{S_{n,\tau_l}(g) - S_{n,\tau_{l-1}}(g)}{N_l(g)}\big| > x\Big)\nonumber\\
    &\le& \IP\Big(\Big|\underset{i \text{ even}}{\sum_{i=1}^{\lfloor \frac{n}{\tau_l}\rfloor + 1}}\frac{T_{i,l}(g)}{N_l(g)}\Big| > \frac{x}{2}\Big) + \IP\Big(\Big|\underset{i \text{ even}}{\sum_{i=1}^{\lfloor \frac{n}{\tau_l}\rfloor + 1}}\frac{T_{i,l}(g)}{N_l(g)}\Big| > \frac{x}{2}\Big)\nonumber\\
    &\le& 4\exp\Big(-\frac{1}{2}\frac{(x/2)^2}{n + 2\sqrt{\frac{n}{H}} x/2}\Big)\nonumber\\
    &\le& 4 \exp\Big(-\frac{1}{8}\frac{x^2}{n + \sqrt{\frac{n}{H}}x}\Big).\label{proof_bernstein_depmeas_eq2}
\end{eqnarray}

\begin{enumerate}
    \item
    Using standard arguments (cf. \cite{Vaart98}, Lemma 19.35), we derive from   \reff{proof_bernstein_depmeas_eq2} that there exists a universal constant $c > 0$ such that
    \[
        \IE \sup_{g\in \sG}\big|\frac{S_{n,\tau_l}(g) - S_{n,\tau_{l-1}}(g)}{N_l(g)}\big| \le c\sqrt{nH}.
    \]
    This shows \reff{bernstein_depmeas_res3}.
    
    \item
     Next, we use
    \begin{eqnarray*}
        \IE\big[\sup_{g \in \sG}|\frac{S_{n,\tau_l}(g) - S_{n,\tau_{l-1}}(g)}{N_l(g)}|^2\big] = \int_0^{\infty}\IP\Big(\sup_{g\in\sG} |\frac{S_{n,\tau_l}(g) - S_{n,\tau_{l-1}}(g)}{N_l(g)}| > \sqrt{t}\Big) dt.
    \end{eqnarray*}
    Put $a := \sqrt{\frac{H}{n}}$ and choose $G := 16na$. Then for $t \ge G^2$, $a^{-1}\sqrt{t} \ge n$. With \reff{proof_bernstein_depmeas_eq2} and  $\int_{b^2}^{\infty}\exp(-b_2\sqrt{t}) dt = \int_b^{\infty}2s \exp(-b_2 s) ds = 2(b_2b+1)b_2^{-2}\exp(-b_2b)$, we obtain
    \begin{eqnarray}
        &&\int_0^{\infty}\IP\Big(\sup_g |\frac{S_{n,\tau_l}(g) - S_{n,\tau_{l-1}}(g)}{N_l(g)}| > \sqrt{t}\Big) dt\nonumber\\
        &=&G^2 + \int_{G^2}^{\infty}\IP\Big(\sup_g |\frac{S_{n,\tau_l}(g) - S_{n,\tau_{l-1}}(g)}{N_l(g)}| > \sqrt{t}\Big) dt\nonumber\\
        &\le& G^2 + 4|\sG|\int_{G^2}^{\infty}\exp\Big(-\frac{1}{8}\frac{t}{n + \sqrt{\frac{n}{H}}\sqrt{t}}\Big) dt\nonumber\\
        &\le& G^2 + 4|\sG|\int_{G^2}^{\infty}\exp\Big(-\frac{1}{8}\frac{t}{n + a^{-1}\sqrt{t}}\Big) dt\nonumber\\
        &\le& G^2 + 4|\sG| \int_{G^2}^\infty \exp(-\frac{1}{16}\frac{\sqrt{t}}{a^{-1}}) dt\nonumber\\
        &\le& G^2 + 8|\sG|\Big(Ga + 1\Big)(16a^{-1})^2 \exp(-\frac{1}{16}Ga)\nonumber\\
        &\le& 2^{11}\Big[\Big(na\Big)^2 + |\sG|\cdot \exp\Big(-na^2\Big)\cdot (n + a^{-2})\Big].\label{proof_bernstein_depmeas_eq7}
    \end{eqnarray}
    We have
    \begin{eqnarray*}
        \IE\big[\sup_{g}|\frac{S_{n,\tau_l}(g) - S_{n,\tau_{l-1}}(g)}{N_l(g)}|^2\big] &\le& 2^{12}\Big[(na)^2 + |\sG|\cdot \exp\big(-na^2\big)\cdot (n + a^{-2})\Big].
    \end{eqnarray*}
    which can be upper bounded by
    \begin{eqnarray*}
        \IE\big[\sup_{g}|\frac{S_{n,\tau_l}(g) - S_{n,\tau_{l-1}}(g)}{N_l(g)}|^2\big] \le 2^{12}\Big[nH + n + \frac{n}{H}\Big] = 2^{14}n H,
    \end{eqnarray*}
    proving \reff{bernstein_depmeas_res3}.
\end{enumerate}
Moving on to $ \IE|S_n^{\circ}(\hat g)|$, the Cauchy-Schwarz inequality yields
    \begin{eqnarray}
        \IE|S_n^{\circ}(\hat g)| &\le& \sum_{l=1}^L \Big[ \Big\|\frac{S_{n,\tau_l}(\hat g) - S_{n,\tau_{l-1}}( \hat g)}{N_l(\hat g)}\Big\|_2 \IE[ V_l(\hat g)^2]^{1/2} \nonumber\\
        && + \IE\Big|\frac{S_{n,\tau_l}(\hat g) - S_{n,\tau_{l-1}}( \hat g)}{N_l(\hat g)}\Big|\cdot \tau_l\sqrt{\frac{H}{n}} \Big]\nonumber\\
        && + \IE[S_{n,0}(\hat g)] \nonumber \\
        &\le& c(\sqrt{nH}\sum_{l = 1}^L\IE[ V_l(\hat g)^2]^{1/2} + qH) + \IE[|S_{n,0}(\hat g)|].\label{convergenceproof_part1_depmeas_eq3}
    \end{eqnarray}
The last summand can be discussed as follows.
Let $N(g):= \sqrt{\frac{H}{n}} \vee \norm{g}_2$,
    \begin{eqnarray}
        \IE[\big|S_{n,0}(\hat g)\big|] &=& \IE[\big|\frac{S_{n,0}(\hat g)}{N(\hat g)} \cdot N(\hat g)\big|]\nonumber\\
        &\le& \norm{\frac{S_{n,0}(\hat g)}{N(\hat g)}}_2 \IE[\norm{g}_2^2\big|_{g=\hat g}]^{1/2}+ \IE\big[\big|\frac{S_{n,0}(\hat g)}{N(\hat g)}\big|\big]\sqrt{\frac{H}{n}}. \label{proof_bernstein_depmeas_eq9}
    \end{eqnarray}
Since $S_{n,0}(g) = \sum_{i=r+1}^{n}W_{i,0}(g)$ is a sum of independent variables with $|W_{i,0}(g)| \le \|g\|_{\infty} \le 1$ and $\|W_{i,0}(g)\|_2 \le 2\|g\|_2$, the Bernstein inequality yields
\[
    \IP\Big(\big|\frac{S_{n,0}(g)}{N(g)}\big| > x\Big) \le 2 \exp\big(-\frac{1}{4}\frac{x^2}{n+N(g)^{-1}x/2}\big),
\] from which we derive
\[
    \IE\big[\sup_{g\in\G}\big|\frac{S_{n,0}(g)}{N(g)}\big|\big] \le c(\sqrt{nH}+N(g)^{-1}H) \le c\sqrt{nH}
\] for some universal constant $c > 0$.
In analogy to the calculation of equation \reff{proof_bernstein_depmeas_eq7},
\[
    \IE\Big[\sup_{g \in \sG}\Big|\frac{S_{n,0}(g)}{N(g)}\Big|^2\Big] \le 2^{10}nH.
\]
Therefore, equation \reff{proof_bernstein_depmeas_eq9} can be bounded by
\begin{equation}
     \IE\sup_{g \in \sG}|S_{n,0}(\hat g)| \le c(\sqrt{nH}\IE[\norm{g(\IX_r)}_1|_{g=\hat g}]^{1/2}+H). \label{proof_bernstein_depmeas_eq10}
\end{equation}
Now, let us define $v_l(x) := \sum_{j = \tau_{l-1}+1}^{\tau_l}\min\{\sqrt{x},\Delta(\lfloor \frac{j}{2} \rfloor)\}$. Then,
\begin{eqnarray*}
     V_l(h)^2 &=& \Big(\sum_{j = \tau_{l-1}+1}^{\tau_l}\min\{\norm{h(\IX_r)}_2,\Delta(\lfloor \frac{j}{2} \rfloor)\}\Big)^2 \\
     &\le& \Big(\sum_{j = \tau_{l-1}+1}^{\tau_l}\min\{\norm{h(\IX_r)}_1^{1/2},\Delta(\lfloor \frac{j}{2} \rfloor)\}\Big)^2  = v_l(\norm{h}_1)^2.
\end{eqnarray*}
This implies the first bound of the following inequality. The second bound follows from Jensen's inequality while taking into account that $v_l^2$ is concave by Lemma \ref{V_concave}:
\begin{equation}
    \IE[ V_l(\hat g)^2 ]^{1/2} \le \IE[v_l(\norm{g(\IX_r)}_1)^2\big|_{g = \hat g}]^{1/2} \le v_l(\IE[\norm{g(\IX_r)}_1 \big|_{g=\hat g}]). \label{proof_bernstein_depmeas_eq8}
\end{equation}
Inserting equations \reff{proof_bernstein_depmeas_eq8}, \reff{proof_bernstein_depmeas_eq10} into \reff{convergenceproof_part1_depmeas_eq3} and applying  \cite[Lemma 8.2]{empproc} afterwards (which allows to replace $\Delta(\lfloor \frac{j}{2}\rfloor)$ in $v_l(\cdot)$ by $\Delta(j)$) gives
\begin{eqnarray*}
     \IE\sup_{g \in \sG}|S_n^{\circ}(\hat g)| &\le& c\Big(\sqrt{nH}(\sum_{l=1}^Lv_l(\IE[\norm{g(\IX_r)}_1 \big|_{g=\hat g}]) + \IE[\norm{g(\IX_r)}_1|_{g=\hat g}]^{1/2} + (q+1)H \Big) \\
     &\le& c2\Big(\sqrt{nH}(2\sum_{j=1}^\infty \min\{\IE[\norm{g(\IX_r)}_1 \big|_{g=\hat g}]^{1/2},\Delta(j)\} + \IE[\norm{g(\IX_r)}_1|_{g=\hat g}]^{1/2}) + qH \Big) \\
     &\le& c(\sqrt{nH}\tilde V (\IE[\norm{g(\IX_r)}_1|_{g=\hat g}]) + qH).
\end{eqnarray*}
\end{proof}

For $g:\IR^{dr} \to \IR^d$, let
\[
    M_n(g) := \sum_{i=r+1}^{n}\frac{1}{d}\langle \varepsilon_i,g(\IX_{i-1})\rangle.
\]
\begin{lem}[Maximal inequalities for martingale sequences under functional dependence]\label{bernstein_depmeas_martingale}
    Assume that $X_i$ is of the form \reff{representation_x} and that Assumption \ref{ass_subgaussian} holds. Suppose that any component of $g \in \sG$ satisfies \reff{cond_depmeas_lipschitz1} and \reff{cond_depmeas_lipschitz2} with $G = 1$ and some $L_{\sG} > 0$. Let $\theta \in (0,1]$. Then, with any decreasing sequence $\Delta(k)$, $k\in\IN_0$, satisfying
    \[
        dr L_{\sG}^{\theta} \cdot \sup_{l=1,...,r}\sup_{j=1,...,r}\delta_{2\theta}^{X_{\cdot,l}}(k-j)^{\theta} \le \Delta(k),
    \]
    there exists another process $M_n^{\circ}(g)$ and some universal constant $c > 0$ such that
    \begin{equation}
        \IE \sup_{g\in \sG}|M_n(g) - M_n^{\circ}(g)| \le cC_\eps \sqrt{nH}\beta^{dep}(q)\label{bernstein_depmeas_res5}.
    \end{equation}
    Furthermore for an estimator $\hat g:\IR^{dr} \to \IR^d$,
    \begin{equation}
        \E[|M_n^\circ(\hat g)|] \le cC_{\varepsilon}(\sqrt{nH}(\sqrt{\log(q)}+1)\IE[ \| |g(\IX_1)|_2 \|_2\big|_{g = \hat g} ]^{1/2} + q^{1/2}H)
        \label{bernstein_depmeas_res6}
    \end{equation}
    
\end{lem}

\begin{proof}[Proof of Lemma \ref{bernstein_depmeas_martingale}]
    We use a similar decomposition as in the proof of Lemma \ref{bernstein_depmeas}. For $j \ge 1$ define
    \[
    M_{n,j}(g) := \sum_{i=r+1}^{n}\bar W_{i,j}(g), \quad\quad \bar W_{i,j}(g) := \IE[\frac{1}{d}\langle \varepsilon_i, g(\IX_{i-1})\rangle|\varepsilon_{i-j},...,\varepsilon_{i}] = \frac{1}{d}\langle \varepsilon_i, W_{i,j}(g)\rangle,
    \]
    where $W_{i,j}(g) := \IE[g(\IX_{i-1})|\varepsilon_{i-j},...,\varepsilon_{i-1}]$. Define
    \[
        M_n^{\circ}(g) := M_{n,q}(g).
    \]
    Note that $(\langle \varepsilon_i, g(\IX_{i-1})\rangle)_i$ is a martingale and for fixed $j$, the sequence
    \begin{eqnarray*}
        (E_{i,j}(g))_{g\in \sG} &=& \big((\bar W_{i,j+1}(g) - \bar W_{i,j}(g))\big)_{g\in \sG}
    \end{eqnarray*}
    is a  $|\sG|$-dimensional martingale difference vector with respect to $\sA^{i} := \sigma(\varepsilon_{i-j},\varepsilon_{i-j+1},...)$.
    Since
    \[
        \sup_{g\in \sG}|E_{i,j}(g)| = \sup_{g\in \sG}|\bar W_{i,j+1}(g) - \bar W_{i,j}(g)| \le \frac{1}{d}|\varepsilon_i|_2\cdot \sup_{g\in \sG}|W_{i,j+1}(g) - W_{i,j}(g)|_2,
    \]
    we have by \reff{standard_bound_depmeas_g} (which also holds with $\sup_{g}$ inside the $\|\cdot\|_2$-norm),
    \begin{eqnarray*}
    \big\| \sup_{g\in\sG}|E_{i,j}(g)|\, \big\|_2 &\le& \frac{1}{d}\sum_{k=1}^d \norm{\varepsilon_{ik}}_2 \norm{ \sup_{g\in\sG} |g(\IX_{i-1})_k-g(\IX_{i-1}^{\ast(i-j)})_k| }_2 \le C_\eps\Delta(j).\label{mart_diff_eq3}
    \end{eqnarray*}
    Therefore, in analogy to the proof of Theorem 3.2 found in \cite{empproc},
    \[
        \IE \sup_{g\in \sG}|M_n(g) - M_n^{\circ}(g)| \le \sum_{j=q}^\infty \norm{\sup_{g\in \sG}\Big|\sum_{i=r+1}^n E_{i,j}(g)\Big|}_2 \le cC_\eps \sqrt{nH}\beta^{dep}(q).
    \]
    for some universal constant $c>0$.

    For $q \in \{1,...,n\}$, it holds that
    \[
        M_{n}^{\circ}(g) =  \sum_{l=1}^{L}\Big[\underset{i \text{ even}}{\sum_{i=1}^{\lfloor \frac{n}{\tau_l}\rfloor + 1}}\bar T_{i,l}(g) + \underset{i \text{ odd}}{\sum_{i=1}^{\lfloor \frac{n}{\tau_l}\rfloor + 1}}\bar T_{i,l}(g)\Big] + M_{n,0}(g)
    \]
    where
    \[
        \bar T_{i,l}(g) := \sum_{k=(i-1)\tau_l+1}^{(i\tau_l)\wedge n}\big[\bar W_{k,\tau_l}(g) - \bar W_{k,\tau_{l-1}}(g)\big] = \frac{1}{d}\sum_{k=(i-1)\tau_l+1}^{(i\tau_l)\wedge n}\langle \varepsilon_k, W_{k,\tau_l}(g) - W_{k,\tau_{l-1}}(g)\rangle.
    \]

Next, let
\[
    N_l(g) = \max\{\tau_l^{1/2}\cdot\sqrt{\frac{H}{n}}, D_l(g)\}, \quad\quad D_l(g) := \IE[\frac{1}{d}|W_{1,\tau_l}(g) - W_{1,\tau_{l-1}}(g)|_2^2]^{1/2}.
\]
We show the following two inequalities first, where $c$ denotes a universal constant.
\begin{enumerate}[(i)]
    \item \begin{equation}
         \IE \sup_{g\in \sG}|\frac{M_{n,\tau_l}(g) - M_{n,\tau_{l-1}}(g)}{N_l(g)}| \le cC_{\varepsilon}\sqrt{nH}\label{bernstein_depmeas_mod_res2}
    \end{equation}
        \item \begin{equation}
         \IE\big[\sup_{g\in \sG}|\frac{M_{n,\tau_l}(g) - M_{n,\tau_{l-1}}(g)}{N_l(g)}|^2\big] \le cC_{\varepsilon}^2 nH.\label{bernstein_depmeas_mod_res3}
    \end{equation}
\end{enumerate}
We have, similar to \reff{bernstein_mixing_martingale_proof_eq100}, by Theorem 2.1 in \cite{rio2009} and Assumption \ref{ass_subgaussian},
\begin{eqnarray}
    &&\frac{1}{\lfloor\frac{n}{\tau_l}\rfloor}\sum_{i=1}^{\lfloor\frac{n}{\tau_l}\rfloor+1}\IE[|\bar T_{i,l}(g )|^m]\nonumber\\
    &\le& (m-1)^{m/2}\frac{1}{\lfloor \frac{n}{\tau_l}\rfloor}\sum_{i=1}^{\lfloor \frac{n}{\tau_l}\rfloor+1}\frac{1}{d}\Big(\sum_{k=(i-1)\tau_l+1}^{(i\tau_l)\wedge n}\IE[|\langle \varepsilon_k, W_{k,\tau_l}(g) - W_{k,\tau_{l-1}}(g)\rangle|^m]^2\Big)^{1/2}\nonumber\\
    &\le& (m-1)^{m/2} \frac{1}{d}\tau_l^{m/2}\IE[|\varepsilon_1|_2^m]\cdot \IE[|W_{1,\tau_l}(g) - W_{1,\tau_{l-1}}(g)|_2^m]\nonumber\\
    &\le& (m-1)^{m/2}\tau_l^{m/2}C_{\varepsilon}^m\cdot \frac{1}{d}\IE[|W_{1,\tau_l}(g) - W_{1,\tau_{l-1}}( g)|_2^2]\norm{g}_\infty^{m-2} \nonumber \\
    &\le& \frac{m!}{2} \cdot 2e^2C_{\varepsilon}^2\tau_l\IE[\frac{1}{d}|W_{1,\tau_l}(g) - W_{1,\tau_{l-1}}(g)|_2^2] \cdot (eC_{\varepsilon}\tau_l^{1/2})^{m-2} \label{bernstein_depmeas_martingale_proof_eq100}
\end{eqnarray}
With $\tilde a:=\tau_l^{1/2}\cdot\sqrt{\frac{H}{n}}$, 
\[
    \frac{1}{\lfloor\frac{n}{\tau_l}\rfloor}\sum_{i=1}^{\lfloor\frac{n}{\tau_l}\rfloor+1}\frac{1}{N_l(g)}\IE[|T_{i,l}(g )|^m] \le \frac{m!}{2} \cdot 2e^2 C_{\varepsilon}^2\tau_l (\tilde a^{-1}eC_{\varepsilon}\tau_l^{1/2})^{m-2}
\]
By Bernstein's inequality for independent variables, we conclude that
\begin{eqnarray}
    &&\IP\Big(\big|\frac{M_{n,\tau_l}(g) - M_{n,\tau_{l-1}}(g)}{N_l(g)}\big| > x\Big)\nonumber\\
    &\le& \IP\Big(\Big|\underset{i \text{ even}}{\sum_{i=1}^{\lfloor \frac{n}{\tau_l}\rfloor + 1}}\frac{T_{i,l}(g)}{N_l(g)}\Big| > \frac{x}{2}\Big) + \IP\Big(\Big|\underset{i \text{ even}}{\sum_{i=1}^{\lfloor \frac{n}{\tau_l}\rfloor + 1}}\frac{T_{i,l}(g)}{N_l(g)}\Big| > \frac{x}{2}\Big)\nonumber\\
    &\le& 4 \exp\Big(-\frac{1}{8}\frac{x^2}{2(e C_{\varepsilon})^2 n + eC_{\varepsilon}\sqrt{\frac{n}{H}}x}\Big).\label{proof_bernstein_depmeas_mod_eq2}
\end{eqnarray}

\begin{enumerate}
    \item
    Using standard arguments (cf. \cite{Vaart98}, Lemma 19.35), we derive from   \reff{proof_bernstein_depmeas_eq2} that there exists a universal constant $c > 0$ such that
    \[
        \IE \sup_{g\in \sG}\big|\frac{M_{n,\tau_l}(g) - M_{n,\tau_{l-1}}(g)}{N_l(g)}\big| \le cC_{\varepsilon}\sqrt{nH}.
    \]
    This shows \reff{bernstein_depmeas_mod_res2}.
    
    \item
     Next, we use
    \begin{eqnarray*}
        \IE\big[\sup_{g \in \sG}|\frac{M_{n,\tau_l}(g) - M_{n,\tau_{l-1}}(g)}{N_l(g)}|^2\big] = \int_0^{\infty}\IP\Big(\sup_{g\in\sG} |\frac{M_{n,\tau_l}(g) - M_{n,\tau_{l-1}}(g)}{N_l(g)}| > \sqrt{t}\Big) dt.
    \end{eqnarray*}
    Put $a := \sqrt{\frac{H}{n}}$. Choose $G := 16(eC_{\varepsilon})na$. Then for $t \ge G^2$, $a^{-1}\sqrt{t} \ge n$. With \reff{proof_bernstein_depmeas_mod_eq2} and  $\int_{b^2}^{\infty}\exp(-b_2\sqrt{t}) dt = \int_b^{\infty}2s \exp(-b_2 s) ds = 2(b_2b+1)b_2^{-2}\exp(-b_2b)$, we obtain
    \begin{eqnarray}
        &&\int_0^{\infty}\IP\Big(\sup_g |\frac{M_{n,\tau_l}(g) - M_{n,\tau_{l-1}}(g)}{N_l(g)}| > \sqrt{t}\Big) dt\nonumber\\
        &=&G^2 + \int_{G^2}^{\infty}\IP\Big(\sup_g |\frac{M_{n,\tau_l}(g) - M_{n,\tau_{l-1}}(g)}{N_l(g)}| > \sqrt{t}\Big) dt\nonumber\\
        &\le& G^2 + 4|\sG|\int_{G^2}^{\infty}\exp\Big(-\frac{1}{8}\frac{x^2}{2(e C_{\varepsilon})^2 n + eC_{\varepsilon}\sqrt{\frac{n}{H}}x}\Big) dt\nonumber\\
        &\le& G^2 + 4|\sG|\int_{G^2}^{\infty}\exp\Big(-\frac{1}{8}\frac{t}{2(eC_{\varepsilon})^2n + eC_{\varepsilon}a^{-1}\sqrt{t}}\Big) dt\nonumber\\
        &\le& G^2 + 4|\sG| \int_{G^2}^\infty \exp(-\frac{1}{16}\frac{\sqrt{t}}{eC_{\varepsilon} a^{-1}}) dt\nonumber\\
        &\le& G^2 + 8|\sG|\Big(\frac{Ga}{eC_{\varepsilon}} + 1\Big)(16eC_{\varepsilon}a^{-1})^2 \exp(-\frac{1}{16eC_{\varepsilon}}Ga)\nonumber\\
        &\le& 2^{11}(eC_{\varepsilon})^2\Big[(na)^2 + |\sG|\cdot \exp\big(-na^2\big)\cdot (n + a^{-2})\Big].\label{proof_bernstein_depmeas_mod_eq7}
    \end{eqnarray}
    We conclude that
    \begin{eqnarray*}
        \IE\big[\sup_{g}|\frac{M_{n,\tau_l}(g) - M_{n,\tau_{l-1}}(g)}{N_l(g)}|^2\big] &\le& 2^{12}(eC_{\varepsilon})^2\Big[(na)^2 + |\sG|\cdot \exp\big(-na^2\big)\cdot (n + a^{-2})\Big]
    \end{eqnarray*}
    which can be upper bounded by
    \begin{eqnarray*}
        \IE\big[\sup_{g}|\frac{M_{n,\tau_l}(g) - M_{n,\tau_{l-1}}(g)}{N_l(g)}|^2\big] \le 2^{12}(eC_{\varepsilon})^2\Big[nH + n + \frac{n}{H}\Big] \le 2^{14}(eC_{\varepsilon})^2 n H.
    \end{eqnarray*}
    This shows \reff{bernstein_depmeas_mod_res3}.
\end{enumerate}
Using \reff{bernstein_depmeas_mod_res2} and \reff{bernstein_depmeas_mod_res3}, we can now upper bound  $ \IE|M_n^{\circ}(\hat g)|$. By the Cauchy-Schwarz inequality,
    \begin{eqnarray}
        \IE|M_n^{\circ}(\hat g)| &\le& \sum_{l=1}^L \Big[ \Big\|\frac{M_{n,\tau_l}(\hat g) - M_{n,\tau_{l-1}}( \hat g)}{N_l(\hat g)}\Big\|_2 \IE[D_l(\hat g)^2]^{1/2} \nonumber\\
        && + \IE\Big|\frac{M_{n,\tau_l}(\hat g) - M_{n,\tau_{l-1}}( \hat g)}{N_l(\hat g)}\Big|\cdot \tau_l^{1/2}\sqrt{\frac{H}{n}} \big]\nonumber\\
        && + \IE[M_{n,0}(\hat g)] \nonumber \\
        &\le& cC_{\varepsilon}(\sqrt{nH}\sum_{l = 1}^L\IE[D_l(\hat g)^2]^{1/2} + q^{1/2}H) + \IE[|M_{n,0}(\hat g)|].\label{convergenceproof_part1_depmeas_mod_eq3}
    \end{eqnarray}
    For $v\in\IR^{L}$, $|v|_1 \le \sqrt{L} |v|_2$. Thus,
    \begin{eqnarray}
        \sum_{l = 1}^L \IE[D_l(\hat g)^2]^{1/2} &\le& \sqrt{L}\cdot \Big(\sum_{l=1}^{L}\IE[D_l(\hat g)^2]\Big)^{1/2}\nonumber\\ &=& \sqrt{L}\IE\Big[\Big(\frac{1}{d}\sum_{l=1}^{L}\IE[|W_{1,\tau_l}(g) - W_{1,\tau_{l-1}}(g)|_2^2]\Big)\Big|_{g=\hat g}\Big]^{1/2}\label{proof_bernstein_depmeas_eq11}
    \end{eqnarray}
    Note that $(W_{1,\tau_l}(g) - W_{1,\tau_{l-1}}(g))_l$ is a martingale difference sequence with respect to $\tilde \sA^{l} := \sigma(\varepsilon_{1-\tau_l},...,\varepsilon_1)$. We therefore have
    \[
        \IE[|W_{1,q}(g) - W_{1,0}(g)|_2^2] = \IE\Big[\Big|\sum_{l=1}^{L}W_{1,\tau_l}(g) - W_{1,\tau_{l-1}}(g)\Big|_2^2\Big] = \sum_{l=1}^{L}\IE[|W_{1,\tau_l}(g) - W_{1,\tau_{l-1}}(g)|_2^2].
    \]
    Since the left hand side is bounded by $4\IE[|g(\IX_{1})|_2^2]$ by the projection property of conditional expectations, insertion into \reff{proof_bernstein_depmeas_eq11} yields
    \begin{equation}
        \sum_{l=1}^{L}\IE[D_l(\hat g)]^{1/2} \le 4\sqrt{L}\cdot \IE[ \frac{1}{d}\| |g(\IX_{1})|_2 \|_2^2\big|_{g=\hat g}]^{1/2}.\label{proof_bernstein_depmeas_eq12}
    \end{equation}
The last summand in \reff{convergenceproof_part1_depmeas_mod_eq3} can be similarly dealt with. With $N(g):= \sqrt{\frac{H}{n}} \vee \norm{|g(\IX_1)|_2}_2$,
    \begin{eqnarray}
        \IE[\big|M_{n,0}(\hat g)\big|] &=& \IE[\big|\frac{M_{n,0}(\hat g)}{N(\hat g)} \cdot N(\hat g)\big|]\nonumber\\
        &\le& \norm{\frac{M_{n,0}(\hat g)}{N(\hat g)}}_2 \IE[\norm{|g(\IX_1)|_2}_2^2\big|_{g=\hat g}]^{1/2}+ \IE\big[\big|\frac{M_{n,0}(\hat g)}{N(\hat g)}\big|\big]\sqrt{\frac{H}{n}}. \label{proof_bernstein_depmeas_mod_eq9}
    \end{eqnarray}
Since $M_{n,0}(g) = \sum_{i=r+1}^{n}\bar W_{i,0}(g)$ is a sum of independent variables, we can proceed as before in Lemma \ref{bernstein_depmeas} and obtain the existence of universal constants $c > 0$ such that
\[
    \norm{\frac{M_{n,0}(\hat g)}{N(\hat g)}}_2\le cC_{\varepsilon} \sqrt{nH}, \quad\quad \IE\big[\big|\frac{M_{n,0}(\hat g)}{N(\hat g)}\big|\big] \le cC_{\varepsilon}^2 \sqrt{nH}.
\]
Insertion into \reff{proof_bernstein_depmeas_mod_eq9} yields
\begin{equation}
    \IE[\big|M_{n,0}(\hat g)\big|] \le cC_{\varepsilon}\big\{\sqrt{nH}\cdot \IE[\norm{|g(\IX_1)|_2}_2^2\big|_{g=\hat g}]^{1/2} + H\big\}.\label{proof_bernstein_depmeas_eq13}
\end{equation}
Insertion of \reff{proof_bernstein_depmeas_eq12} and \reff{proof_bernstein_depmeas_eq13} into \reff{convergenceproof_part1_depmeas_mod_eq3} yields the result.
\end{proof}

\subsubsection{Oracle inequalities under functional dependence}

Let $\sF \subset\{f:\IR^{dr} \to \IR^d \text{ measurable}\}$ such that any $f = (f_j)_{j=1,...,d} \in \sF$ satisfies
\begin{equation}
    \sup_{j\in \{1,...,d\}}|f_j(x) - f_j(x')| \le L_{\sF}\cdot |x-x'|_{\infty}\label{lipschitzcond_depmeas_f1}
\end{equation}
and
\begin{equation}
    \sup_{j\in \{1,...,d\}}\sup_{x\in \text{supp}(\IW)}|f_j(x)| \le F\label{lipschitzcond_depmeas_f2}
\end{equation}
where $\IW:\IR^{dr} \to [0,1]$ is an arbitrary weight function depending on $\varsigma > 0$ with
\[
    |\IW(x) - \IW(x')| \le \frac{1}{\varsigma}\cdot |x-x'|_{\infty}.
\]
Let
\[
    \hat f \in \argmin_{f\in \sF}\hat R_n(f).
\]
The main result of this section is the following theorem. Here, $H(\delta) = \log N(\delta, \sF, \|\cdot\|_{\infty})$.

\begin{thm}\label{theorem_oracle_inequality_dep}
    Suppose that $X_i$ is of the form \reff{representation_x} and that Assumption \ref{ass_subgaussian} holds. Assume that there exist $F > 0, L_{\sF} > 0$ such that $\sF$ satisfies \reff{lipschitzcond_depmeas_f1} and \reff{lipschitzcond_depmeas_f2}. Furthermore, suppose that $f_0:\IR^{dr} \to \IR^d$ from \reff{model_time_evolution} is such that $|f_{0}(x) - f_{0}(x')|_{\infty} \le K|x-x'|_{\infty}$ for some $K > 0$.

    Suppose that Assumption \ref{ass_compatibility2} holds with $L_{\sG} = 2dr\big(\frac{2}{\varsigma} + \frac{(L_{\sF} + K)}{F}\big)$. Let $\delta \in (0,1)$. Then, for any $\eta > 0$ there exists a constant $\IC = \IC(\eta, C_{\varepsilon}, F)$ such that
    \[
        \IE D(\hat f) \le (1+\eta)^2 \inf_{f\in \sF}D(f) + \IC\cdot \big\{ \Lambda(\frac{H(\delta)}{n}) + \delta\big\}.
    \]
\end{thm}

\begin{proof}[Proof of Theorem \ref{theorem_oracle_inequality_dep}]
  Let $\eta > 0$. We follow the proof of Theorem \ref{theorem_oracle_inequality}. Define
\begin{eqnarray*}
    R_{1,n} &:=& (1+\eta)cF^2 q^{*,dep}(\sqrt{\frac{H}{n}})\frac{H}{n} + \frac{\eta F^2}{2} (\tilde V^{-1})^{*}\Big(2\frac{1+\eta}{\eta}\sqrt{\frac{H}{n}}\Big),\\
    R_{1,\delta} &:=& cF^2\sqrt{\frac{H}{n}}\tilde V(2F^{-2}\delta^2),\\
    R_{2,n} &:=& 2cC_{\varepsilon}Fq^{*,dep}(\frac{H}{n})\frac{H}{n},\\
    R_{2,\delta} &:=& C_{\varepsilon}\delta + 2cC_{\varepsilon}F \sqrt{\frac{H}{n}}\delta.
\end{eqnarray*}
Then, as in the mixing case, by Lemma \ref{convergenceproof_dep_part1}, \reff{riskproof_eq1} and Lemma \ref{convergenceproof_dep_part2},  
\begin{eqnarray*}
    \IE D(\hat f) &\le& (1+\eta)\IE \hat D_n(\hat f) + R_{1,n} + (1+\eta)R_{1,\delta}\\
    &\le& (1+\eta)\Big\{\inf_{f \in \sF}D(f) + 2cC_{\varepsilon}F \sqrt{\frac{H}{n}}\sqrt{q^{*,dep}({\sqrt{\frac{H}{n}}})}\IE[D(\hat f)]^{1/2} + R_{2,n} + R_{2,\delta} + R_{1,\delta} \Big\} + R_{1,n} .
\end{eqnarray*}
Due to $2ab \le a^2 + b^2$ with $a := (1+\eta)cC_{\varepsilon}F\sqrt{\frac{H}{n}}\sqrt{q^{*,dep}({\sqrt{\frac{H}{n}}})}(\frac{1+\eta}{\eta})^{1/2}$, $b:= (\frac{\eta}{1+\eta})^{1/2}\IE[D(\hat f)]^{1/2}$, we obtain
\begin{eqnarray*}
    \IE D(\hat f) &\le& (1+\eta)\inf_{f \in \sF}D(f) + \frac{(1+\eta)^3}{\eta}(cC_{\varepsilon}F)^2q^\ast(\sqrt{\frac{H}{n}})\frac{H}{n} + \frac{\eta}{1+\eta}\IE[D(\hat f)]\\
    &&\quad\quad\quad\quad + (1+\eta)(R_{2,n} + R_{2,\delta} +  R_{1,\delta}) + R_{1,n}.
\end{eqnarray*}
This implies
\begin{eqnarray}
    \IE D(\hat f) &\le& (1+\eta)^2 \inf_{f \in \sF}D(f) + (1+\eta)R_{1,n} \nonumber\\
    &&\quad\quad\quad\quad + (1+\eta)^2 (R_{2,n} + R_{2,\delta} + R_{1,\delta}) + \frac{(1+\eta)^4}{\eta}(c C_{\varepsilon}F)^2q^{*,dep}(\sqrt{\frac{H}{n}})\frac{H}{n}.\label{riskproof_dep_eq2}
\end{eqnarray} Using Young's inequality applied to $\tilde V^{-1}$ ($\tilde V^{-1}$ is convex) and Lemma \ref{lemma_depmeas_raten}, we obtain
\[
    R_{1,\delta} \le cF^2(\tilde V^{-1})^{*}\big( \sqrt{\frac{H}{n}}\big) + 2c \delta^2 \le cF^2 \Lambda(\frac{H}{n}) + 2c \delta^2,
\] as well as $R_{2,n} \le 4c F \Lambda(\frac{H}{n})$. Furthermore,
\[
    R_{1,n} \le (1+\eta)cF^2 \Lambda(\frac{H}{n}) + \frac{\eta F^2}{2} \Big(2\frac{1+\eta}{\eta}\Big)^2 \Lambda(\frac{H}{n}).
\] and
\[
    R_{2,\delta} \le C_{\varepsilon}(\delta + cF\delta^2 + cF\frac{H}{n}).
\]
Insertion of these results into \reff{riskproof_dep_eq2} yields the assertion.
\end{proof}

To prove Theorem \ref{theorem_oracle_inequality_dep}, the following two lemmata are used.

\begin{lem}\label{convergenceproof_dep_part2}
    Suppose that $X_i$ is of the form \reff{representation_x} and that Assumption \ref{ass_subgaussian} holds. Assume that there exist $F > 0, L_{\sF} > 0$ such that $\sF$ satisfies \reff{lipschitzcond_depmeas_f1} and \reff{lipschitzcond_depmeas_f2}. Furthermore, suppose that $f_0:\IR^{dr} \to \IR^d$ from \reff{model_time_evolution} is such that $|f_{0}(x) - f_{0}(x')|_{\infty} \le K|x-x'|_{\infty}$ for some $K > 0$.

    If additionally Assumption \ref{ass_compatibility2} holds with $L_{\sG} = 2dr \big(\frac{K + L_{\sF}}{F} + \frac{2}{\varsigma}\big)$, then
    \begin{eqnarray*}
        &&\Big|\IE\Big[\frac{1}{nd}\sum_{i=1}^{n}\langle \varepsilon_i, \hat f(\IX_{i-1})\rangle \IW(\IX_{i-1})\Big]\Big|\\
        &\le& C_{\varepsilon}\delta + 2c C_{\varepsilon}F\Big[q^{\ast,dep}(\sqrt{\frac{H}{n}})\frac{H}{n} + \sqrt{\frac{H}{n}}\Big(\sqrt{q^{\ast,dep}(\sqrt{\frac{H}{n}})}\IE[D(\hat f)]^{1/2} + \delta\Big)\Big].
    \end{eqnarray*}
\end{lem}
\begin{proof}[Proof of Lemma \ref{convergenceproof_dep_part2}]
    As in the proof of Lemma \ref{convergenceproof_part2}, let $j^{*} \in \{1,...,\sN_n\}$ be such that $\|\hat f - f_{j^{*}}\|_{\infty} \le \delta$. Since $\varepsilon_i$ is independent of $\IX_{i-1}$ and $\IE \varepsilon_i = 0$, we have
    \begin{equation}
         \Big|\IE\Big[\frac{1}{nd}\sum_{i=1}^{n}\langle \varepsilon_i, \hat f(\IX_{i-1})\rangle \IW(\IX_{i-1})\Big]\Big| \le \delta \cdot \underbrace{\frac{1}{nd}\sum_{i=1}^{n}\IE|\varepsilon_i|_1}_{\le \frac{1}{d}\sum_{k=1}^{d}\IE|\varepsilon_{1k}| \le C_{\varepsilon}} + \frac{F}{n}|\IE M_n(g_{j^{*}})|,\label{convergenceproof_dep_part2_eq1}
    \end{equation}
    where $g_j(x) := \frac{1}{F}(f_{j}(x) - f_0(x))\IW(x)$ and $M_n(\cdot)$ is from Lemma \ref{bernstein_depmeas_martingale}. We choose $\sG = \{g_j: j \in \{1,...,\sN_n\}\}$. Since
    \[
        \sup_{j=1,...,\sN_n}\sup_{k=1,...,d}\|g_{jk}\|_{\infty} \le \frac{1}{F}\cdot F\cdot \|\IW\|_{\infty} \le 1
    \]
    and
    \begin{eqnarray*}
        |g_{jk}(x) - g_{jk}(x')| &\le& \frac{1}{F}|f_{jk}(x) - f_{0k}(x) - f_{jk}(x') - f_{0k}(x')|\cdot \IW(x)\\
        &&\quad + \frac{1}{F}|f_{jk}(x') - f_{0k}(x')|\cdot |\IW(x) - \IW(x')|\\
        &\le& \Big(\frac{K + L_{\sF}}{F} + \frac{2}{\varsigma}\Big)\cdot |x-x'|_{\infty}
    \end{eqnarray*}
    That is, $\sG$ satisfies \reff{cond_depmeas_lipschitz1} and \reff{cond_depmeas_lipschitz2} with $G = 1$ and $L_{\sG} = \frac{K + L_{\sF}}{F} + \frac{2}{\varsigma}$. With the argument \reff{standard_bound_depmeas_g}, we conclude that for any $l\in \{1,...,d\}$,
    \[
       \delta_2^{g_{jl}(\IX)}(k) \le 2dr L_{\sG}\cdot \sup_{j\in \{1,...,r\}}\sup_{j=1,...,r}\delta_{2\theta}^{\IX_{\cdot,l}}(k-j)^\theta \le \Delta(k).
    \]
    
    By Lemma \ref{bernstein_depmeas_martingale}, \reff{bernstein_depmeas_res5} and \reff{bernstein_depmeas_res6} (taking $q := q^{\ast,dep}(\sqrt{\frac{H}{n}})$), we obtain
    \begin{eqnarray}
        \IE|M_n(g_{j^{*}})| &\le& c C_{\varepsilon}\sqrt{nH} \beta^{dep}(q^{\ast,dep}(\sqrt{\frac{H}{n}})) +  \IE|M_n^{\circ}(g_{j^{*}})|\nonumber\\
        &\le& c C_{\varepsilon}H q^{\ast,dep}(\sqrt{\frac{H}{n}}) + c C_{\varepsilon} \Big(\sqrt{nH}(\sqrt{\log(q^{\ast,dep}(\sqrt{\frac{H}{n}}))} + 1)\nonumber\\
        &&\quad\quad\quad\quad\quad\quad\times\IE\big[\| |g(\IX_r)|_2\|_2\big|_{g=g_{j^{*}}}]^{1/2} + q^{\ast,dep}(\sqrt{\frac{H}{n}})^{1/2}H\Big)\nonumber\\
        &\le& 2c C_{\varepsilon} \Big(H q^{*,dep}(\sqrt{\frac{H}{n}}) +  \sqrt{nH}\sqrt{q^{*,dep}(\sqrt{\frac{H}{n}})}\cdot \IE[D(f_{j^{*}})]^{1/2}\Big).\label{convergenceproof_dep_part2_eq2}
    \end{eqnarray}
    Since $\|\hat f_k - f_{j^{*}k}\|_{\infty} \le \delta$, $k = 1,...,d$, we have
    \begin{equation}
        \IE[D(f_{j^{*}})]^{1/2} \le \frac{1}{\sqrt{d}}\| |\hat f(\IX_r) - f_{j^{*}}(\IX_r)|_2 \IW(\IX_r)\|_2 + \IE[D(\hat f)]^{1/2} \le \delta + \IE[D(\hat f)]^{1/2}.\label{convergenceproof_dep_part2_eq3}
    \end{equation}
    Insertion of \reff{convergenceproof_dep_part2_eq2} and \reff{convergenceproof_dep_part2_eq3} into \reff{convergenceproof_dep_part2_eq1} yields the result.
\end{proof}

\begin{lem}\label{convergenceproof_dep_part1}
Suppose that $X_i$ is of the form \reff{representation_x} and that Assumption \ref{ass_subgaussian} holds. Assume that there exist $F > 0, L_{\sF} > 0$ such that $\sF$ satisfies \reff{lipschitzcond_depmeas_f1} and \reff{lipschitzcond_depmeas_f2}. Furthermore, suppose that $f_0:\IR^{dr} \to \IR^d$ from \reff{model_time_evolution} is such that $|f_{0}(x) - f_{0}(x')|_{\infty} \le K|x-x'|_{\infty}$ for some $K > 0$.

If additionally Assumption \ref{ass_compatibility2} holds with $L_{\sG} = 2dr\big(\frac{2}{\varsigma} + \frac{(L_{\sF} + K)}{F}\big)$, then there exists some universal constant $c > 0$ such that for every $\eta > 0$,
\begin{eqnarray*}
    \IE D(\hat f) &\le& (1+\eta)\IE \hat D_n(\hat f) \\
    && \qquad+ \Big\{(1+\eta)cF^2 q^{\ast,dep}(\sqrt{\frac{H}{n}})\frac{H}{n} + c\frac{F^2}{2}\eta (\tilde V^{-1})^{*}\Big(2\frac{1+\eta}{\eta}\sqrt{\frac{H}{n}}\Big)\Big\}\\
    &&\qquad + (1+\eta)cF^2\sqrt{\frac{H}{n}}\tilde V(2F^{-2}\delta^2).
\end{eqnarray*}
\end{lem}
\begin{proof}[Proof of Lemma \ref{convergenceproof_dep_part1}]
    The proof follows a similar structure to Lemma \ref{convergenceproof_part1}. Let $(f_j)_{j = 1,...,\sN_n}$ be a $\delta$-covering of $\sF$ w.r.t. $\|\cdot\|_{\infty}$, where $\sN_n := N(\delta, \sF, \|\cdot\|_{\infty})$. Let $j^{*} \in \{1,...,\sN_n\}$ be such that $\|\hat f - f_{j^{*}}\|_{\infty} \le \delta$. Without loss of generality, assume that  $\delta \le F$.
    
    Let $(X_i')_{i\in\IZ}$ be an independent copy of the original time series $(X_i)_{i\in\IZ}$. We have
    \begin{eqnarray}
        \big|\IE D(\hat f) - \IE \hat D_n(\hat f)|
        \le \frac{2F^2}{n}\IE|S_n(g_{j^{*}})| + 10\delta F,\label{convergenceproof_dep_part1_eq1}
    \end{eqnarray}
    where for $x,x' \in \IR^{dr}$,
    \[
        g_j(x,x') := \frac{1}{2dF^2}|f_{j}(x') - f_0(x')|_2^2\IW(x') - \frac{1}{dF^2}|f_{j}(x) - f_0(x)|_2^2\IW(x),
    \]
    and $S_n(\cdot)$ is from Lemma \ref{bernstein_depmeas} based on the process $(X_i,X_{i}')$, where $X_i'$, $i \in\IZ$, is an independent copy of $X_i$, $i\in\IZ$.
    
    Here, due to the assumption on $f_0$ and on $\sF$,  
    \begin{eqnarray*}
        |g_j(x,x') - g_j(y,y')| &\le& \frac{1}{2dF^2}\big(|f_j(x') - f_0(x')|_2^2 - |f_j(y') - f_0(y')|_2^2\big) \IW(x')\\
        &&\quad\quad + \frac{1}{2dF^2}|f_j(y') - f_0(y')|_2^2\cdot |\IW(x') - \IW(y')|\\
        &&\quad\quad + \frac{1}{2dF^2}\big(|f_j(x) - f_0(x)|_2^2 - |f_j(y) - f_0(y)|_2^2\big) \IW(x)\\
        &&\quad\quad + \frac{1}{2dF^2}|f_j(y) - f_0(y)|_2^2\cdot |\IW(x) - \IW(y)|\\
        &\le& \Big(\frac{2}{\varsigma} + \frac{(L_{\sF} + K)}{F}\big)(|x-y|_{\infty} + |x' - y'|_{\infty}),
    \end{eqnarray*}
    and
    \[
        \|g_j\|_{\infty} \le \frac{2}{2dF^2}\cdot dF^2\cdot \|\IW\|_{\infty} \le 1.
    \]
    Thus, $\sG = \{g_j:j=1,...,\sN_n\}$ satisfies the conditions \reff{cond_depmeas_lipschitz1} and \reff{cond_depmeas_lipschitz2} with $G = 1$ and
    \[
        L_{\sG} = \Big(\frac{2}{\varsigma} + \frac{L_{\sF} + K}{F}\Big).
    \]
    Since $X_i', i\in\IZ$ has the same distribution as $X_i, i\in\IZ$, the argument \reff{standard_bound_depmeas_g} yields for $j \in \{1,...,\sN_n\}$ that
    \[
        \delta_2^{g_j(\IX_{\cdot-1}, \IX_{\cdot-1}')}(k) \le 2dr L_{\sG} \cdot \sup_{l=1,...,r}\sup_{j=1,...,r}\delta_{2\theta}^{X_{\cdot,l}}(k-j)^{\theta} \le \Delta(k).
    \]
    
    Furthermore, there exists another process $S_n^{\circ}(\cdot)$ and some universal constant $c > 0$ such that
    \begin{equation*}
        \IE|S_n(g_{j^{*}}) - S_n^{\circ}(g_{j^{*}})| \le \IE \sup_{g\in \sG}|S_n(g) - S_n^{\circ}(g)| \le c\sqrt{nH}\beta^{dep}(q).
    \end{equation*} For $q = q^\ast(\sqrt{\frac{H}{n}})$,
    \begin{equation}
        \IE|S_n(g_{j^{*}}) - S_n^{\circ}(g_{j^{*}})| \le c  H q^\ast (\sqrt{\frac{H}{n}}). \label{convergenceproof_part1_dep_eq2}
    \end{equation}
    Insertion of \reff{convergenceproof_part1_dep_eq2} and \reff{bernstein_depmeas_res2} into \reff{convergenceproof_dep_part1_eq1} yields
    \begin{equation}
        |\IE D(\hat f) - \IE \hat D_n(\hat f)| \le 2cF^2\Big[q^{*}(\sqrt{\frac{H}{n}})\frac{H}{n} + \sqrt{\frac{H}{n}} \tilde V (\IE[\norm{g(\IX_r)}_1|_{g = g_{j^\ast}}])\Big].\label{convergenceproof_part1_dep_eq4}
    \end{equation} Now, observe that
    \[
       \tilde V (\IE[\norm{g(\IX_r)}_1|_{g = g_{j^\ast}}]) \le \tilde V(F^{-2}\IE D(f_{j^\ast})) \le \tilde V(2F^{-2}\delta^2) + \tilde V(2F^{-2}\IE D(\hat f))
    \] which together with Lemma \ref{lemma_recursive_risk_general} delivers,
    \begin{eqnarray*}
        \IE D(\hat f) &\le& (1+\eta)\Big[\IE \hat D_n(\hat f) + 2cF^2 q^{*}(\sqrt{\frac{H}{n}})\frac{H}{n} + 2cF^2\tilde V(2F^{-2}\delta^2) \Big] \\
        && \qquad + c\eta F^2(\tilde V^{-1})^\ast\Big(2\frac{1+\eta}{\eta}\sqrt{\frac{H}{n}} \Big).
    \end{eqnarray*}
\end{proof}

\subsection{Approximation results}

In this section we consider the approximation error as well as the size of the corresponding network class.

\subsubsection{Proof of the approximation error, Section \ref{sec_approx_error}}

\begin{proof}[Proof of Theorem \ref{approx_error}]
    We follow the proof given by \cite[Theorem 1]{schmidthieber2017} and employ \cite[Theorem 5]{schmidthieber2017} (recited here as part of Theorem \ref{theorem_approximation_network}), adapting it to the ``encoder-decoder" structure. Since $C$ is not explicitly given, it is enough to prove the result for large enough $n$. Fix $N\in\IN$ and choose $m = \lceil \log_2(n)\rceil$. 
    
    By Theorem \ref{theorem_approximation_network}, we find for arbitrarily chosen $N > 0$ functions
\[
    \tilde g_{enc,0} \in \sF(L_{enc,0} + 2,(dr,p_{enc,0},D),D(s_{enc,0}+4))
\]
and
\[
    \tilde g_{enc,1} \in \sF(L_{enc,1} + 2,(D,p_{enc,1},\tilde d),\tilde d(s_{enc,1}+4))
\]
where 
\begin{eqnarray*}
    L_{enc,i} &=& 8 + (m+5)(1 + \log_2(t_{enc,i} \vee \beta_{enc,i})), \\
    p_{enc,0} &=& D(6(t_{enc,0}+\lceil\beta_{enc,0}\rceil)N,\dots,6(t_{enc,0}+\lceil\beta_{enc,0}\rceil)N) \in \IR^{L_{enc,0}+2}, \\
    p_{enc,1} &=& \tilde d(6(t_{enc,1}+\lceil\beta_{enc,1}\rceil)N,\dots,6(t_{enc,1}+\lceil\beta_{enc,1}\rceil)N) \in \IR^{L_{enc,1}+2}
    \end{eqnarray*} and
\[
    s_{enc,i} \le 141((t_{enc,i}+\beta_{enc,i}+1)^{3+t_{enc,i}}N(m+6), \qquad i = 0,1,
\]
such that 
\[
    \norm{(g_{enc,i})_j - (\tilde g_{enc,i})_j}_{\infty} \le (2K+1)(1+t_{enc,i}^2+\beta_{enc,i}^2)6^{t_{enc,i}}N2^{-m} + K3^{\beta_{enc,i}}N^{\frac{\beta_{enc,i}}{t_{enc,i}}}
\]
for $i = 0,1$.
The composed network $\tilde f_{enc} := \tilde g_{enc,1} \circ \tilde g_{enc,0}$ satisfies
\begin{eqnarray*}
    \tilde f_{enc} \in \sF(L_{enc,0} + L_{enc,1} + 5, (dr,p_{enc,0},D,p_{enc,1},\tilde d),\tilde d(D(s_{enc,0}+4)+\tilde d(s_{enc,1}+4)))
\end{eqnarray*}
 as well as $\tilde f_{enc} \in \sF(L_1,\bar p, \bar s)$ for
 \[
    L_{enc,0} + L_{enc,1} + 5 \le \sum_{i \in \{enc,0;enc,1\}} \log_2(4(t_i \vee \beta_i))\log_2(n) \le L_1
\]
(the first inequality holds for $n$ large enough) and
 \begin{eqnarray*}
    \bar p &:=& (\underbrace{dr,...,dr}_{(\bar k + 1) \text{ times}},p_{enc,0},D,p_{enc,1},\tilde d)\\
    \bar s &:=& \tilde d(D(s_{enc,0}+4)+\tilde d(s_{enc,1}+4))) + \bar k dr
 \end{eqnarray*} where $\bar k := L_1 - (L_{enc,0} + L_{enc,1} + 5)$ (cf. \cite[Section 7.1]{schmidthieber2017}). Furthermore, by Theorem \ref{theorem_approximation_network} there exists a network
\begin{eqnarray*}
    \tilde f_{dec} &\in& \sF(L_{dec} + 2,(\tilde d,p_{dec},d),d(s_{dec}+4))
\end{eqnarray*}
where
\begin{eqnarray*}
    L_{dec} &=& 8 + (m+5)(1 + \log_2(t_{dec} \vee \beta_{dec})),\\
    p_{dec} &=& d(6(t_{dec}+\lceil\beta_{dec}\rceil)N,\dots,6(t_{dec}+\lceil\beta_{dec}\rceil)N) \in \IR^{L_{dec} + 2},\\
    s_{dec} &\le& 141((t_{dec}+\beta_{dec}+1)^{3+t_{dec}}N(m+6),
\end{eqnarray*}
such that
\[
    \|(f_{dec})_j - (\tilde f_{dec})_j\|_{\infty} \le (2K+1)(1+t_{dec,i}^2+\beta_{dec,i}^2)6^{t_{dec,i}}N2^{-m} + K3^{\beta_{dec,i}}N^{\frac{\beta_{dec,i}}{t_{dec,i}}}
\]
for $j = 1,...,d$.
We then obtain $\tilde f_0 = \tilde f_{dec} \circ \tilde f_{enc} \in \sF(L', p',s')$ by composing the networks $\tilde f_{enc}$ and $\tilde f_{dec}$ (cf. \cite[Section 7.1]{schmidthieber2017}) with the values 
\begin{eqnarray*}
    L' &:=& L_1 + L_{dec} + 1,\\ 
    p' &:=& (\bar p,p_{dec},d),\\
    s' &:=& \bar s + d(s_{dec}+4).
\end{eqnarray*}
The composition also satisfies $\tilde f_0 \in \sF(L,p,s)$ by additional layers where
\[
    L' \le L_1 + \log_2(4(t_{dec} \vee \beta_{dec}))\log_2(n) \le L
\]
(where the first inequality holds for $n$ large enough) and $s$, $p$ are set according to \cite[Section 7.1, equation (18)]{schmidthieber2017}, i.e.
\[
   k = L - L', \qquad p = (\underbrace{dr,...,dr}_{k \text{ times}},p'), \qquad s= s' + kp'_0.
\]
The conditions (ii) to (v) are automatically met. In analogy to \cite[Section 7.1, Lemma 3]{schmidthieber2017},
\begin{eqnarray} \label{approx_error_eq}
    \|\tilde f_0 - f_0\|_{\infty}^2 &\le& C \max_{k \in \{dec;enc,0;enc,1\}}\big\{\frac{N}{n} + N^{-\frac{2\beta_k}{t_k}}\big\}
\end{eqnarray} for a constant $C$ that only depends on $\tb, \betab$. By Theorem \ref{theorem_approximation_network}, since $N2^{-m} \lsim 1$, $\tilde f_0$ has Lipschitz constant
\[
    \|\tilde f_0\|_{\mathrm{Lip}} \le \|\tilde f_{dec}\|_{\mathrm{Lip}} \cdot \|\tilde g_{enc,1}\|_{\mathrm{Lip}} \cdot \|\tilde g_{enc,0}\|_{\mathrm{Lip}} \le C_2
\]
for a constant $C_2$ only depending on $\betab$, $\tb$.

Up to now, $\tilde f_0$ is not bounded by a given $F$. For large enough $n$ we are able to generate a sequence $(\tilde f_n)_{n \in \IN}$ in $\sF(L,L_1,p,s,\bar F,C_2)$ ($\bar F$ chosen arbitrarily large) satisfying equation \reff{approx_error_eq}.
If we define $f_n^\ast:=(\frac{\norm{f_0}_\infty}{\norm{\tilde f_n}_\infty} \wedge 1)\tilde f_n$,
\[
    \norm{f_n^\ast}_{\infty} \le \norm{f_0}_{\infty} \le \norm{f_{dec}}_{\infty} \le K \le F
\]
by assumption (i). Therefore, $f_n^\ast \in \sF(L,L_1,p,s,F,C_2)$. Equation \reff{approx_error_eq} also holds for the class $\sF(L,L_1,p,s,F,C_2)$ since $\|f_n^\ast - f_0\|_{\infty} \le 2\|\tilde f_n - f_0\|_{\infty}$. This completes the proof.
\end{proof}

We cite \cite[Remark 1]{schmidthieber2017} in order to maintain a consistent reading flow and for the sake of completeness.

\begin{prop} \label{covering_bound}
    For the network $\sF(L,L_1,p,s,\infty)$ we have the covering entropy bound
    \begin{eqnarray*}
       \log \sN(\delta,\sF(L,L_1,p,s,\infty),\norm{\cdot}_{\infty}) \le (s+1)\log(2^{2L+5}\delta^{-1}(L+1)p_0^2p_{L+1}^2s^{2L}).
    \end{eqnarray*}
\end{prop}

\subsubsection{Approximation error and Lipschitz continuity of neural networks}

The first part of the following theorem is taken from \cite[Theorem 5]{schmidthieber2017}. The second part \reff{lipschitz_constant_approx_network} is proved below.

\begin{thm}\label{theorem_approximation_network}
    For any function $f\in C_\dt^{\beta}([0,1]^\dt,K)$ and any integers $m \ge 1$, $N \ge (\beta+1)^\dt \vee (K+1)e^{\dt}$, there exists a network
    \[
        \tilde f \in \sF(L,(\dt,6(\dt+\lceil \beta \rceil)N,...,6(\dt+\lceil \beta \rceil)N,1),s,\infty)
    \]
    with depth
    \[
        L = 8+(m+5)(1+\lceil \log_2(\dt \vee \beta)\rceil)
    \]
    and number of active parameters
    \[
        s \le 141 (\dt+\beta+1)^{3+r}N(m+6)
    \]
    such that
    \[
        \|\tilde f - f\|_{\infty} \le (2K+1)(1+\dt^2+\beta^2)6^\dt N 2^{-m} + K 3^{\beta}N^{-\frac{\beta}{\dt}}.
    \]
    Furthermore, $\tilde f$ satisfies for any $x,y\in[0,1]^\dt$ that
    \begin{equation}
        |\tilde f(x) - \tilde f(y)| \le \mathrm{Lip}(N,m)\cdot |x-y|_{\infty} \label{lipschitz_constant_approx_network}
    \end{equation}
    where
    \[
        \mathrm{Lip}(N,m) := 2\beta F(K+1)e^\dt( 24 \dt^6 2^\dt N 2^{-m} + 3\dt).
    \]
\end{thm}

To prove \reff{lipschitz_constant_approx_network}, we first recap how $\tilde f$ is constructed in \cite[Theorem 5]{schmidthieber2017}.

As in \cite{schmidthieber2017}, we define for $x,y\in [0,1]$, $m\in\IN$,
    \[
        \text{mult}_m(x,y) := \Big(\sum_{k=1}^{m+1}\big\{R^k(\frac{x-y+1}{2}) - R^k(\frac{x+y}{2})\big\} + \frac{x+y}{2} - \frac{1}{4}\Big)_{+}
    \]
    where
    \[
        R^k := T^k \circ T^{k-1} \circ ... \circ T^1,\quad k\in\IN,
    \]
    and
    \[
        T^k(x) := \min\{\frac{x}{2}, 2^{1-2k} - \frac{x}{2}\}, \quad k\in\IN.
    \]
    
    \begin{lem}\label{lemma_mult_derivative}
        For $x,y \in [0,1]$ where $\text{mult}_m$ is differentiable, it holds that
        \[
            \partial_1\text{mult}_m(x,y) = y + \text{res}_1(x,y), \quad\quad \partial_2\text{mult}_m(x,y) = x + \text{res}_2(x,y)
        \]
        where $|\text{res}_i(x,y)| \le 2^{-m-1}$, $i = 1,2$. Furthermore,
        \[
            |\text{mult}_m(x,y) - x\cdot y| \le 2^{-m-1}(x+y) \le 2^{-m}.
        \]
    \end{lem}
    \begin{proof}[Proof of Lemma \ref{lemma_mult_derivative}]
    A straightforward calculation yields
    \[
        \partial_1 R^k(x) = \begin{cases}\frac{1}{2^k}, & x \in A_{k+},\\
        -\frac{1}{2^k}, & x \in [0,1] \backslash A_{k+}
        \end{cases} = \frac{1}{2^k}(2\cdot \Ii_{A_{k+}}(x)-1)
    \]
    where
    \[
        A_{k+} := \bigcup_{j=0}^{2^k-1}[\frac{j}{2^k}, \frac{j+1}{2^k}].
    \]
    We conclude that
    \begin{eqnarray}
        \partial_1 \text{mult}_m(x,y) &=& \sum_{k=1}^{m+1}\big\{ \partial_1 R^k(\frac{x-y+1}{2})\cdot \frac{1}{2} - \partial_1 R^k(\frac{x+y}{2})\cdot \frac{1}{2}\big\} + \frac{1}{2}\nonumber\\
        &=& \sum_{k=1}^{m+1}\frac{1}{2^k}\big\{\Ii_{A_{k+}}(\frac{x-y+1}{2}) - \Ii_{A_{k+}}(\frac{x+y}{2})\big\} + \frac{1}{2}.\label{mult_derivative}
    \end{eqnarray}
    Suppose that the following binary representations hold for $x,y\in[0,1]$:
    \[
        \frac{x+y}{2} = \sum_{k=1}^{\infty}\frac{a_k}{2^k}, \quad\quad \frac{x-y+1}{2} = \sum_{k=1}^{\infty}\frac{b_k}{2^k}
    \]
    where $a_k, b_k \in \{0,1\}$ ($k\in\IN$). Then,
    \[
        \Ii_{A_{k+}}(\frac{x-y+1}{2}) = 1-b_k, \quad\quad \Ii_{A_{k+}}(\frac{x+y}{2}) = 1-a_k.
    \]
    Insertion into \reff{mult_derivative} yields
    \begin{eqnarray*}
        \partial_1 \text{mult}_m(x,y) &=& \sum_{k=1}^{m+1}\frac{1}{2^k}\big\{a_k - b_k\big\} + \frac{1}{2} = \frac{x+y}{2} - \frac{x-y+1}{2} + \frac{1}{2} +  \text{res}(x,y) = y + \text{res}(x,y)
    \end{eqnarray*}
    where
    \[
        \text{res}_k(x,y) := \sum_{k=m+2}^{\infty}\frac{b_k}{2^k} - \sum_{k=m+2}^{\infty}\frac{a_k}{2^k}.
    \]
    Due to $a_k,b_k \in \{0,1\}$ ($k\in\IN$), we see that $|\text{res}(x,y)| \le 2^{-(m+1)}$. The proof for $\partial_2 \text{mult}_m$ is similar.
    
    The second statement follows by the first one using the fundamental theorem of analysis:
    \begin{eqnarray*}
        \big|\text{mult}_m(x,y) - xy\big| &\le& x\int_0^{1}|\partial_1 \text{mult}_m(xt,yt) - yt| dt + y\int_0^{1}|\partial_2 \text{mult}_m(xt,yt)-xt| dt\\
        &\le& 2^{-m-1}(x+y).
    \end{eqnarray*}
\end{proof}

As in \cite{schmidthieber2017}, define recursively for $x \in [0,1]$,
\[
    \IM_m(x) := x,
\]
for $x = (x_1,...,x_{2^q}) \in [0,1]^{2^q}$ ($q\in\IN$),
\[
    \IM_m(x) := \text{mult}_m(\IM_m(x_1,...,x_{2^{q-1}}),\IM_m(x_{2^{q-1}+1},...,x_{2^q})),
\]
and for $x = (x_1,...,x_\dt)\in [0,1]^\dt$, $q = \lceil \log(r)\rceil$, 
\[
    \IM_m(x) := \IM_m(x_1,...,x_\dt,\underbrace{1,...,1}_{\text{$(2^q-\dt)$ ones}}).
\]

The first part of the following lemma is taken from \cite{schmidthieber2017}, Lemma A.3.
\begin{lem}\label{lemma_produkt_approx}
    For $y\in [0,1]^\dt$, it holds that
    \begin{equation}
        |\IM_m(y_1,...,y_\dt) - \prod_{k=1}^{r}y_k| \le \dt^2\cdot 2^{-m}\label{lemma_produkt_approx_eq1}
    \end{equation}
    and for $j \in \{1,...,\dt\}$, at the points $y$ where $\IM_m$ is differentiable, 
    \begin{equation}
        |\partial_j \IM_m(y_1,...,y_\dt) - \prod_{k=1, k\not= j}^{\dt}y_k| \le 2 \dt^{3}\cdot 2^{-m}\label{lemma_produkt_approx_eq2}
    \end{equation}
\end{lem}
\begin{proof}[Proof of Lemma \ref{lemma_produkt_approx}]
We only have to show \reff{lemma_produkt_approx_eq2}. We restrict ourselves to $j = 1$ for simplicity. With some abuse of notation, overload $y := (y,1,...,1)$ (where we added $2^{q}-\dt$ ones). Then by Lemma \ref{lemma_mult_derivative},  \reff{lemma_produkt_approx_eq1} and $|y_k| \le 1$ ($k=1,...,2^q$),
\begin{eqnarray*}
    &&|\partial_{y_1}\IM_m(y) - \prod_{k=2}^{\dt}y_k|\\
    &\le& \sum_{i=1}^{q}\Big\{\Big(\prod_{k=2^{q-i+1}+1}^{2^{q}}y_k\Big)\\
    &&\quad\quad\quad\quad\quad\quad\times \Big|\partial_1 \IM_m(\IM_m(y_1,...,y_{2^{q-i}}),\IM_m(y_{2^{q-i}+1},...,y_{2^{q-i+1}})) - \prod_{k=2^{q-i}+1}^{2^{q-i+1}}y_k\Big|\\
    &&\quad\quad\quad\quad\quad\quad \times \prod_{j=i+1}^{q}\partial_1 \IM_m(\IM_m(y_1,...,y_{2^{q-j}}),\IM_m(y_{2^{q-j}+1},...,y_{2^{q-j+1}}))\Big\}\\
    &\le& \sum_{i=1}^{q}\Big\{\Big(\prod_{k=2^{q-i+1}+1}^{2^{q}}y_k\Big)\cdot \big(\big|\IM_m(y_{2^{q-i}+1},...,y_{2^{q-i+1}}) - \prod_{k=2^{q-i}+1}^{2^{q-i+1}}y_k\big| + 2^{-m-1}\big)\\
    &&\quad\quad\quad\quad\quad\quad\quad\quad\quad\quad \times \prod_{j=i+1}^{q}\big(\IM_m(y_{2^{q-j}+1},...,y_{2^{q-j+1}}) + 2^{-m-1}\big)\Big\}\\
    &\le& \sum_{i=1}^{q}\Big\{\Big(\prod_{k=2^{q-i+1}+1}^{2^{q}}y_k\Big)\cdot 2\cdot 4^{q-i}2^{-m} \cdot \prod_{j=i+1}^{q}\big(\prod_{k=2^{q-j}+1}^{2^{q-j+1}}y_k + 4^{q-j}2^{-m}\big)\Big\}\\
    &\le& 2^{-m+1}\cdot \sum_{i=1}^{q}\Big\{4^{q-i}\cdot \prod_{j=i+1}^{q}(1 + 4^{q-j}2^{-m})\Big\}\\
    &\le& 2^{-m}\sum_{i=1}^{q}8^{q-i} \le \frac{8}{7} \dt^{3} 2^{-m}.
\end{eqnarray*}
\end{proof}

Now we show \reff{lipschitz_constant_approx_network}. To do so, we derive the mathematical expression $Q_3$ used in \cite[Theorem 5]{schmidthieber2017} to describe $\tilde f$.

Let $M$ be the largest integer such that $(M+1)^\dt \le N$. Define the grid
\[
    D(M) := \{x_l := (\ell_j/M)_{j=1,...,\dt}: (\ell_1,...,\ell_\dt) \in \{0,1,...,M\}^\dt\}.
\]
For $x,y\in [0,1]$, put
\[
    I_{y}(x) := (\frac{1}{M} - |x-y|)_{+},
\]
and for $x,y\in[0,1]^\dt$,
\[
    \text{Hat}_{x}(y) := \IM_m(I_{x_1}(y_1),...,I_{x_r}(y_r)).
\]
For $a,x \in [0,1]^r$, let
\[
    P_{a}^{\beta}f(y) = \sum_{0 \le |\alpha| < \beta}(\partial^{\alpha}f)(a)\cdot \frac{(y-a)^{\alpha}}{\alpha!} = \sum_{0 \le |\gamma| < \beta}c_{\gamma}(a) \cdot y^{\gamma}
\]
denote the multivariate Taylor polynomial of $f$ with degree $\beta$ at $a$. In the above formula, $\alpha$ and $\gamma$ denote multi-indices.

Then (cf. \cite{schmidthieber2017}, (31)-(35) therein), $|c_{\gamma}| \le \frac{K}{\gamma!}$ and $\sum_{\gamma \ge 0}|c_{\gamma}| \le Ke^\dt \le \frac{1}{2} B$, where $B := \lceil 2Ke^{\dt}\rceil$. Put
\[
    Q_1(y)_{x_l} := \frac{1}{B}\sum_{0 \le |\gamma| < \beta}c_{\gamma}(x_l)\cdot \IM_m(y_{\gamma}) + \frac{1}{2},
\]
where $y_{\gamma} := (y_{\gamma_1},y_{\gamma_2},...,y_{\gamma_\dt})$. Define
\[
    Q_2(y) := \sum_{x_l \in D(M)}\text{mult}_m\big(Q_1(y)_{x_l}, \text{Hat}_{x_l}(y)\big),
\]
and
\[
    Q_3 := S \circ Q_2
\]
where $S(x) := B M^\dt (x - \frac{1}{2M^\dt})$. Since $\tilde f = Q_3$, \reff{lipschitz_constant_approx_network} follows from Lemma \ref{lemma_lipschitz_q3}.

\begin{lem}\label{lemma_lipschitz_q3}
    For $x,y\in [0,1]^\dt$, it holds that
    \[
        |Q_3(x) - Q_3(y)| \le \beta F B( 24 \dt^6 2^\dt N 2^{-m} + 3\dt)\cdot |x-y|_{\infty}.
    \]
\end{lem}
\begin{proof}[Proof of Lemma \ref{lemma_lipschitz_q3}]
Since $Q_2$ is piecewise linear, it is enough to consider its first derivative at the points where it is differentiable to derive its Lipschitz constant.

With $q = \lceil \log(\dt)\rceil$, it holds that
\begin{eqnarray*}
    \partial_{y_1}\text{Hat}_{x_l}(y) &=& \partial_1 \IM_m(\IM_m(y_1,...,y_{2^{q-1}}),\IM_m(y_{2^{q-1}+1},...,y_\dt,1,...,1))\\
    &&\quad\quad\times \prod_{i=2}^{q-1}\partial_1 \IM_m(\IM_m(y_1,...,y_{2^{q-i}}),\IM_m(y_{2^{q-i}+1},...,y_{2^{q-i+1}}))\\
    &&\quad\quad \times \partial_1 \text{mult}_m(y_1,y_2).
\end{eqnarray*}
By Lemma \ref{lemma_produkt_approx},
\begin{eqnarray}
    |Q_1(y)_{x_l} - \big(\frac{1}{B}P_{x_l}^{\beta}f(y) + \frac{1}{2}\big)| &\le& \frac{1}{B}\sum_{0 \le |\gamma| < \beta}|c_{\gamma}(x_l)|\cdot |\IM_m(y_{\gamma}) - y^{\gamma}|\nonumber\\
    &\le& \frac{r^2 2^{-m}}{B}\cdot \sum_{0 \le |\gamma| < \beta}|c_{\gamma}| \le \frac{1}{2}\dt^2 2^{-m}.\label{lemma_lipschitz_q3_eq1}
\end{eqnarray}
Furthermore,
\begin{eqnarray*}
    \partial_{y_1}Q_1(y)_{x_l} &=& \frac{1}{B}\sum_{0\le |\gamma|< \beta, \gamma_1 \ge 1}c_{\gamma}(x_l) \cdot \partial_{y_1}\IM_m(y_{\gamma}) = \frac{1}{B}\sum_{0\le |\gamma|< \beta}c_{\gamma}(x_l) \cdot \sum_{j=1}^{\gamma_1}\partial_j \IM_m(y_{\gamma}).
\end{eqnarray*}
Thus by Lemma \ref{lemma_produkt_approx}, 
\begin{eqnarray}
    \big|\partial_{y_1}Q_1(y)_{x_l} - \frac{1}{B}\partial_{y_1}P_{x_l}^{\beta}f(y)\big| &\le& \frac{1}{B}\sum_{0\le |\gamma|< \beta, \gamma_1 \ge 1}|c_{\gamma}(x_l)| \cdot \sum_{j=1}^{\gamma_1}\big|\partial_j \IM_m(y_{\gamma}) - y^{\gamma-(1,0,...,0)}\big|\nonumber\\
    &\le& \frac{1}{B}\sum_{0 \le |\gamma|< \beta}|c_{\gamma}(x_l)|\cdot 2\gamma_1 \dt^3 2^{-m} \le \beta \dt^3 2^{-m}.\label{lemma_lipschitz_q3_eq2}
\end{eqnarray}
Finally, Lemma  \ref{lemma_produkt_approx} yields
\begin{equation}
    \big|\text{Hat}_{x_l}(y) - \prod_{k=1}^{\dt}I_{(x_l)_k}(y_k)\big| \le \dt^2 2^{-m}\label{lemma_lipschitz_q3_eq3}
\end{equation}
and for $j \in \{1,...,\dt\}$, since $\partial_{y_j}I_{(x_l)_j}(y_j) \in \{-1,+1\}$, 
\begin{eqnarray}
    &&\big|\partial_{y_j}\text{Hat}_{x_l}(y) - \partial_{y_j}\prod_{k=1}^{\dt}I_{(x_l)_k}(y_k)\big|\nonumber\\
    &=& \Big|\partial_j \IM_m(I_{(x_l)_1}(y_1),...,I_{(x_l)_\dt}(y_\dt)) - \prod_{k=1, k\not=j}^{\dt}I_{(x_l)_k}(y_k) \Big|\cdot |\partial_{y_j}I_{(x_l)_j}(y_j)|\nonumber\\
    &\le& \dt^2 2^{-m}.\label{lemma_lipschitz_q3_eq4}
\end{eqnarray}
Note furthermore that
\[
    |\partial_{y_j}\prod_{k=1}^{\dt}I_{(x_l)_k}(y_k)| \le \prod_{k=1, k\not=j}^{\dt}I_{(x_l)_k}(y_k)\cdot |\partial_{y_j}I_{(x_l)_j}(y_j)| \le 1
\]
and
\[
    |\frac{1}{B}\partial_{y_1}P_{x_l}^{\beta}f(y)| \le \frac{1}{B}\sum_{0\le|\gamma| < \beta}|c_{\gamma}(x_l)|\cdot \gamma_1 y^{\gamma-(1,0,...,0)} \le \frac{\beta}{2}.
\]
By Lemma \ref{lemma_mult_derivative}, \reff{lemma_lipschitz_q3_eq2} and \reff{lemma_lipschitz_q3_eq4}, it holds that
\begin{eqnarray}
    &&\Big|\partial_{y_1} Q_2(y)\nonumber\\
    &&\quad\quad\quad\quad - \sum_{x_l \in D(M), |x_l - y|_{\infty} \le M^{-1}}\Big\{\text{Hat}_{x_l}(y)\cdot \partial_{y_1}Q_1(y)_{x_l} + Q_1(y)_{x_l}\cdot \partial_{y_1}\text{Hat}_{x_l}(y)\Big\}\Big|\nonumber\\
    &\le& \sum_{x_l \in D(M), |x_l - y|_{\infty} \le M^{-1}}\Big\{\big|\partial_1 \text{mult}_m\big(Q_1(y)_{x_l}, \text{Hat}_{x_l}(y)\big) - \text{Hat}_{x_l}(y)\cdot \partial_{y_1}Q_1(y)_{x_l}\big|\nonumber\\
    &&\quad\quad\quad\quad\quad\quad\quad\quad\quad\quad\quad\quad\times |\partial_{y_1} Q_1(y)_{x_l}|\nonumber\\
    &&\quad\quad\quad\quad\quad\quad\quad\quad\quad\quad + \big|\partial_2 \text{mult}_m(Q_1(y)_{x_l}, \text{Hat}_{x_l}(y)) - Q_1(y)_{x_l}\big|\cdot \big|\partial_{y_1}\text{Hat}_{x_l}(y)\big|\Big\}\nonumber\\
    &\le& 2^{-m-1}\sum_{x_l \in D(M), |x_l - y|_{\infty} \le M^{-1}}\big\{|\partial_{y_1} Q_1(y)_{x_l}| + |\partial_{y_1}\text{Hat}_{x_l}(y)\big|\big\}\nonumber\\
    &\le& 2^{-m-1}\cdot 2^\dt\cdot \big\{(\beta \dt^3 2^{-m} + \frac{\beta}{2}) + (\dt^2 2^{-m}+1)\big\}\nonumber\\
    &\le& 4\beta \dt^3 2^{\dt}\cdot 2^{-m}.\nonumber\\
    &&\label{lemma_lipschitz_q3_eq5}
\end{eqnarray}
In a similar manner, we obtain with \reff{lemma_lipschitz_q3_eq1}, \reff{lemma_lipschitz_q3_eq2}, \reff{lemma_lipschitz_q3_eq3} and \reff{lemma_lipschitz_q3_eq4} that
\begin{eqnarray}
    &&\Big|\sum_{x_l \in D(M), |x_l - y|_{\infty} \le M^{-1}}\text{Hat}_{x_l}(y)\cdot \partial_{y_1}Q_1(y)_{x_l}\nonumber\\
    &&\quad\quad\quad\quad\quad\quad - \sum_{x_l \in D(M), |x_l - y|_{\infty} \le M^{-1}}\Big(\prod_{k=1}^{\dt}I_{(x_l)_k}(y_k)\Big)\cdot \frac{1}{B}\partial_{y_1}P_{x_l}^{\beta}f(y)\Big|\nonumber\\
    &\le& 2^\dt\cdot \big(\dt^2 2^{-m}\cdot (\beta \dt^3 2^{-m} + \frac{\beta}{2}) + 1\cdot (\beta \dt^3 2^{-m})\big)\nonumber\\
    &\le& 4\beta \dt^5 2^\dt\cdot 2^{-m},\label{lemma_lipschitz_q3_eq6}
\end{eqnarray}
and
\begin{eqnarray}
    &&\Big|\sum_{x_l \in D(M), |x_l - y|_{\infty} \le M^{-1}}Q_1(y)_{x_l}\cdot \partial_{y_1}\text{Hat}_{x_l}(y)\nonumber\\
    &&\quad\quad\quad\quad - \sum_{x_l \in D(M), |x_l - y|_{\infty} \le M^{-1}}\Big\{ \big(\frac{1}{B}P_{x_l}^{\beta}f(y) + \frac{1}{2}\big)\cdot \Big(\prod_{k=2}^{\dt}I_{(x_l)_k}(y_k)\Big)\cdot \partial_{y_1}I_{(x_l)_1}(y_1)\Big\}\Big|\nonumber\\
    &\le& 2^\dt\cdot \big(\frac{1}{2}\dt^2 2^{-m}\cdot (\dt^2 2^{-m} + 1) + 1\cdot (\dt^2 2^{-m})\big)\nonumber\\
    &\le& 4\dt^4 2^\dt\cdot 2^{-m}.\nonumber\\
    &&\label{lemma_lipschitz_q3_eq7}
\end{eqnarray}
Now, we have
\begin{eqnarray}
    &&\Big|\sum_{x_l \in D(M), |x_l - y|_{\infty} \le M^{-1}}\Big(\prod_{k=1}^{\dt}I_{(x_l)_k}(y_k)\Big)\cdot \frac{1}{B}\partial_{y_1}P_{x_l}^{\beta}f(y)\Big|\nonumber\\
    &\le& \frac{\beta}{2}\cdot \sum_{x_l \in D(M), |x_l - y|_{\infty} \le M^{-1}}\Big(\prod_{k=1}^{\dt}I_{(x_l)_k}(y_k)\Big) \le \frac{\beta}{2}\cdot M^{-\dt}.\label{lemma_lipschitz_q3_eq8}
\end{eqnarray}
Let $u \in D(M)$ be the grid point which satisfies $u_j \le y_j \le u_j + M^{-1}$, $j = 1,...,\dt$.

Let $b(|\alpha|) := b$, if $|\alpha| = \beta-1$ and $b(\alpha) = 1$, otherwise. For general $a,a' \in [0,1]^\dt$ with $|y-a|_{\infty}, |y-a'|_{\infty} \le M^{-1}$, $|a-a'|_{\infty} \le M^{-1}$, it holds that
\begin{eqnarray*}
    && |P^{\beta}_a f(y) - P^{\beta}_{a'} f(y)| \le \sum_{0
    \le |\alpha|< \beta}\frac{1}{\alpha!}\cdot\big\{ \big|\partial^{\alpha}f(a) - \partial^{\alpha}f(a')\big|\cdot |(y-a)^{\alpha}|\\
    &&\quad\quad\quad\quad + |\partial^{\alpha}f(a')|\cdot \big|(y-a)^{\alpha} - (y-a')^{\alpha}\big|\big\}\\
    &\le& \sum_{0 \le |\alpha| < \beta}\frac{1}{\alpha!}\cdot \big\{ K|a-a'|_{\infty}^{b(|\alpha|)} M^{-|\alpha|}\\
    &&\quad\quad + (K+F)\sum_{j=1}^{\dt}\Big(\prod_{k=1}^{j-1}|y_k - a_k|^{\alpha_k}\Big)\cdot \Big(\prod_{k=j+1}^{\dt}|y_k - a_k'|^{\alpha_k}\Big)\cdot |(y_j - a_j)^{\alpha_j} - (y_j - a_j')^{\alpha_j}|\big\}\\
    &\le& \sum_{0\le|\alpha| < \beta}\frac{1}{\alpha!}\cdot \big\{K|a-a'|_{\infty}^{b(|\alpha|)}M^{-|\alpha|} + \sum_{j=1}^{\dt}\alpha_j M^{-(\alpha_j-1)}|a_j - a_j'|\big\}\\
    &\le& \sum_{0\le|\alpha| < \beta}\frac{1}{\alpha!}\cdot \big\{K|a-a'|_{\infty}^{b(|\alpha|)}M^{-|\alpha|} + (K+F)\beta M^{-1}\big\}\\
    &\le& (K+\beta K + \beta F)e^t M^{-1}.
\end{eqnarray*}
The last step is due to the fact that $f$ is assumed to have at least H\"older exponent 1.

Using this result, we obtain 
\begin{eqnarray}
    &&\Big|\sum_{x_l \in D(M), |x_l - y|_{\infty} \le M^{-1}}\Big\{ \big(\frac{1}{B}P_{x_l}^{\beta}f(y) + \frac{1}{2}\big)\cdot \Big(\prod_{k=2}^{\dt}I_{(x_l)_k}(y_k)\Big)\cdot \partial_{y_1}I_{(x_l)_1}(y_1)\Big\}\Big|\nonumber\\
    &\le& \sum_{(i_2,...,i_\dt) \in \{0,1\}^r}\Big(\prod_{k=2}^{\dt}I_{u_k + \frac{i_k}{M}}(y_k)\Big)\cdot \Big| (\frac{1}{B}P_{(u_1 + M^{-1},u_2 + \frac{i_2}{M},...,u_\dt + \frac{i_\dt}{M})}^{\beta}f(y) + \frac{1}{2})\cdot \partial_{y_1}I_{u_1 + M^{-1}}(y_1)\nonumber\\
    &&\quad\quad\quad\quad\quad\quad\quad\quad\quad\quad\quad\quad\quad\quad\quad\quad + (\frac{1}{B}P_{(u_1,u_2 + \frac{i_2}{M},...,u_\dt + \frac{i_\dt}{M})}^{\beta}f(y) + \frac{1}{2})\cdot \partial_{y_1}I_{u_1}(y_1)\Big|\nonumber\\
    &\le& \frac{1}{B}\sum_{(i_2,...,i_\dt) \in \{0,1\}^\dt}\Big(\prod_{k=2}^{\dt}I_{u_k + \frac{i_k}{M}}(y_k)\Big)\cdot \big|P_{(u_1 + M^{-1},u_2 + \frac{i_2}{M},...,u_\dt + \frac{i_\dt}{M})}^{\beta}f(y) - P_{(u_1,u_2 + \frac{i_2}{M},...,u_\dt + \frac{i_\dt}{M})}^{\beta}f(y)\big|\nonumber\\
    &\le& \frac{(K+\beta K + \beta F)e^{\dt}}{B}M^{-(\dt-1)}\cdot M^{-1} = \frac{(K+\beta K + \beta F)e^{\dt}}{B}\cdot  M^{-\dt}.\nonumber\\
    &&\label{lemma_lipschitz_q3_eq9}
\end{eqnarray}
Using the bounds \reff{lemma_lipschitz_q3_eq5}, \reff{lemma_lipschitz_q3_eq6}, \reff{lemma_lipschitz_q3_eq7}, \reff{lemma_lipschitz_q3_eq8} and \reff{lemma_lipschitz_q3_eq9}, we obtain with $K \ge 1$ that
\[
    |\partial_{y_1}Q_2(y)| \le 24 \beta F \dt^5 2^{\dt} 2^{-m} + 3\beta M^{-\dt}.
\]
The proof for the other derivatives $\partial_{y_j}$, $j = 2,...,\dt$, is completely similar. Thus, for $x,y\in [0,1]^\dt$,
\[
    |Q_2(y) - Q_2(x)| \le \int_0^{1}|\langle \partial Q_2(x + t(y-x)), y-x\rangle| dt \le \dt \sup_y|\partial Q_2(y)|_{\infty}\cdot |y-x|_{\infty}.
\]
We obtain
\[
    |Q_3(x) - Q_3(y)| \le BM^\dt |Q_2(x) - Q_2(y)| \le \beta F B( 24 \dt^6 2^\dt M^\dt 2^{-m} + 3\dt)\cdot |x-y|_{\infty}.
\]
\end{proof}


\end{document}